\documentclass{amsart}
\usepackage{amssymb,mathrsfs}

\usepackage{color}
\usepackage{enumerate}
\usepackage[textheight=25 cm, textwidth=16cm, headheight=10pt, headsep=20pt, footskip=10pt, left=1.5cm, right=1.5cm]{geometry}

\usepackage{hyperref}
\hypersetup{pdfborder=0 0 0}

\def\bC {\mathbf{C}}

\def\bN {\mathbf{N}}

\def\bR {\mathbf{R}}

\def\bT {\mathbf{T}}

\def\bZ {\mathbf{Z}}

\def\fH {\mathfrak{H}}

\def\cA {\mathcal{A}}

\def\cC {\mathcal{C}}
\def\cD {\mathcal{D}}

\def\cH {\mathcal{H}}

\def\cL {\mathcal{L}}

\def\cN {\mathcal{N}}
\def\cP {\mathcal{P}}

\def\cR {\mathcal{R}}
\def\cS {\mathcal{S}}

\def\cW {\mathcal{W}}

\def\a {{\alpha}}
\def\b {{\beta}}
\def\g {{\gamma}}

\def\de {{\delta}}
\def\eps {{\epsilon}}

\def\l {{\lambda}}

\def\si {{\sigma}}

\def\d {{\partial}}
\def\grad {{\nabla}}
\def\Dlt {{\Delta}}

\def\indc {{\bf 1}}

\def\la {\langle}
\def\ra {\rangle}
\def \La {\bigg\langle}
\def \Ra {\bigg\rangle}

\newcommand{\Dom}{\operatorname{Dom}}

\newcommand{\Tr}{\operatorname{trace}}

\newcommand{\Lip}{\operatorname{Lip}}

\newcommand{\MKd}{\operatorname{dist_{MK,2}}}
\newcommand{\Op}{\operatorname{Op}}

\newcommand{\ba}{\begin{aligned}}
\newcommand{\ea}{\end{aligned}}

\newcommand{\be}{\begin{equation}}
\newcommand{\ee}{\end{equation}}

\newcommand{\lb}{\label}

\newcommand{\bvc}{\tau}

\newcommand{\bcr}{ }

\newtheorem{Thm}{Theorem}[section]
\newtheorem{Rmk}[Thm]{Remark}
\newtheorem{Prop}[Thm]{Proposition}
\newtheorem{Cor}[Thm]{Corollary}
\newtheorem{Lem}[Thm]{Lemma}
\newtheorem{Def}[Thm]{Definition}

\newcommand{\ket}[1]{\langle #1|}
\newcommand{\bra}[1]{| #1\rangle}

\newcommand{\hb}{\hbar}

\begin{document}

\title[low regularity semiclassical  evolution]{\LARGE
semiclassical evolution with low regularity 
}
\LARGE
\author[F. Golse]{\Large Fran\c cois Golse}
\address[F.G.]{Ecole polytechnique, CMLS, 91128 Palaiseau Cedex, France}
\email{francois.golse@polytechnique.edu}

\author[T. Paul]{\Large Thierry Paul}
\address[T.P.]{CNRS \& LJLL Sorbonne Université 4 place Jussieu 75005 Paris, France}
\email{thierry.paul@upmc.fr}
\LARGE
\begin{abstract}
We prove semiclassical estimates for the Schr\"odinger-von Neumann  evolution with $C^{1,1}$ potentials and density matrices whose square root have either Wigner functions with low regularity independent of the dimension, or matrix elements between Hermite functions having long range decay. The estimates are settled in  different weak topologies and apply to  initial density operators  whose square root have Wigner functions  $7$ times differentiable, independently of the dimension. They also apply to the $N$ body quantum dynamics uniformly in $N$. In a appendix, we finally  estimate the dependence in the dimension of the constant appearing on the Calderon-Vaillancourt Theorem.
\end{abstract}
%
%
%
%
%
\maketitle
\tableofcontents

\section{Introduction}\label{intro}
The semiclassical approximation links the quantum dynamics of Hamiltonian, say,  $H=-\frac{\hbar^2}2\Delta+V(x)$ on $L^2(\bR^d,dx)$, to the underlying classical one, namely the  flow generated by the Hamiltonian $h(p,q)=\frac{1}2p^2+V(q)$ on $\bR^{2d}\simeq T^{*}\bR^d$. This quite indirect link, particularly efficient   when the Planck constant takes  small values,  relies on the presence of fast oscillations in the initial data whose speed diverges linearly  in $\hbar^{-1}$ as $\hbar\to 0$.

To our knowledge, all results in semiclassical approximations are subjects to the following alternative:
\begin{itemize}
\item 
either no specific ansatz is made for the initial data of the quantum dynamics. In that case, under some tightness conditions of the initial data, revealing the size of oscillations, and along subsequences of values of $\hbar$ tending to $0$, the Wigner function (or equivalently the Husimi one) of the solution of the quantum evolution is shown to tend to a solution 
of the Liouville classical equation, with no estimate of the rate  of convergence provided (see, e.g., \cite{gerard,gerard3,lionspaul}, and also \cite{A,flp} for an extension to potentials whose gradient has only $BV$ regularity)). 

\noindent Let us recall that, by quantum and classical evolution, we mean the content of the following table:

$$
   quantum\ \ \ \  
\left\{
\begin{array}{llr}
i\hbar\partial_t\psi=H\psi,& \psi\in\ L^2({\bR^d},)& \  \ \ \mbox{  \ \ \  Schr\"odinger}\\
\frac d{dt}R=\frac1{i\hbar}[H,R],& R>0,\ \Tr{R}=1,& \mbox{ \ \ \ \ \  \ \ \ von Neumann}
\end{array}
\right.
$$

$$
  classical\ \ \ \  
\left\{
\begin{array}{llr}
\begin{pmatrix}
\dot p\\
\dot q
\end{pmatrix}=
\begin{pmatrix}
-\partial_qh(p,q)\\
\ \ \partial_ph(p,q)
\end{pmatrix},\ & (p,q)\in T^{*}\bR^d& \ \ \mbox{ \ \ \  Hamilton}\\
\frac d{dt}\rho=\{h,\rho\},& \rho>0,\ \int_{\bR^{2d}}\rho=1&\mbox{ Liouville}
\end{array}
\right.
$$
\item
or very precise estimates of the rate of convergence are obtained after some ansatz is made on the initial quantum data (WKB, coherent states for the Schr\"odinger equation; various type of quantization (pseudodifferential, Weyl, T\"oplitz calculus) for the von Neumann equation). In this case a ``quasimode" is constructed, e.g. a solution of an approximate equation, and the unitarity of the quantum flow (supposed a priori) provides a remainder estimate of order $\hbar^{1/2}$ or $\hbar$  in $L^2(\bR^d)$ or $\cL(L^2(\bR^d))$ topology. Let us remark that, although the quasimode is constructed trough the  solution of the equations of the classical paradigm, the estimate of the rate of convergence is settled  in the topology of the quantum one.
\end{itemize}
\vskip 1cm
{
Let us mention also that the accuracy of the results for the von Neumann equation in  the second case of the preceding alternative  degrades quite rapidly when the space dimension or the number of particles involved get large. This for two reasons. First, estimating the remainder is done through the so-called Calderon-Vaillancourt Theorem, a  result  very regularity consuming when the total dimension, namely $d=3N$ in the case of $N$ particle in three dimensions, increases. Basically $C^{[\frac{3N}2]}$-differentiability  for both the potential and the initial data is required. In addition, the $O(\hbar)$ size of the remainder is degraded by the multiplicative constant $\gamma_d$ appearing in Calderon-Vaillancourt Theorem. We give an estimate of $\gamma_d$ in Appendix \ref{cvconst} (Theorem  \ref{cvgamma}), as we couldn't find it elsewhere. Certainly not sharp, the result we obtain is $\log{\gamma_d}\sim 
\frac{11}4 d \log{d}$,  and a value of $\gamma_d$  larger than the inverse of the Planck constant $h={2\pi}\hbar\sim6.626\times 10^{-34}$ already for $N=4$
 in three dimensions}.
\vskip 1cm
Recently we found a way of somehow estimating directly the  ``distance"  between quantum and classical objects. Given $R(t)$ and $\rho(t)$, solutions of the von Neumann and Liouville equations respectively, we defined in \cite{FGPaul} a positive  number $E_\hbar(\rho(t),R(t))$ satisfying for all time $t$ a ``Gronwall type" estimate
\be\label{armat}
E_\hbar(\rho(t),R(t))^2
\leq
e^{\max{(1,4Lip(\nabla V))}t}E_\hbar(\rho(0),R(0))^2,
\ee
where $Lip(\nabla V)$ denotes the Lipschitz constant of $\nabla V$.
\noindent The definition of $E_\hbar$ is given in Definition \ref{arma} below, and is a generalization of a quantum analog of the Wasserstein distance of exponent $2$ introduced in \cite{GMouPaul}.

Let us remark immediately that \eqref{armat} is uniform in $\hbar$ and doesn't contain any extra semiclassical error term as appearing in estimates involving quasimodes. e.g. $O(\hbar^\infty)$. Consistently $E_\hbar(R(0),\rho(0))$ cannot vanish. In fact
\be\label{armaborninf}
E_\hbar\geq \tfrac12d\hbar.
\ee
On the other side, the smallness of $E_\hbar(R,\rho)$ has a true meaning thanks to the fact that it almost dominates a distance between $\rho$ and the Husimi function of $R$: the following bound holds true for any  $R,\rho\geq 0$ with $\Tr{R}=\int_{\bR^{2d}}\rho dpdq=1$, proven in  \cite{FGPaul} Theorem 2.4 2 and recalled in Theorem \ref{armaprop} of the present paper,
\be\label{armahus}
\MKd(\widetilde W[R],\rho)^2\leq E_\hbar(R,\rho)^2+\tfrac12d\hbar.
\ee
The inequality \eqref{armahus} was proven in  \cite{FGPaul} Theorem 2.4 2 and is recalled in Theorem \ref{armaprop} of the present paper,  the definition of the Wasserstein distance $\MKd$  is recalled in \eqref{defmkd} and $\widetilde W[R]$, the Husimi function of $R$,  in \eqref{defhhus}. 

 \noindent Moreover, $E_\hbar$ shares with the notion of distance the following type of triangle inequality, valid for any  $R\geq0,\Tr{R}=1$ and probability measures $f,f'$, see Lemma \ref{armatriang} in Section \ref{semicw},
\be\label{armatriangform}
E_\hbar(f',R_\hbar)\leq \MKd(f',f)+E_\hbar(f,R_\hbar).
\ee

Since the correspondence between $R$ and $\widetilde W[R]$ is one-to-one (see Section \ref{mainr}), formula \eqref{armat} together with \eqref{armaborninf} will give an estimate for the semiclassical
 evolution  of $R(t)$ solution of the von Neumann equation,  estimate expressed this time in the classical paradigm at the contrary of the quasimode ones. More precisely, we will get this estimates if we are able to  find a probability measure $\rho^{in}$ such that
 \be\label{armainitial}
 E_\hbar(\rho^{in},R(0))=o(1)\mbox{ as }\hbar\to 0.
 \ee
 This task was achieved in the case where $R(0)$ is a T\"oplitz operator whose  symbol $\rho$ is any probability measure on $\bR^{2d}$ (see Definition in Section \ref{mainr} below): for all 
 $t\in\bR$, 
 Theorem 2.7 in \cite{FGPaul}  (see also Section \ref{semicw} of the present article) leads to
 \be\label{armatop}
 \MKd(\rho(t),\widetilde W[R(t)])^2
 \leq
 \frac12(1+e^{\max{(1,4Lip(\nabla V)^2)}t})d\hbar,
 \ee
 \noindent Note that this result, being valid for any probability measure type symbol $\rho$ and therefore not requiring any regularity condition, is, at our knowledge, unreachable by usual T\"oplitz calculus proof.
 
 The goal of the present paper is:
 \begin{enumerate}
 \item to get rid of the T\"oplitz ansatz we just mentioned
 \item to find practical efficient 
conditions for  quantum initial data  to be semiclassically evolved by  the von Neumann equation, 
  without any need of concrete ansatz
 \item to obtain propagation estimates of the form \eqref{armatop} for initial 
density matrices, the Wigner function of  the square root of them  
 having low regularity $C^7$, independently of the dimension and the number of particles.
 {
 \item to estimate the semiclassical propagation in some ``direct" weak distances instead of in the Wasserstein-Husimi formulation of  \cite{FGPaul}, namely a  distance $\delta$  between Wigner functions generated by the dual of a space of test functions, or likewise a  distance $d$ between density operators associated to the trace duality of a set of test operators (see definition above in this section).
 \item to get semiclassical estimates  uniform in the number of particles involved.
} 
 \end{enumerate}
\vskip 1cm
Our results will be developed on three nested levels of generality.
The most general one involves the existence of a probability density, linked to the initial condition $R^{in}$ of the von Neumann equation, satisfying explicit conditions. The second one gives a concrete realization of such a density under
 explicit  conditions on the size  of the matrix elements of $\sqrt{R(0)}$ between (semiclassically scaled) Hermite functions. The third one shows that such  conditions are satisfied by  operators whose Wigner function satisfies low regularity conditions. 
\vskip 0.5cm

Each of these results will be expressed in different forms, using the distance $\MKd$ or in terms of the following ones, introduced in \cite{splitting} Appendix B and studied in Appendix \ref{linfty} on the present paper. 

\noindent We recall the definition for two Hilbert-Schmidt (in particular density)  operators $R$ and $S$ and $M>0$.
\begin{eqnarray}
\delta_M(W_{\hbar}[R],W_{\hbar}[S])
&=&\sup_{\substack{
 \max\limits_{|\a|,|\b|\le M}\|\d^\a_x\d^\b_\xi f\|_{L^\infty}\leq 1}}\left|\int (W_{\hbar}[R](x,\xi)-W_{\hbar}[S](x,\xi))f(x,\xi)dxd\xi\right|,\ \nonumber\\
d_M(R,S)&=&\sup_{\substack{
\max\limits_{\substack{|\a|,|\b| \le M}}\|\mathcal D^\a_{-i\hbar\nabla}\mathcal D^\b_{x}F\|_{1}\leq 1}}|\Tr{(F(R-S))}|\hskip 2cm \cD_{\mbox{\LARGE $\bullet$}}:=\tfrac1{i\hbar}[\bullet,\cdot].\nonumber
\end{eqnarray}

The different ``distances" used  in this article are nested in the following chain of inequalities (see Propositions \ref{propnest} and \ref{armapropcor}) valid for $R,S$ density operators
{\large $$
d_{2[d/4]+2}(R,S)\le 2^d\delta_{2[d/4]+2}(W_{\hbar}[R],W_{\hbar}[S])\le\MKd(\widetilde W_{\hbar}[R],\widetilde W_{\hbar}[S])+\frac{2d\gamma_d}{\sqrt\pi}\sqrt{\hbar}\leq E_\hbar(\widetilde W_\hbar[R],S)+(\tfrac{2d\gamma_d}{\sqrt\pi}+\tfrac d2)\sqrt{\hbar}
$$
}
\begin{eqnarray}
d_{2[d/4]+2}(R,S)&\le&2^d\delta_{2[d/4]+2}(W_{\hbar}[R],W_{\hbar}[S])\nonumber\\
\le\MKd(\widetilde W_{\hbar}[R],\widetilde W_{\hbar}[S])+\frac{2d\gamma_d}{\sqrt\pi}\sqrt{\hbar}&\leq& E_\hbar(\widetilde W_\hbar[R],S)+(\tfrac{2d\gamma_d}{\sqrt\pi}+\tfrac d2)\sqrt{\hbar}\nonumber
\end{eqnarray}
where $\g_d$ is the constant that appears in the Calderon-Vaillancourt Theorem \ref{cvgamma} below.


We will use finally a last distance $\delta$  whose definition, independent of the dimension and involving $L^2$ test functions is the following (once again $R$ and $S$ are Hilbert-Schmidt operators). 
$$
 \delta
 (
 W_{\hbar}[R],W_{\hbar}[S)]
 )=
 \sup_{\substack{
Lip(f)\leq 1\\
\| f\|_{L^2(\bR^{2d})},\ \|\nabla f\|_{L^2(\bR^{2d})}\leq 1
}}
|\int (W_\hbar[R]-W_\hbar[S])f(x,\xi)dxd\xi|.
$$
The 
  series of nested links between $\delta, \MKd$ and $E_\hbar$,
this time independent of the dimension, is the following (see Section \ref{ltwo} in Appendix \ref{wassweak} below), for $R,S$ density operators.
{\Large
$$
\delta(\widetilde W_{\hbar}[R],\widetilde W_{\hbar}[S])\le
\MKd(\widetilde W_{\hbar}[R],\widetilde W_{\hbar}[S])+\sqrt\hbar
\|W_\hbar[R-S]\|_{L^2}
\leq\sqrt2E_\hbar(\widetilde W_\hbar[R],S)+\sqrt\hbar
\|W_\hbar[R-S]\|_{L^2}.
$$
}
In our last main result, Theorem \ref{mainN}, the distance $\delta$ will provide a topology for the $N$-body semiclassical propagation of factorized (e.g. bosonic) initial data which will be uniform in $N$ as $N\to\infty$.

\vskip 0.5cm
 The three type of results mentioned earlier are expressed in Section \ref{mainr} by Theorems \ref{main1}, \ref{main2} and \ref{main3} respectively, in  reverse order of nesting for pedagogical purposes. Their proofs are given in Section \ref{proofmains} by showing that they are mostly corollaries of the three items of Theorem \ref{main} proven  in Sections \ref{proofmain3}, \ref{proofmain2} and \ref{proofmain1} respectively. Section \ref{semicw} is devoted to the quantum analogue of the Wasserstein distance of exponent two and its properties.
 Semiclassical uniform in $N$ estimates for the $N$-body quantum propagation are stated and proved in Section \ref{npart} and the three appendices \ref{profftriang}, \ref{wassweak} and \ref{cvconst} are devoted respectively to the proof of the triangle type inequality for $E_\hbar$ just mentioned, the comparison of $E_\hbar$ with weak topologies and an estimation of the Calderon-Vaillancourt constant.
 \vskip 0.3cm
 
 \vskip 1cm
 
\section{Main results}\label{mainr}

For any real function  $V$ of class $C^{1,1}$ on $\bR^d$ such that the operator $-\frac12\hbar^2\Delta+V$ is essentially self-adjoint on $L^2(\bR^d)$, we consider the von Neumann equation
\be\label{vne}
i\hbar\frac{d}{dt} R_\hbar(t)=
[-\frac12\hbar^2\Delta+V,R_\hbar(t)],\ R_\hbar(0)=R^{in}
\ee
where the initial condition $R^{in}$ is a \textit{density matrix}, that is  that $R^{in}$ is positive and $\Tr{R^{in}}=~1$.

Obviously
\be\label{adeux}
R^{in}=A^2,\ \mbox{ $A$ being a positive  Hilbert-Schmidt operator}.
\ee
Let us recall the definition of the Wigner and Husimi transforms of a density operator on $\fH$ (see  e.g. \cite{lionspaul} or \cite{GMouPaul} Appendix B for a short review). If $R$ is the density matrix of  integral kernel $r$, its Wigner transform at scale $\hbar$ is the function on $\bR^d\times\bR^d$ defined by the formula
\be\label{defwwig}
W_\hbar[R](x,\xi):=\frac1{(2\pi)^d}\int_{\bR^d}e^{-i\xi\cdot y}r(x+\tfrac12\hbar y,x-\tfrac12\hbar y)dy\,.
\ee
The Husimi transform of $R$ is
\be\label{defhhus}
\widetilde W_\hbar[R]:=e^{\hbar\Dlt_{x,\xi}/4}W_\hbar[R]
\ee
and we recall that
$$
\widetilde W_\hbar[R]\ge 0\ , 
\int_{\bR^d\times\bR^d}\widetilde W_\hbar[R](x,\xi)dxd\xi=\int_{\bR^d\times\bR^d}W_\hbar[R](x,\xi)dxd\xi=1\,.
$$
In particular, $\widetilde W_\hbar[R]$ is a probability density on $\bR^d\times\bR^d$ for each $R\in\cD(L^2(\bR^d))$.

Moreover, by Remark 2.3. in \cite{FGPaul}, $\widetilde{W_\hbar}{[R]}$ determines $R$ uniquely.
 \vskip 1cm 
 We denote by $\rho(t)$ the solution of the following Liouville equation on $\bR^{2d}$ with initial condition $\widetilde W[R^{in}]$:
 \be\label{liuov}
 \dot\rho=\{\tfrac12\xi^2+V(x),\rho\},\ \ \ \rho(0)=\widetilde W[R^{in}]
 \ee
 We denote also by $\Phi^t$ the Hamiltonian flow of Hamiltonian $\tfrac12p^2+V(q)$ so that
 \be\label{phit}
 \rho(t)=\rho(0)\circ\Phi^{-t}.
 \ee
 Moreover we define
 $$
 \lambda=\frac{1+\max{(4Lip(\nabla V)^2,1)}}2.
 $$
 Finally we define $\mathcal R(t)$ as the operator whose Wigner function is $\rho(t)$:
\be\label{deffnathcalr}
 W_\hbar[\mathcal R(t)]:=\rho(t).
 \ee
%
\vskip 1cm
 In the sequel we will denote 
 \begin{itemize}
 \item by $z=(x,\xi)$ a point in $T^*\bR^d$,
 \item by $\cP(\bR^{2d})$ the set of probability densities on $T^*\bR^d$
 \item by $\cP_2(\bR^{2d})$ the set of probability densities on $T^*\bR^d$ with finite second moments
 \item by $\cD(L^2(\bR^d))$ the set of density matrices on $L^2(\bR^d)$
 \item by $\cD_2(L^2(\bR^d))$ the set of density matrices $R$ on $L^2(\bR^d)$ satisfying
 $$\Tr(-\hbar^2\Delta+x^2)R<\infty.
 $$
 
 \end{itemize}
 \vskip 1cm
\vskip 1cm
For $k=(k_1,\dots,k_d)\in\bZ^d$ we define:
the Fourier coefficients $a_k(z)$ of a function $a(z_1,\dots,z_d)$, $z_i\in T^*\bR$ 
by
\begin{eqnarray}\label{baraka}
 a_{k}(z)&=&
 \int_{\bT^d}a(z_1e^{ik_1\theta_1},\dots,z_de^{ik_d\theta_d})e^{ik\cdot\theta}d\theta:=
\int_{\bT^d}a(e^{i\theta}z)e^{ik\theta}d\theta.
\end{eqnarray}
 
\vskip 1cm

 \textbf{Definition.}\  
 For $K,M\in\bR^+,\ N\in \bN$, we define 
 on $L^1(\bR^{2d})$ 
 the norm
 \be\label{defmkn}
 \|a\|_{M,K,N}:=\sup_{k\in\bZ^d}\sup_{z\in\bR^{2d}}
 \prod_{l=1}^d(|z_l|^2+1)^{M}(|k_l|+d)^{K}
 \sup_{|\alpha_1|,\dots,|\alpha_d|
 \leq N}
 |
\prod_{m=1}^d(\sqrt\hbar D_{z_m})^{\alpha_m} 
 a_k(z)|,
 \ee
{\bcr
\vskip 3cm
}
\begin{Thm}\label{main1}
Let $R(t)$ be the solution of the von Neumann equation \eqref{vne} with initial condition $R^{in}$ and $\mathcal R(t)$ the operator whose Wigner function solves the Liouville equation \eqref{liuov} with initial data the Husimi function of $R^{in}$. Let $\delta$ the distance defined in Definition \ref{defddel}.

{\bcr Let us suppose that
 there exist two functions $\mu(\hbar),\ \nu(\hbar)$ satisfying
 $$
\sqrt\hbar\mu(\hbar)=o(1), \ 
\mu(\hbar)\nu(\hbar)=o(1)\mbox{ as }\hbar\to 0,
 $$
  such that, for some $\epsilon\in]0,1[$,
\begin{enumerate}
\item \label{tightwig}
$\|W_\hbar[\sqrt{R^{in}}]
\|_{
{\color{black} \frac34+\epsilon,\frac72+3\epsilon,3}
}\leq 
(2\pi\hbar)^{-\frac d2}
{\mu(\hbar)}
$

\item\label{semicwig}
$
\hbar^{\frac12}\| W_\hbar[\sqrt{R^{in}}]
\|
_{
\frac34
+\epsilon,
\frac52
+3\epsilon,
4
}
\leq 
(2\pi\hbar)^{-\frac d2}
{\nu(\hbar)}
$
\end{enumerate}
%

Then, 
for all $t\in\bR$,
{
$$
\MKd\big(\widetilde W[R_\hbar(t)],
\rho(t)\big)
\leq {D_{\ref{m1}}}{}e^{\lambda|t|}\max{(\sqrt\hbar,\sqrt\hbar\mu(\hbar),\sqrt{\mu(\hbar)\nu(\hbar)})}.
$$ 
 Moreover, if $ W[R_\hbar^{in}]
\in L^1
(\bR^{2d})$ 
for each $\hbar\in ]0,1[$, we have
\begin{eqnarray}
&&2^{-d}d_{2[d/4]+3}\big(R(t),\mathcal R(t)\big)\leq
\delta_{2[d/4]+3}\big(W_{\hbar}[R(t))],W_{\hbar}[R^{in}]\circ\Phi^{-t}\big)\nonumber\\
&&\leq 
{D_{\ref{m1}}}e^{\lambda|t|}\max{(\sqrt\hbar,\sqrt\hbar\mu(\hbar),\sqrt{\mu(\hbar)\nu(\hbar)})}
+
\sqrt\hbar(\tfrac{2d\gamma_d}{\sqrt\pi}+e^{1+\Lip{(\nabla V)})|t|)}\|W_\hbar[R^{in}]\|_{L^1(\bR^{2d})})
\nonumber
\end{eqnarray}
}
}
Here $D_{\ref{m1}}$ is given in \eqref{dm1}.
\end{Thm}
\vskip 1cm

Using \eqref{mkncn} and taking $\mu(\hbar)=1,\ \nu(\hbar)=\hbar^{\frac12}$ so that $\sqrt{\mu(\hbar)\nu(\hbar)}=\hbar^{\frac14}$, we derive easily the following (maybe more tractable) corollary.

\textbf{Corollary} \textit{
Let $R_\hbar$ be a family of density matrices of $L(\bR^d)$. 
Let us assume that the Wigner function  $W_\hbar[\sqrt{R^{in}}](x,\xi)$ of $\sqrt{R^{in}}$ has derivatives up to order $
\tfrac{13}2+3\epsilon
$ 
bounded by
$C(2\pi\hbar)^{-\frac d2}((\xi^2+x^2)+d)^{
-\frac{10}4-4\epsilon
}, C>0$, i.e.
$$
\sup_{\substack{|\beta_1|,\dots,|\beta_d|\leq 7}}
 |\prod_{m=1}^dD_{(x,\xi)}^{\beta_m}   W_\hbar[\sqrt{R^{in}}](x,\xi)| 
 \leq \frac {C(2\pi\hbar)^{-\frac d2}}{((\xi^2+x^2)^2+d)^{\frac{10}4+3\epsilon}}\hfil\ \ \  \forall (x,\xi)\in\bR^{2d}.
$$
Then there exist $D,D'$ such that,
for all $t\in\bR$,
$$
\MKd\big(\widetilde W[R_\hbar(t)],
\rho(t)
\big)
\leq 
{D}\hbar^{\frac14} e^{\lambda|t|}.
$$
Moreover, if, in addition, 
$ \|W_\hbar[R^{in}]
\|_{L^1
(\bR^{2d})}\leq C'<\infty$,
 for each $\hbar\in [0,1]$, 
we have
$$
\delta_{2[d/4]+3}\big(W_{\hbar}[R(t))],W_{\hbar}[R^{in}]\circ\Phi^{-t}\big)
\leq D'\hbar^{\frac14} e^{2\lambda|t|}.
$$
}
\vskip 1cm
Let us make a few remarks concerning Theorem \ref{main1}.

 \vskip 0.3cm
 The factor 
 $(2\pi\hbar)^{-\frac d2}$ 
 in Theorem \ref{main1} and its corollary might seem strange, but it is in fact natural, if we think that the Wigner function of 
 $(\sqrt{R^{in}})^2$ has to be the one of $R^{in}$, a density operator.
 
 \noindent Indeed if we think that the Wigner function of an operator is the quotient of its Weyl symbol by $(2\pi\hbar)^d$, we have, using the Moyal product $\star_{Moyal}$, that
 $$
 (2\pi\hbar)^dW_\hbar[R^{in}]=(2\pi\hbar)^dW_\hbar[\sqrt{R^{in}}]\star_{Moyal}(2\pi\hbar)^dW_\hbar[\sqrt{R^{in}}
 $$ so that
 $$
 W_\hbar[R^{in}]=((2\pi\hbar)^{\frac d2}W_\hbar[\sqrt{R^{in}}])\star_{Moyal}((2\pi\hbar)^{\frac d2}W_\hbar[\sqrt{R^{in}}).
 $$ 
 
In order to directly implement semiclassical approximation in weak sense, other than the results presented in the introduction obtained by compactness methods without rate of convergence, one should work with Wigner or Husimi functions. But then one faces the difficulty  of the non positiveness of the Wigner function. Of course, one can cure this default  by  using instead the Husimi function, but the Husimi function follows an evolution equation involving analyticity regularity (see \cite{AthaPaul}).

 Therefore it seems to us that the only way of obtaining precise weak semiclassical results by ``standard" methods consists in constructing quasimodes in strong topology, a way obviously very regularity consuming.
On the contrary, the results in the present paper uses weak topology from the beginning, and allows us to obtain results requiring   little regularity independent of the dimension. 

Let us mention finally that we studied in \cite{FGPaulCohSta}, section $5$, the pertinence of weak versus strong topologies concerning the transition between quantum and classical paradigms.
\vskip 1cm
Actually the content of Theorem \ref{main1} is a particular case of the following more general result.

Let us define the Hermite orthonormal basis of $L^2(\bR^d)$ as $\{H_j,j\in\bN^d\}$, i.e., for $j=(j_1,\dots,j_d)$, 
\be\label{defher}
H_j=h_{j_1}\otimes\dots\otimes h_{j_d}\mbox{ where }(-\hbar^2\frac{d^2}{dx^2}+x^2)h_k=(2k+1)h_k,\ ||h_k||_{L^2(\bR)}=1.
\ee

We will denote by $(\cdot,\cdot)$ the scalar product in $L^2(\bR^d)$.
\vskip 1cm

\begin{Thm}\label{main2}
{ Let us suppose that
 there exist two functions $\mu'(\hbar),\ \nu'(\hbar)$ satisfying
 $$
\sqrt\hbar\mu'(\hbar)=o(1), \ 
\mu'(\hbar)\nu'(\hbar)=o(1)\mbox{ as }\hbar\to 0,
 $$
  such that, for some $\epsilon\in]0,1[$,
\begin{enumerate}
\item \label{iaj}
$|(H_i,
\sqrt{R^{in}}
 H_j)|\leq (2\pi\hbar)^{\frac d2} \mu'(\hbar)
 {\prod\limits_{1\leq l\leq d}|\hbar j_l+\frac12|^{-
\frac34
-\epsilon}(|i_l-j_l|+1)^{-2-\epsilon}}$,

\item\label{ioaj}
$\sup\limits_{O\in\Omega_1}{|(H_i,[O,
\sqrt{R^{in}}
] H_j)|}\leq{(2\pi\hbar)^{\frac d2}\nu'(\hbar)}
{\prod\limits_{1\leq l\leq d}|\hbar j_l+\frac12|^{-
\frac12
-\epsilon}(|i_l-j_l|+1)^{-1-\epsilon}}$,

\ where $\Omega_1=\{y_j,
\pm\hbar
\partial_{y_j}\ on\ L^2(\bR^d,dy),\ j=1,\dots,d\}$.
\end{enumerate}
}

Then, 
for all $t\in\bR$,
{
$$
\MKd\big(\widetilde W[R_\hbar(t)],
\rho(t)\big)
\leq {D_{\ref{m2}}}{}e^{\lambda|t|}\max{(\sqrt\hbar,\sqrt\hbar\mu'(\hbar),\sqrt{\mu'(\hbar)\nu'(\hbar)})}.
$$ 
 Moreover, if $ W[R_\hbar^{in}]
\in L^1
(\bR^{2d})$ 
for each $\hbar\in ]0,1[$, we have

\begin{eqnarray}
&&2^{-d}d_{2[d/4]+3}\big(R(t),\mathcal R(t)\big)\leq
\delta_{2[d/4]+3}\big(W_{\hbar}[R(t))],W_{\hbar}[R^{in}]\circ\Phi^{-t}\big)\nonumber\\
&&\leq
{D_{\ref{m2}}}e^{\lambda|t|}\max{(\sqrt\hbar,\sqrt\hbar\mu'(\hbar),\sqrt{\mu'(\hbar)\nu'(\hbar)})}.
+
\sqrt\hbar(\tfrac{2d\gamma_d}{\sqrt\pi}+e^{1+\Lip{(\nabla V)})|t|)}\|W_\hbar[R^{in}]\|_{L^1(\bR^{2d})})
\nonumber
\end{eqnarray}
Here $D_{\ref{m2}}$ is given in \eqref{dm2}.
}

%
%

\end{Thm}
\noindent\textbf{Remark}

\noindent\textit{Here also the factor $(2\pi\hbar)^d$ is natural. For example, thanks to it, estimate (\ref{tight}) gives that $\Tr R^{in}$ is finite.
}
\vskip 1cm
Once again, Theorem \ref{main2} is also a particular case of a more general one. Let us first set up the following definition (see \cite{GMouPaul}, Appendix B for further details).
\vskip 1cm

\noindent\textbf{Definition}\ \textit{[T\"oplitz operators]
For each $z=x+i\xi\in\bC^{d}$, we denote 
\be\label{defetc}
|z,\hbar\ra:\,y\mapsto(\pi\hbar)^{-d/4}e^{-|y-x|^2/2\hbar}e^{i\xi\cdot(y-x)/\hbar}\,\ \ \ \| |z,\hbar\ra\|_{L^2(\bR^d)}=1,
\ee
and we designate by $|z,\hbar\ra\la z,\hbar|$ the orthogonal projection on the line $\bC|z,\hbar\ra$ in $\fH$. 
For each Borel probability density $\mu$ on $\bR^d\times\bR^d$, we define the T\"oplitz operator of symbol $\mu$ as the operator defined weakly on $L^2(\bR^d)$ by the formula
$$
\Op^T_\hbar(\mu):=\frac1{(2\pi\hbar)^d}\int_{\bR^d\times\bR^d}|x+i\xi,\hbar\ra\la x+i\xi,\hbar|\mu(x,\xi)dxd\xi.
$$
}
One easily shows that
\be\label{decomp}
\Tr{\Op^T_\hbar({\color{black}(2\pi\hbar)^d}\mu)}=1
\ee
so that $\Op^T_\hbar({\color{black}(2\pi\hbar)^d}\mu)$ is 
trace class
and admits an $L^2$ spectral decomposition
\be\label{psecdecomp}
{\Op^T_\hbar({\color{black}(2\pi\hbar)^d}\mu)}=\sum_{i\in\bN}\mu_i|\psi_i\ra\la\psi_i|,\ \ \ \psi_i\in L^2(\bR^d).
\ee
\vskip 1cm
When $F>0$, 
$\Op^T_\hbar({\color{black}(2\pi\hbar)^d}F)$ is 
injective, as any element $\psi$ of its kernel would satisfy $\langle x,\xi|\psi\rangle=0$ for almost all $(x,\xi)\in\bR^{2d}$ which would imply $\psi=0$ by completeness of discrete families of coherent states (see \cite{bgz}). So $\Op^T_\hbar({\color{black}(2\pi\hbar)^d}\mu]^{-1/2}$ exists as a unbounded operator on $L^2(\bR^d)$ as defined by the spectral theorem
\be\label{defmoinsundemi}
\Op^T_\hbar({\color{black}(2\pi\hbar)^d}F)^{-1/2}=\sum\limits_{i\in\bN}\mu_i^{-1/2}|\psi_i\ra\la\psi_i|
\ee
on the domain $\mathcal D=\{\psi\in L^2(\bR^d),\ \sum\limits_{i\in\bN}\mu_i^{-1}|\la\psi_i|\psi\ra|^2<\infty\}.$
\begin{Thm}\label{main3}
Let $R_\hbar^{in}=R^{in}$ be a family of density  matrices on $L^2(\bR^d)$ satisfying the following hypothesis
\begin{enumerate}
\item\label{tight}
(tightness)

\noindent there exists a probability density $F>0$ on $\bR^{2d}$ 
 satisfying $\int(|x|+|\xi|)F(dx,d\xi)<~\infty$ 
such that, for each $\hbar\in(0,1]$,
$$
\Op^T_\hbar({\color{black}(2\pi\hbar)^d}F)^{-1/2}\sqrt{R^{in}},\ \sqrt{R^{in}}\Op^T_\hbar({\color{black}(2\pi\hbar)^d}F)^{-1/2}\mbox{ have   bounded extensions to  $L^2(\bR^d)$}
$$
(we will denote by the same symbols the operators and their bounded extensions).
\item\label{semic}
(semiclassical hypothesis)\ 

\noindent there exists $
\tau(\hbar)=o(1)\mbox{ as }\hbar\to 0$ such that
$$
\sup_{O\in\Omega}\big\|
\Op^T_\hbar({\color{black}(2\pi\hbar)^d}F)^{-1/2}\sqrt{R^{in}}\big[O,\sqrt{R^{in}}\Op^T_\hbar({\color{black}(2\pi\hbar)^d}F)^{-1/2}\big]
\big\|
=\tau(\hbar),
$$
\ where $\Omega=\{-\hbar^2\Delta_y+|y|^2,y_j,
\pm\hbar
\partial_{y_j}
{\bcr \ on\ L^2(\bR^d,dy)},\ j=1,\dots,d\}$.
\end{enumerate}

 Let $R_\hbar(t)$ the solution of the 
von Neumann equation \eqref{vne} and $\rho(t)$ the solution of the  Liouville equation \eqref{liuov} on $\bR^{2d}$ with initial condition $\widetilde W[R^{in}]$. 
 
\vskip 0.5cm 
Then, for all $t\in\bR$,
{
$$
E_\hbar(\rho(t),R(t))\leq 
\frac{D_{\ref{m3}}}{\sqrt2}e^{\lambda|t|}\max{(\sqrt\hbar,\sqrt{\tau(\hbar)})}
$$ 
and therefore
$$
\MKd(\widetilde W_\hbar[R(t)],\rho(t))
\leq
D_{\ref{m3}}
e^{\lambda|t|}
\max{(\sqrt\hbar,\sqrt{\tau(\hbar)})},
$$
}
with   $D_{\ref{m3}}$ given by \eqref{dm3}.
\small
\end{Thm}
\vskip 0.5cm
\begin{Rmk}\

\noindent

\noindent 
Hypothesis \eqref{tight} is not empty as $R_\hbar:=\Op^T_\hbar({\color{black}(2\pi\hbar)^d}F)$ satisfies obviously \eqref{tight}. 

\noindent When the tightness bound in \eqref{tight} is uniform in $\hbar$, \eqref{semic} can be replaced by

\noindent $
\sup\limits_{O\in\Omega}\big\|
\big[O,R_\hbar^{1/2}\Op^T_\hbar({\color{black}(2\pi\hbar)^d}F)^{-1/2}\big]
\big\|
=o(1)$.
\end{Rmk}
\vskip 0.5cm
\begin{Rmk}\label{rm2}\ 

\noindent The condition \eqref{tight} 
implies that
\noindent $\Op^T_\hbar({\color{black}(2\pi\hbar)^d}F)^{-1/2}R_\hbar\Op^T_\hbar({\color{black}(2\pi\hbar)^d}F)^{-1/2}$ is 
bounded.
Therefore,  writing the spectral decomposition $\Op^T_\hbar({\color{black}(2\pi\hbar)^d}F)=\sum\limits_{j\in\bN}\lambda_j|j\rangle\langle j|$,
$$
\langle j|R_\hbar|j\rangle=O(\lambda_j)\mbox{ so that }
\langle j|R_\hbar|j\rangle\to 0\ as\ j\to\infty,
$$
since $\Op^T_\hbar({\color{black}(2\pi\hbar)^d}F)$ is trace class so that $\lambda_j\to 0$ as $j\to\infty$.

Moreover, since $F$ and therefore $\Op^T_\hbar({\color{black}(2\pi\hbar)^d}F)$
have 
finite  moments, the eigenstates $|j\rangle$ must be phase-space localized at infinity as $j\to \infty$ 
 and condition \eqref{tight} 
reflects 
 the lack of concentration of $R_\hbar$ at infinity.
\end{Rmk}
\vskip 0.5cm
\vskip 1cm

\section{Semiclassical Wasserstein}\label{semicw}

Let us start this section by recalling the definition of the second order Wasserstein distance $\MKd$ (see \cite{VillaniAMS,VillaniTOT}).
 
The Wasserstein distance of order two between two probability measures $\mu,\nu$ on $\bR^m$ with finite second moments is defined as
\be\label{defmkd}
\MKd(\mu,\nu)^2
=
\inf_{\gamma\in\Gamma(\mu,\nu)}\int_{\bR^m\times\bR^m}
|x-y|^2\gamma(dx,dy)
\ee
where $\Gamma(\mu,\nu)$ is the set of probability measures on $\bR^m\times\bR^m$ whose marginals on the two factors are $\mu$ and $\nu$,  i.e., for ny test function $a$, 
$$
\int_{\bR^m\times\bR^m}a(x)\gamma(dx,dy)=\int_{\bR^m}a(x)\mu(dx),\ \int_{\bR^m\times\bR^m}a(y)\gamma(dx,dy)=\int_{\bR^m}a(y)\nu(dy).
$$

Let us define now the notion of coupling between classical and quantum densities.
\begin{Def}
Let $p\equiv p(x,\xi)$ be a probability density on $\bR^d\times\bR^d$, and let $R\in\cD(\fH)$. A coupling of $p$ and $R$ is a measurable function $Q:\,(x,\xi)\mapsto Q(x,\xi)$ defined a.e. on $\bR^d\times\bR^d$ and with values in $\cL(\fH)$ 
s.t. $Q(x,\xi)\in\cD(\fH)$  for a.e. $(x,\xi)\in\bR^d\times\bR^d$, and
$$
\left\{
\ba
{}&\Tr(Q(x,\xi))=p(x,\xi)\hbox{ for a.e. }(x,\xi)\in\bR^d\times\bR^d\,,
\\
&\iint_{\bR^d\times\bR^d}Q(x,\xi)dxd\xi=R\,.
\ea
\right.
$$
The set of all such functions is denoted by $\cC(p,R)$.
\end{Def}

Mimicking the definition of Monge-Kantorovich distances, we next define the pseudo-distance between $R_\hbar$ and $f$ in terms of an appropriate ``cost function'' analogous to the quadratic cost function used in optimal transport.

\begin{Def}\label{arma}\label{ehbar}
For each probability density $p\equiv p(x,\xi)$ on $\bR^d\times\bR^d$ and each $\rho\in\cD(\fH)$, we set
$$
E_\hbar(p,R):=\left(\inf_{Q\in\cC(p,R)}\int_{\bR^d\times\bR^d}\Tr(c_\hbar(x,\xi)Q(x,\xi))dxd\xi\right)^{1/2}\in[0,+\infty]\,,
$$
where the transport
 cost $c_\hbar$ is the function of $(x,\xi)$ with values in the set of unbounded operators on $\fH=L^2(\bR^d_y)$ defined by the formula
$$
c_\hbar(x,\xi):=\tfrac12(|x-y|^2+|\xi+i\hbar\grad_y|^2)\,.
$$
\end{Def}
Immediate properties of $E_\hbar$ are stated in the following result.
\begin{Thm}\lb{armaprop}[from Theorem 2.4 in \cite{FGPaul}]\ 

\noindent (1) For each probability density $p$ on $\bR^d\times\bR^d$ such that
$
\int_{\bR^d\times\bR^d}(|x|^2+\xi|^2)p(x,\xi)dxd\xi<~\infty
$
and each $R\in\cD(\fH)$, one has
$$
E_\hbar(p,R)^2\ge\tfrac12d\hbar\,.
$$
(2) Let $R_\hbar=\Op^T_\hbar((2\pi\hbar)^d\mu)$, where $\mu$ is a Borel probability measure on $\bR^d\times\bR^d$. Then
$$
E_\hbar(p,R_\hbar)^2\le\MKd(p,\mu)^2+\tfrac12d\hbar\,.
$$
(3) For each $R\in\cD(\fH)$, one has
$$
\MKd(p,\widetilde W_\hbar[R])^2\le E_\hbar(p,R)^2+\tfrac12d\hbar
\leq2E_\hbar(p,R)^2\,.
$$
\end{Thm}
\begin{Cor}\label{armapropcor}
Let $R$ be a T\"oplitz operator of symbol ${\color{black}(2\pi\hbar)^d}\mu$. Then
$$
E_\hbar(\mu,R)^2=\frac12 d\hbar,
$$ and an optimal coupling is 
$$
Q_0(x,\xi):=\mu(x,\xi)|x,\xi\rangle\langle x,\xi|. 
$$
\end{Cor}
\begin{proof}
the first equality  follows from the items $(1) \mbox{ and }(2)$ in Theorem \ref{armaprop} and the second from the easy computation
$$
\int_{\bR^d\times\bR^d}\Tr(c_\hbar(x,\xi)Q_0(x,\xi))dxd\xi=\frac12 d\hbar.
$$
\end{proof}
\vskip 1cm
We will also need the following ``triangle inequality" proven 
in Appendix \ref{profftriang}.

\begin{Thm}\label{armatriang}  
Let $f,g$ be two probability densities on 
$\bR^{2d}$ and  $R$ be a density operator on 
$L^2(\bR^d)$ 
satisfying
$$
\int_{\bR^{2d}}(p^2+q^2)(f(p,q)+g(p,q)) dpdq<\infty\mbox{ and }
\Tr(-\hbar^2\Delta+x^2)R<\infty.
$$
The following inequality holds true:

\noindent
 $$
 E_\hbar(f,R_1)\le\MKd(f,g)+E_\hbar(g,R_1).
 $$
 \end{Thm}
\vskip 1cm
The proofs of Theorems \ref{main1}, \ref{main2} and \ref{main3} will rely extensively on Theorem 2.7 in \cite{splitting}, a slight improvement of Theorem 2.7 in \cite{FGPaul} in the special case $N=1$, that we recall now.
\begin{Thm}[Theorem 2.7 in \cite{splitting}]\label{armaarma}
Let $\rho^{in}$ be a probability density on $\bR^{2d}$ satisfying
\be\label{momdeux}
\int_{\bR^{2d}}(p^2+q^2)\rho^{in}(p,q) dpdq<\infty,
\ee
and let $R(t)$ and $\rho(t)$ the solutions of  the von Neumann equation \eqref{vne} and the Liouville equation \eqref{liuov} with initial conditions $R^{in}$ and $\rho^{in}$ respectively. then, for all $t\in\bR$,
$$
E_\hbar(\rho(t),R(t))^2\leq E_\hbar(\rho^{in},R^{in})^2e^{(1+\max{(1,4Lip(\nabla V)^2))}|t|}.
$$
\end{Thm}
\vskip 0.5cm

Using the right inequality in item $(3)$ of Theorem \ref{armaprop} we get the following corollary.
\begin{Cor}\label{armaarmacor}
Let $R_\hbar(t)$ and $\rho(t)$ solve \eqref{vne} and \eqref{liuov} respectively with initial data $R^{in}$ and $\widetilde W[R_\hbar]$ respectively. Let $R^{in}$ satisfies
$$
\int_{\bR^{2d}}(p^2+q^2)\widetilde W[R_\hbar](p,q) dpdq<\infty,
$$ Then
$$
\MKd(\widetilde W[R_\hbar(t)],
\rho(t))
\leq 
\sqrt{2}E_\hbar(\widetilde W[R_\hbar],R^{in})
e^{\lambda |t|}.
$$
\end{Cor}
%
\vskip 1cm

\section{Proofs of Theorems \ref{main1}, \ref{main2} and \ref{main3}}\label{proofmains}
The conclusions of the Theorems \ref{main1}, \ref{main2} and \ref{main3} are  consequences of  the following theorem which is the heart of this article and  whose proof will occupy Section \ref{proofmain3}, \ref{proofmain2} and \ref{proofmain1}.
\begin{Thm}\label{main}
Let $R^{in}$ satisfy the hypothesis 
of
Theorems  \ref{main1},   \ref{main2} or \ref{main3}. 
Then
\begin{enumerate}[(I)]
\item \label{m1}
$E_\hbar(\widetilde W[R^{in}],R^{in})\leq 
\frac{D_{\ref{m1}}}{\sqrt2}\max{(\sqrt\hbar,\sqrt\hbar\mu(\hbar),\sqrt{\hbar\mu(\hbar)\nu(\hbar)})}$
\hfill [case of  Theorem \ref{main1}]

\item \label{m2}
 
$E_\hbar(\widetilde W[R^{in}],R^{in})\leq \frac{D_{II}}{\sqrt2}\max{(\sqrt\hbar,\sqrt\hbar\mu'(\hbar),\sqrt{\mu'(\hbar)\nu'(\hbar)})}$ \hfill [case of  Theorem \ref{main2}]

\item \label{m3}

$E_\hbar(\widetilde W[R^{in}],R^{in})\leq 
\frac{D_{\ref{m3}}}{\sqrt2}\max{(\sqrt\hbar,\sqrt{\tau(\hbar)})}$ \hfill [case of  Theorem \ref{main3}]

\end{enumerate}
Here $D_{\ref{m1}}$, $D_{\ref{m2}}$,$D_{\ref{m3}}$ are given in \eqref{dm1}, \eqref{dm2}, \eqref{dm3} respectively.

\end{Thm}

\begin{proof}[Proof of Theorem \ref{main1}]\ 
Item \ref{m1} and Corollary \ref{armaarmacor} give immediately the first result in Theorem \ref{main1}. The second result is a corollary of the first thanks to Theorem \ref{ineqphinf}.
\end{proof}
\begin{proof}[Proof of Theorem \ref{main2}]\ 
Item \ref{m2} and Corollary \ref{armaarmacor} give again the first result in Theorem \ref{main}. The second result is a a corollary of the first thanks to Theorem \ref{ineqphinf} and Proposition \ref{propnest}.
\end{proof}
\begin{proof}[Proof of Theorem \ref{main3}]\ 
Finally, item \ref{m3}, Corollary \ref{armaarmacor} and Theorem \ref{ineqphinf} gives the statement.
\end{proof}
\begin{proof}[Proof of the corollary of Theorem \ref{main1}]\ 

 
 For $v\in\bR$ let us define $\{v\}$ as the smallest integer greater or equal to $v$.
 \begin{Lem}\label{lemdiff}\ 
 For any $M,K,N>0$, there exist $C_1>0$ such that
 \begin{eqnarray}\label{mkncn}
 \|a\|_{M,K,N}
 &\leq &
 C_1
 \sup_{\substack{(x,\xi)\in\bR^{2d}\\|\beta_1|,\dots,|\beta_d|\leq N+K}}
 (\xi^2+x^2+d)^{M+\{\frac K2\}}
 |\prod_{m=1}^dD_{(x,\xi)}^{\beta_m}  a(x,\xi)|\nonumber
 \end{eqnarray}
 \end{Lem}
 \begin{proof}
 Let us first remark that, with the defintion \eqref{baraka},
 \begin{eqnarray}
D_{z_m}^{\alpha_m} 
 a_k(z)&=&
 \int D_{z_m}^{\alpha_m} 
 a(e^{it}z) P(t)e^{ikt}dt
\end{eqnarray}
for a certain trigonometric polynomial $P$.

We have, for $K$ even,
\begin{eqnarray}
(k_l^2+1)^{\frac K2}
D_{z_m}^{\alpha_m} 
 a_k(z)&=&
 \int D_{z_m}^{\alpha_m} 
 a(e^{it}z) P(t)(\partial_t^2+1)^{\frac K2}e^{ikt}dt\nonumber\\
 &=&
 \int e^{ikt}(\partial_t+i)^{\frac K2}(\partial_t-i)^{\frac K2}(D_{z_m}^{\alpha_m} 
 a(e^{it}z) P(t))dt.\nonumber
\end{eqnarray}
Therefore, since $(|k_l|+1)^K\leq (k_l^2+1)^{\frac K2}$,
\begin{eqnarray}
&&\sup_{|\alpha_1|,\dots,|\alpha_d|\leq N}(|k_l|+d)^{K}
|\prod_{m=1}^dD_{z_m}^{\alpha_m}
 a_k(z)|
 \nonumber\\
 &\leq&
 \sup_{|\alpha_1|,\dots,|\alpha_d|\leq N}|(k_l^2+1)^{\frac K2}
\prod_{m=1}^dD_{z_m}^{\alpha_m} 
 a_k(z)|\nonumber\\
 &\leq&C
 \prod_{n=1}^d(|z_n|^2+1)^{\frac K2}\sup_{|\alpha_1|,\dots,|\alpha_d|\leq N+K}
 \int |\prod_{m=1}^dD_{z_m}^{\alpha_m} 
 a(e^{it}z)e^{ikt}|dt\nonumber\\
 &\leq&2\pi C
 (|z|^2+d)^{\frac K2}\sup_{\substack{|\alpha_1|,\dots,|\alpha_d|\leq N+K\\t\in [0,2\pi]}} |\prod_{m=1}^dD_{z_m}^{\alpha_m} 
 a(e^{it}z)|\label{profflemmmm}
\end{eqnarray} 
Since the family of norms $\|\cdot\|_{M,K,N}$ is non decreasing in $K$, we conclude by noticing that$\|\cdot\|_{M,K,N}\leq \|\cdot\|_{M,2\{\frac K2\},N}$, plugging \eqref{profflemmmm} in the definition \eqref{defmkn} and define $C_1=2\pi C$.
\end{proof}
Using Lemma \ref{lemdiff}, one sees easily that when $W_\hbar[\sqrt{R^{in}}]$ satisfies the hypothesis of the Corollary of Theorem \ref{main1}, the two hypothesis \eqref{tightwig} and \eqref{semicwig} of Theorem \ref{main1} are satisfied for $\mu(\hbar)=1$ and $\nu(\hbar)=\sqrt\hbar$ and the conclusion follows. 
\end{proof}
%
\vskip 1cm

\section{Proof of Theorem \ref{main} item \eqref{m3}}\label{proofmain3}
Let us start by a  lemma, interesting per se.
\begin{Lem}\label{verstriang}
Let $f$ be a probability density on 
$\bR^{2d}$ and  $R$ be a density operator on 
$L^2(\bR^d)$ 
satisfying
$$
\int_{\bR^{2d}}(p^2+q^2)f(p,q) dpdq<\infty\mbox{ and }
\Tr(-\hbar^2\Delta+x^2)R<\infty.
$$
The following inequality holds true:
\be\label{sqrtgen}
E_\hbar(\widetilde W[R],R)\leq
\sqrt{E_\hbar(f,R)^2+\frac{d\hbar}2}+E_\hbar(f,R)
\leq (\sqrt2+1)E_\hbar(f,R).
\ee
\end{Lem}
\begin{Cor}\label{corsqrtgen}
For any density operator satisfying 
$
\Tr(-\hbar^2\Delta+x^2)R<\infty,
$
$$
E_\hbar(\widetilde W[R],R)
\leq (\sqrt2+1)\inf_{g\in\cP(\bR^{2d})} E_\hbar(g,R).
$$
\end{Cor}
\begin{proof}
By Theorem \ref{armaprop} (3),
$$
\MKd(\widetilde W[R,f)^2\leq E_\hbar(f,R)^2+\frac{d\hbar}2
$$
and \eqref{sqrtgen} follows 
by Theorem \ref{armatriang} and Theorem \ref{armaprop} (1).
\end{proof}
\begin{Rmk}\label{rmksqrtgen}
The meaning of Corollary  \ref{corsqrtgen} is the following: thought $\widetilde W_\hbar[R]$ might not be the  classical density ``closets" to $R$, it belongs to a ``ball" of radius $(\sqrt2+1)$ around it.
\end{Rmk}
Let us come back tp the proof of Theorem \ref{main} item \eqref{m3} and  suppose  satisfied Hypothesis \ref{tight} and \ref{semic} in Theorem \ref{main3} . 

Let $F$  be given by Hypothesis \ref{tight} and  
{\bcr let us define, for each $0<\epsilon\leq 1$ 
$$
F_\epsilon(x,\xi)=(\epsilon (x^2+\xi^2)+1)^{-1}F(x,\xi)
$$
so that $F_\epsilon$ has finite second moments. }

Let $Q_0(x,\xi):=F_\epsilon(x,\xi)|x,\xi\rangle\langle x,\xi|$. By Corollary \ref{armapropcor}, $Q_0$ is a minimal coupling of $F_\epsilon$ and $\mbox{OP}^T_\hbar({\color{black}(2\pi\hbar)^d}F_\epsilon)$.
 
 Let us denote $A:=\sqrt{R^{in}}$ and $T_{F_\epsilon}:=\mbox{OP}^T_\hbar({\color{black}(2\pi\hbar)^d}F_\epsilon)$, and let
$$
Q(x,\xi)=AT_{F_\epsilon}^{-1/2}Q_0(x,\xi)T_{F_\epsilon}^{-1/2}A.
$$
 By the hypothesis of boundedness of $AT_{F_\epsilon}^{-1/2}$ and $T_{F_\epsilon}^{-1/2}A$, $Q$ is a coupling between $R^{in}$ and $f_A(x,\xi):={F_\epsilon}(x,\xi)\langle x,\xi|T_{F_\epsilon}^{-1/2}A^2T_{F_\epsilon}^{-1/2}|x,\xi\rangle$.
Therefore, by 
Lemma \ref{verstriang},
\be\label{sqrt}
E_\hbar(\widetilde W[R_\hbar],R_\hbar)\leq
\sqrt{E_\hbar(f_A,R_\hbar)^2+\frac{d\hbar}2}+E_\hbar(f_A,R_\hbar).
\ee
So the problem of estimating $E_\hbar(\widetilde W_\hbar[R^{in},R^{in})$
reduces
to estimating $E_\hbar(f_A,R_\hbar)$.
\vskip 0.5cm
\begin{Lem}\label{lemcost}
$$
c(x,\xi)Q_0(x,\xi)=\frac{d\hbar}2 Q_0(x,\xi).
$$
\end{Lem}
\begin{proof}
Denoting by $T^\hbar_{x,\xi},x,\xi\in\bR^d$ the Weyl operators acting on any function 
$\psi\in L^2(\bR^d)$ by
$
T^\hbar_{x,\xi}\psi(y)=\psi(y-x)e^{i(\xi.(y-x/2)}
$
(see e.g. \cite{FGPaulCohSta}, Section 1)
, we have
\begin{eqnarray}
c(x,\xi)Q_0(x,\xi)&=&
T(x,\xi)c(0,0)T(x,\xi)^{-1}Q_0(x,\xi)\nonumber\\
&=&
T(x,\xi)c(0,0)Q_0(0,0)T(x,\xi)^{-1}\nonumber\\
&=&
T(x,\xi)\frac{d\hbar}2 Q_0(0,0)T(x,\xi)^{-1}\nonumber\\
&=&
\frac{d\hbar}2 Q_0(x,\xi).\nonumber
\end{eqnarray}
\end{proof}
Using Lemma \ref{lemcost}, we get
\begin{eqnarray}
E_\hbar(f_A,R)^2&\leq&
\int dxd\xi\Tr{\left(c(x,\xi)Q(x,\xi)\right)}\nonumber\\
&=&
\int dxd\xi\Tr{\left(c(x,\xi)AT_{F_\epsilon}^{-1/2}Q_0(x,\xi)T_{F_\epsilon}^{-1/2}A\right)}\nonumber\\
&=&
\int dxd\xi\Tr{\left(AT_{F_\epsilon}^{-1/2}c(x,\xi)Q_0(x,\xi)T_{F_\epsilon}^{-1/2}A\right)}\nonumber\\
&+&
\int dxd\xi\Tr{\left([c(x,\xi),AT_{F_\epsilon}^{-1/2}]Q_0(x,\xi)T_{F_\epsilon}^{-1/2}A\right)}\nonumber\\
&=&
\frac{d\hbar}2\int dxd\xi\Tr{Q(x,\xi)}\nonumber\\
&+&
\int dxd\xi
\Tr{\left(T_{F_\epsilon}^{-1/2}A[2\xi.y+2 x.\nabla_y+y^2+\Delta_y,AT_{F_\epsilon}^{-1/2}]Q_0(x,\xi)\right)}\nonumber\\
&\leq&
\frac{d\hbar}2\nonumber\\
&+&
\int dxd\xi
\|T_{F_\epsilon}^{-1/2}A[2\xi.y+2 x.\nabla_y+y^2+\Delta_y,AT_{F_\epsilon}^{-1/2}]\|
\Tr{Q_0(x,\xi)}\nonumber\\
&=&
\frac{d\hbar}2+
\int dxd\xi
\|T_{F_\epsilon}^{-1/2}A[2\xi.y+2 x.\nabla_y+y^2+\Delta_y,AT_{F_\epsilon}^{-1/2}]\|
F(x,\xi)\nonumber\\
&\leq&
\frac{d\hbar}2+
\int(1+|x|+|\xi|){F_\epsilon}(dx,d\xi)
\sum_{O\in\Omega}\|T_{F_\epsilon}^{-1/2}A[O,AT_{F_\epsilon}^{-1/2}]\|\nonumber
\\
&=&
\frac{d\hbar}2+
\tau(\hbar)\int(1+|x|+|\xi|){F_\epsilon}(dx,d\xi)
\label{der}.
\end{eqnarray}

We conclude by using \eqref{der} and \eqref{sqrt}:

\begin{eqnarray}\label{sqrtdh}
E_\hbar(\widetilde W[R_\hbar],R_\hbar)&\leq&
2\sqrt{
\tau(\hbar)\int(1+|x|+|\xi|){F_\epsilon}(dx,d\xi)+{d\hbar}}
\nonumber
\leq 
2D_{\ref{m3}}(F)\max{(\sqrt\hbar,\sqrt{\tau(\hbar)})}.
\end{eqnarray} 
with, since ${F_\epsilon}\leq F$,
\be\label{dm3}
D_{\ref{m3}}(F)=2\max{(\sqrt{\int(1+|x|+|\xi|)F(dx,d\xi)},\sqrt d)}.
\ee
\vskip 1cm

\section{Proof of Theorem \ref{main} item \eqref{m2}}\label{proofmain2}
{\bcr We first notice the following elementary result.
\begin{Lem}\label{lemlem}
Let $\rho(\hbar),\tau'(\hbar)$ be two functions such that $\rho(\hbar)\tau'(\hbar)=o(1)$ as $\hbar\to 0$, and let $\sqrt{R^{in}}$ satisfy the two following hypothesis
\vskip 1cm 
(\ref{tight}')\centerline{$||T_F^{-1/2}\sqrt{R^{in}}||,\ ||\sqrt{R^{in}}T_F^{-1/2}||\leq 
\rho(\hbar)
$} 


(\ref{semic}')\centerline{$\sup\limits_{O\in\Omega}||[O,\sqrt{R^{in}}T_F^{-1/2}||\leq \tau'(\hbar).
$}
Then  \eqref{tight} and \eqref{semic} in Theorem \ref{main3} hold true with
$$
\tau(\hbar)=\rho(\hbar)\tau'(\hbar).
$$
\end{Lem}
}

\vskip 0.7cm

By Lemma \ref{lemlem}, Theorem \ref{main} item \eqref{m2} is  a consequence of the following result.
\begin{Prop}\label{main2soft}
{\bcr Let us suppose that $R_\hbar$ satisfies the hypothesis 
of Theorem \ref{main2} for some $\mu', \nu', \epsilon$. Then $R_\hbar$ satisfies (\ref{tight}') and (\ref{semic}') for 
any function $F$ of the form
$
F(z_1,\dots,z_d)=f(x^2_1+\xi^2_1)\dots f(x^2_d+\xi^2_d)$ with $f$ satisfying \eqref{condF} below 
and $\rho$ and $\tau'$ are given respectively by \eqref{mutorho} and \eqref{munutotau}.
}
\end{Prop}
\vskip 1cm

\begin{proof}
The proof of Proposition \ref{main2soft} is split in three steps. In Section \ref{constfmoinsundemi} we will first realize  $T_F^{-\frac12}$, for $F$ satisfying \eqref{condF} below, as a diagonal operator  on the Hermite basis. 
This will simplify the computation of the matrix elements and therefore the norms of the operators arising in  the hypothesis 
(\ref{tight})'  
and \eqref{semic}' 
in Lemma \ref{lemlem} out of the estimates of the matrix elements provided by hypothesis \eqref{iaj} and \eqref{ioaj} in Theorem \ref{main2}.This is done in Sections \ref{hypone} and \ref{hyptwo} respectively.

\subsection{Construction of $T_F^{-\frac12}$}\label{constfmoinsundemi}
Let us take $F$ in the form $F=f(x^2_1+\xi^2_1)\dots f(x^2_d+\xi^2_d)$. 
In other words, calling $h_0(x,\xi)=(\xi^2+x^2)/2$ defined on  $T^*(\bR)$, $F=f(h_0)^{\otimes d}$. It is easy to see that, since $|x,\xi\ra\la x,\xi|$ evolves by conjugation by the quantum flow of the harmonic oscillator by the action of the classical underlying flow on $(x,\xi)$, $\Op^T_\hbar({\color{black}(2\pi\hbar)^d}F)$  is  invariant by conjugation by the quantum flow of one dimensional harmonic oscillators $\cH_0$ as soon as $F$ is invariant by the action of the corresponding classical flow. 

Therefore $\Op^T_\hbar({\color{black}(2\pi\hbar)^d}F)=G(\cH_0)^{\otimes d}$ for some $G$
and, by \eqref{defher} and after denoting  $I=(1,\dots,1)\in\bN^d$,
$$
\mbox{OP}^T_\hbar({\color{black}(2\pi\hbar)^d}F)={\color{black}(2\pi\hbar)^d}\sum\limits_{j\in\bN^d}G^{\otimes d}((j+\tfrac12I)\hbar)|H_j\rangle\langle H_j|
=\big(
\sum\limits_{j\in\bN}{\color{black}(2\pi\hbar)}G
((j+\tfrac12)\hbar)|h_j\rangle\langle h_j|
\big)^{\otimes d}
.
$$ 
When there is no confusion we will drop the notation $I$ and write  for any $j\in\bN^d$,
$$
j+\tfrac12I=j+\tfrac12.
$$

The function $G$ can be evaluated on the spectrum of $\cH_0$ by the following argument: let $z'\in\bC$, $j\in\bN$ and $|z',\hbar\rangle$ be defined by \eqref{defetc}. By a simple computation using the generating function of the Hermite polynomials (see Section \ref{hermitegauss}) we easily find that
\be\label{zprimej}
\langle z'|h_j\rangle
=
\frac{{z'}^j}{\sqrt{j!\hbar^j}}e^{-\frac{|z'|^2}{2\hbar}},
\ee 
so that
\be\label{zprimez}
\langle z'|z\rangle
=\sum_{j\in\bN}\langle z'|h_j\rangle\langle h_j|z\rangle
=
e^{\frac{z'\bar z}{\hbar}}
e^{-\frac{|z|^2}{2\hbar}}
e^{-\frac{|z'|^2}{2\hbar}},
\ee  Then
\begin{eqnarray}
G((j+1/2)\hbar)\langle z'|h_j\rangle&=&\langle z'|G(\cH_0)|h_j\rangle=
\langle z'|\mbox{OP}^T_\hbar(f(h_0))|h_j\rangle\nonumber\\
(\mbox{by \eqref{zprimej}})\ \ \ &=&
\int  \langle z'|f(z\bar z)|z\rangle 
\frac{\bar{z}^j}{\sqrt{j!\hbar^j}}e^{-\frac{|z|^2}{2\hbar}}
dzd\bar z/(2\pi\hbar)
\nonumber\\
&=&
\int f(z\bar z) \langle z'|z\rangle 
\frac{\bar{z}^j}{\sqrt{j!\hbar^j}}e^{-\frac{|z|^2}{2\hbar}}
dzd\bar z/(2\pi\hbar)
\nonumber\\
(\mbox{by \eqref{zprimez}})\ \ \ &=&
\frac1\hbar 
\frac{{z'}^j}{\sqrt{j!\hbar^j}}e^{-\frac{|z'|^2}{2\hbar}}
\int_0^\infty f(\rho^2)\frac{\rho^{2j}}{j!\hbar^j}e^{-\rho^2/\hbar}\rho d\rho\nonumber\\
(\mbox{by \eqref{zprimej} again})\ \ \ &=&
\langle z'|h_j\rangle
\frac1\hbar 
\int_0^\infty f(\rho^2)\frac{\rho^{2j}}{j!\hbar^j}e^{-\rho^2/\hbar}\rho d\rho\nonumber
\end{eqnarray}
so that
\be\label{goutoff}
G((j+1/2)\hbar)=\int_{\bR^+} f(\rho^2)\frac{(\rho^2/\hbar)^j}{j!\hbar}e^{-\rho^2/\hbar}\rho d\rho,\ \ \    
\ee

In particular, since $0\neq f\geq 0$, $G((j+1/2)\hbar)>0$. 
Therefore $\mbox{OP}^T_\hbar(F)^{-1/2}$ can be defined  
as an unbounded operator by
\be\label{definv}
\mbox{OP}^T_\hbar({\color{black}(2\pi\hbar)^d}F)^{-1/2}|h_j\ra=
(({\color{black}(2\pi\hbar)}G)^{-1/2})^{\otimes d}((j+1/2)\hbar)
|h_j\ra,\ \forall j\in\bN^d.
\ee

\vskip 1cm
\begin{Rmk}\label{gauss}
Note that, when $f\in C^2(\bR)$, the stationary phase lemma gives that $G(A)\sim f(A)$ as $ j\to\infty,\hbar\to 0,(j+1/2)\hbar\to A<\infty$. Moreover, \eqref{goutoff} can be inverted by
$$
\hbar\sum_{j\in\bN}G((j+1/2)\hbar)(1-\lambda)^j=\mathbb L(f)(\lambda),
$$
where $\mathbb L$ is the Laplace transform.

Note also that $G$ is bounded since ${(\rho^2/\hbar)^j}e^{-\rho^2/\hbar}\leq j^je^{-j}\leq j!$ and $f$ is integrable, in fact $G\leq 1/\hbar$.

Note finally that the Gaussian case is explicitly computable.
\end{Rmk}
\vskip 1cm
\begin{Lem}\label{puissnace}

Let us suppose there exist {\bcr $\nu> \tfrac32$},  and  $C_1,C_0>0$ such that, 
\be\label{condF}
\left\{
\begin{array}{ccl}
- 2C_0(\rho^{2}+1)^{-\nu-1}&\leq &f'(\rho^2)\leq 0\\
C_1(\rho^2+1)^{-\nu}
&\leq& f(\rho^2).\
\end{array}
\right.
\ee

Then, for all $A\in \bR^+$
\be\label{sup1}
C'(A+1)^{-\nu}\leq G(A)
\ee
where $C'$ is given by 
\eqref{sup3} below.

Moreover
\begin{eqnarray}\label{diff}
G((j+1/2)\hbar)-G((j+3/2)\hbar)&=&-\hbar \int_0^\infty f'(\rho^2)\frac{(\rho^2/\hbar)^{j+1}}{(j+1)!\hbar}e^{-\rho^2/\hbar}\rho d\rho\nonumber\\
&\leq&
 4^{\nu+1}\hbar C_0(j\hbar+1/2))^{-(\nu+1)}.
\end{eqnarray}
\end{Lem}
Note that the condition on $\nu$ is compatible with the finiteness of 
$\int(|x|+|\xi|)F(dx,d\xi)$. 
\begin{proof}
We first remark that $C_1(\rho^2+1)^{-\nu}$ implies that, for all $\alpha>0$ and
 $C=C_1(\frac{1+\alpha}\alpha)^\nu$.
\begin{eqnarray}
C\rho^{-2\nu}&\leq& f(\rho^2),\ \ \ \ \rho^2\geq \alpha\nonumber\\
C_1&\leq&f(\rho^2).\ \ \ \ \rho^2< \alpha\nonumber
\end{eqnarray}

Since $f>0$ we have, for $j\geq \nu$, by \eqref{condF}
\begin{eqnarray}
G((j+\tfrac12)\hbar)&\geq&
\int_{\alpha}^\infty
f(\rho^2)\frac{(\rho^2/\hbar)^j}{j!\hbar}e^{-\rho^2/\hbar} d(\rho^2/2)\nonumber\\
&\geq&
\int_{\alpha}^\infty
\frac C{\hbar^\nu j(j-1)\dots(j-\nu+1)}\frac{(\rho^2/\hbar)^{j-\nu}}{(j-\nu)!\hbar}e^{-\rho^2/\hbar}d(\rho^2/2)\nonumber\\
&\geq&
C(j\hbar)^{-\nu}\big(1-\int_0^{\frac\alpha\hbar} \frac{x^{j-\nu}e^{-x}}{(j-\nu)!}dx\big)\nonumber\\
j'=j-\nu\ \ \ 
&\geq&
C(j\hbar)^{-\nu}\big(1-(\alpha/\hbar)^{j'}(j')^{-j'}e^{j'}\int_0^\infty e^{-x}dx\big)\hskip 1cm\mbox{(by Stirling)}\nonumber\\
&\geq&
C(j\hbar)^{-\nu}\big(1-e^{j'(1+\log{\frac\alpha{j'\hbar}})}\big)\nonumber\\
&\geq&
C(j\hbar)^{-\nu}\big(1-\frac1e\big)
\mbox{ for }  j\hbar\geq e^2\alpha+\nu\hbar.
\nonumber\\
&\geq&
C\tfrac{e-1}e((j+\tfrac12)\hbar+1)^{-\nu}
\nonumber
\end{eqnarray}
For $j=0$ we have that
\begin{eqnarray}
G(\hbar/2)&=&
\frac12
\int_{\bR^+}
f(\lambda)
e^{-\frac\lambda\hbar}
\frac{d\lambda}\hbar
=
\frac12
\int_{0}^1
f(\hbar\lambda)
e^{-\lambda}
{d\lambda}
\geq\frac{C_1}{2e}
:=C''
\nonumber
\end{eqnarray}
For $0<A:=j\hbar\leq e^2\alpha+\nu\hbar$, we get, since 
 $j!\leq K\sqrt{2\pi}j^j\sqrt je^{-j}$
  for some $K>1$, 
\begin{eqnarray}
G((j+1/2)\hbar)=G(A+\hbar/2)&\geq&\frac{e^{\frac A\hbar}}{2K\sqrt{2\pi}\sqrt{A}}\int_{\bR^+}
f(\lambda)e^{-\frac{\lambda-A\log{(\frac\lambda A)}}\hbar}\frac{d\lambda}{\sqrt\hbar}\nonumber\\
%
&\geq&
\frac{e^{\frac A\hbar}}{2K\sqrt{2\pi}\sqrt{A}}\int_{A-\frac12\sqrt{\hbar A}}^{A+\frac12\sqrt{\hbar A}}
f(\lambda)e^{-\frac{\lambda-A\log{(\frac\lambda A)}}\hbar}\frac{d\lambda}{\sqrt\hbar}\nonumber\\
\mbox{(since $\log{(1+x)}-x\geq -\tfrac12 x^2$})\ \ \ \ &\geq& \frac1{2K\sqrt{2\pi}}\int_{-\frac12}^{+\frac12}f(A+\sqrt{\hbar A}\mu)e^{-\frac{\mu^2}2}d\mu\nonumber\\
&\geq&
\frac1{2K\sqrt{2\pi}}\inf_{-\frac12\leq\mu\leq+\frac12}f(A+\sqrt{\hbar A}\mu)
\int_{-\frac12}^{+\frac12}e^{-\frac{\mu^2}2}d\mu\nonumber\\
&\geq&
\frac{C_1e^{-\tfrac18}}{K\sqrt{2\pi}}
:=C'''>0\nonumber
\end{eqnarray}
Therefore \eqref{sup1} holds true if we define 
\be\label{sup3}
C'=\inf{(C,C'',C''')}={C_1}/{2\pi e^{\frac18}}.
\ee
Moreover, by integration by part,
\begin{eqnarray}
G((j+1/2)\hbar)&=&
-\int_0^\infty
\frac{(\rho^2/\hbar)^{j+1}}{(j+1)!\hbar}\frac d{d(\frac{\rho^2}\hbar)}(f(\rho^2)e^{-\rho^2/\hbar})\rho d\rho\nonumber\\
&=&
\int_0^\infty
f(\rho^2)\frac{(\rho^2/\hbar)^{j+1}}{(j+1)!\hbar}e^{-\rho^2/\hbar}\rho d\rho\nonumber\\
&-&
\int_0^\infty
\hbar f'(\rho^2)\frac{(\rho^2/\hbar)^{j+1}}{(j+1)!\hbar}e^{-\rho^2/\hbar}\rho d\rho
\nonumber
\end{eqnarray}
This proves the first part of \eqref{diff}. 

For $0<\eta<\infty$ we decompose the middle term in \eqref{diff} in two parts. Since by the inequality of \eqref{condF}, $-f'(\rho^2)\leq 2C_0$, we first compute
\begin{eqnarray}
-\int_0^{\sqrt\eta}
\hbar f'(\rho^2)\frac{(\rho^2/\hbar)^{j+1}}{(j+1)!\hbar}e^{-\rho^2/\hbar}\rho d\rho
&\leq&\hbar C_0\int_0^{\eta/\hbar}
\frac{(\lambda)^{j+1}}{(j+1)!}e^{-\lambda} d\lambda
\leq \hbar C_0\nonumber
\end{eqnarray}
Using the first inequality of \eqref{condF} we get that
\begin{eqnarray}
-\int_{\sqrt\eta}^\infty
\hbar f'(\rho^2)\frac{(\rho^2/\hbar)^{j+1}}{(j+1)!\hbar}e^{-\rho^2/\hbar}\rho d\rho
&=&
-\frac12\int_{\eta}^\infty f'(\lambda') \frac{e^{-\frac{\lambda'-(j+1)\hbar\log{\lambda'}}\hbar}}{\hbar^{j+1}(j+1)!}d\lambda'\nonumber\\
&=&
-\frac\hbar2
\int_{\eta/\hbar}^\infty
 f'(\hbar\lambda)\frac{\lambda^{j+1}}{(j+1)!}e^{-\lambda} d\lambda\nonumber\\
 &\leq&
 \hbar C_0\int_0^\infty (\lambda\hbar+1)^{-(\nu+1)}\frac{\lambda^{j+1}}{(j+1)!}e^{-\lambda} d\lambda\nonumber\\
 &\leq&
 \hbar C_0\int_0^\infty
 \min{\big(\frac{\hbar^{-(\nu+1)}}{(j+1)j(j-1)}\frac{\lambda^{j-\nu}}{(j-2)!},
 \frac{\lambda^{j+1}}{(j+1)!}
 \big)}e^{-\lambda} d\lambda\nonumber\\
 &\leq&
 \hbar C_0\min{((\hbar(j-(\nu-1)))^{-(\nu+1)},1)}\nonumber\\
 &\leq& \hbar C_02^{\nu+1}(\hbar(j-(\nu-1))+1)^{-(\nu+1)}\nonumber\\
&\leq&
 4^{\nu+1}\hbar C_0(j\hbar+\tfrac12)^{-(\nu+1)}.
\nonumber
\end{eqnarray}
\end{proof}
Let us summarize what has been done in this section: out of any $F=f^{\otimes d}$ where $f$ satisfies   condition \eqref{condF}, we have defined 
 $\mbox{OP}^T_\hbar({\color{black}(2\pi\hbar)^d}F)^{-1/2}$  by the formula
\begin{eqnarray}\label{defsqrt}
\mbox{OP}^T_\hbar({\color{black}(2\pi\hbar)^d}F)^{-1/2}
&=&
(2\pi\hbar)^{-d/2}\sum\limits_{j\in\bN^d}
g^{\otimes d}((j+1/2)\hbar)
|H_j\rangle\langle H_j|,\nonumber
\end{eqnarray}
in the sense that
\be\label{diag}
\langle H_j|\Op^T_\hbar({\color{black}(2\pi\hbar)^d}F)^{-1/2}
|H_i\rangle
=
(2\pi\hbar)^{-d/2}\delta_{i,j}
g^{\otimes d}((j+1/2)\hbar)
\ee
where $g$ satisfies, for all $A\in\bR^+$,
\begin{eqnarray}
g(A)&\leq& 
{\bcr (C')^{-\frac 12}}
(A+1)^{\frac\nu2},
\label{defg2}\\
g^{-2}((j+\tfrac12)\hbar)-g^{-2}((j+\tfrac32)\hbar)&\leq& 
 4^{\nu+1}\hbar C_0(A+\tfrac12)^{-(\nu+1)}
\label{defg3}
\end{eqnarray}
for $C,C_0$ defined in \eqref{condF} and $C'$ in \eqref{sup3}.

From now on, we will suppose that $f$ satisfies \eqref{condF} {\bcr for some $\nu>\tfrac32$}.

\subsection{Hypothesis \eqref{iaj} in Theorem \ref{main2} 
$\Longrightarrow$ (\ref{tight}') in Lemma \ref{lemlem}}\label{hypone}\ 

Condition (\ref{tight}')  is easily checkable by the control of the decay of $\ket{H_{j+k}}\sqrt{R^{in}}\bra{H_j}$ as both $j\mbox{ and }|k|\to\infty$. 
%

Let us recall that, \cite{ds} Chapter VI, Paragraph 9, Exercise 54, the operator norm of  any operator $B$ on a separable Hilbert space $\mathcal H$,  satisfies, for any orthonormal basis $\{\bra{j}\}$ of $\mathcal H$,
\be\label{lemds}
\|B\|\leq \max{\big(\sup_j\sum_{j'}|\ket{j'}B\bra{j}|,\sup_{j'}\sum_{j}|\ket{j'}B\bra{j}|\big)}
\ee

In our case $B=\sqrt{R^{in}}\mbox{OP}^T_\hbar({\color{black}(2\pi\hbar)^d}F)^{-1/2}$ so that, after taking {\bcr $\nu=\tfrac32+2\epsilon, \ 0<\epsilon\leq 1,$ (say) }in \eqref{diag},
$$
\ket{j+k}B\bra{j}={\color{black}(2\pi\hbar)^{-d/2}}\ket{j+k}\sqrt{R^{in}}\bra{j}g^{\otimes d}((j+1/2)\hbar)
$$
 with  $\sqrt{R^{in}}$ self-adjoint and 
$g$ satisfying 
\eqref{defg2}.

Therefore,  one has, under the condition \eqref{iaj}  in Theorem \ref{main2},
\begin{eqnarray}
\ket{j'}\sqrt{R^{in}}\mbox{OP}^T_\hbar({\color{black}(2\pi\hbar)^d}F)^{-1/2}\bra{j}
&\leq& {\bcr (C')^{-\frac d2}\mu'(\hbar)\prod\limits_{1\leq l\leq d}(|j_l-j'_l|+1)^{-\frac32-\epsilon}}
\nonumber
\end{eqnarray}
%
{\bcr Since $$\sum_{j'_l=0}^\infty(|j_l-j'_l|+1)^{-\frac32-\epsilon}=
\sum_{j_l=0}^\infty(|j_l-j'_l|+1)^{-\frac32-\epsilon}\leq 1,$$ 

 and $C'= C_1/(2\pi e^{\frac18})$, \eqref{lemds} gives
\be\label{supj}
\|\sqrt{R^{in}}\mbox{OP}^T_\hbar(F)^{-1/2}\|
\leq
(2\pi)^{\frac d2}e^{\frac d{16}}C_1^{-\frac d2}\mu'(\hbar)
\ee
Therefore, the tightness condition (\ref{tight}') in Proposition \ref{main2soft} is satisfied with
\be\label{mutorho}
\rho(\hbar)=
(2\pi)^{\frac d2}e^{\frac d{16}}C_1^{-\frac d2}
\mu'(\hbar).
\ee
}

\vskip 1cm

\vskip 1cm
\vskip 1cm
\subsection{Hypothesis \eqref{iaj} and \eqref{ioaj} in Theorem \ref{main2} $\Longrightarrow$ (\ref{semic}') in Lemma \ref{lemlem}}\label{hyptwo} \ 

In order to prove  the semiclassical condition (\ref{semic}'),
we first remark that
\be\label{semicla}
[O,\sqrt{R^{in}}T_F^{-1/2}]=\sqrt{R^{in}}[O,T_F^{-1/2}]+[O,\sqrt{R^{in}}]T_F^{-1/2}.
\ee

The second term of the right hand side is treated exactly the same way as for (\ref{tight}'):

{\bcr - for $O=\mathcal H_l=-\hbar^2\nabla_l^2+x_l^2$ we first note that
$$
\la j'|[\mathcal H_l,\sqrt{R^{in}}]|j\ra=(j'_l-j_l)\hbar
\la j'|\sqrt{R^{in}}|j\ra,
$$
so that, by \eqref{iaj}  in Theorem \ref{main2},
$$
|\la j'|[\mathcal H_l,\sqrt{R^{in}}]T_F^{-\frac12}|j\ra|
\leq C'^{-d}\hbar\mu'(\hbar)
\prod\limits_{1\leq l\leq d}(|j_l-j'_l|+1)^{-1-\epsilon}.
$$
Since $\prod\limits_{1\leq l\leq d}(|j_l-j'_l|+1)^{-1-\epsilon}\leq 2/\epsilon$ and $C'= C_1/(2\i e^{\frac18})$, we get, by \eqref{lemds},
\be\label{oscharm}
\|[\mathcal H_l,\sqrt{R^{in}}]T_F^{-\frac12}\|
\leq 
(2\pi)^{\frac d2}e^{\frac d{16}}(C_1\epsilon^2)^{-\frac d2}
\hbar\mu'(\hbar).
\ee

- for $O\neq \mathcal H_l$, \eqref{ioaj} in Theorem \ref{main2} gives precisely by the same argument
\be\label{noscharm}
\|[O,\sqrt{R^{in}}]T_F^{-\frac12}\|
\leq 
(2\pi)^{\frac d2}e^{\frac d{16}}(C_1\epsilon^2)^{-\frac d2}\nu'(\hbar).
\ee
}

{\bcr For the first term in the right hand-side of \eqref{semicla}, let us start with $O_l=x_l+\hbar\partial_{x_l}, l=1,\dots,d$, and let us remember that $O_lH_j=\sqrt{(j_l+1)\hbar}H_{j+1_l}$ and $O_l^*H_j=\sqrt{j_l\hbar}H_{j-1_l}$, as is easily checked.

Defining $1_l\in\bN^d$ by  $(1_l)_k=\delta_{k,l},k,l=1,\dots,d$, one has, with $A:=\sqrt{R^{in}}$,
\begin{eqnarray}
&&\la H_i|A[O_l,T^{-1/2}]|H_j\ra\nonumber\\
&=& (
g
^{\otimes d}((j+\tfrac12)\hbar)
-
g
^{\otimes d}((j+1_l+\tfrac12)\hbar))
\big(\sqrt{(j_l+1)\hbar}\la H_i|A|H_{j+1_l}\rangle\nonumber
\end{eqnarray}

so that (remember $g=G^{-\frac12}$)
\begin{eqnarray}
&&\label{truc}\\
&&|\la H_{j+k}|A[O_l,T^{-1/2}]|H_j\ra|\leq\nonumber\\
&& |G((j_l+\tfrac12)\hbar)^{-1/2}
-
G((j_l+\tfrac32)\hbar)^{-1/2}|)
\sqrt{(j_l+1)\hbar} |((j_l+\tfrac12)\hbar+1)^{-\tfrac12-\epsilon}
(|k_l-1|+1)^{-2-\epsilon}
\nonumber\\
&&\times
\mu'(\hbar)
\prod_{m\neq l}
|g(j_m+\tfrac12)\hbar)
|((j_m+\tfrac12)\hbar+1)|
^{-\tfrac12-\epsilon}
(|k_m|+1)^{-2-\epsilon}.
\nonumber
\end{eqnarray}


- the second factor in the  right hand-side of \eqref{truc}, namely
$$
C
\prod_{m\neq l}
|G(j_m+\tfrac12)\hbar)
|((j_m+\tfrac12)\hbar+1)|
^{-\tfrac12-\epsilon}
(|k_m|+1)^{-1-\epsilon},
$$
 will gives again,  aa for the derivation of \eqref{supj} by Lemma \ref{lemds}, a contribution to $\|A[O,T^{-1/2}]\|$ bounded by  $
\mu'(\hbar)
(2\pi)^{d-1}e^{\frac {d-1}8}C^{-(d-1)}_1
$.


- for the estimate of the first factor in the right hand-side of \eqref{truc}, 
\begin{eqnarray}
&&
|G((j_l+\tfrac12)\hbar)^{-1/2}
-
G((j_l+\tfrac32)\hbar)^{-1/2}|)
\sqrt{(j_l+1)\hbar} |((j_l+\tfrac12)\hbar+1)^{-\tfrac12-\epsilon}
(|k_l-1|+1)^{-2-\epsilon}\nonumber\\
&\leq&
|G((j_l+\tfrac12)\hbar)^{-1/2}
-
G((j_l+\tfrac32)\hbar)^{-1/2}|)
((j_l+\tfrac12)\hbar+1)^{-\epsilon} (|k_l-1|+1)^{-2-\epsilon}\nonumber
\end{eqnarray}
we first remark that, since $f'\leq0$ and \eqref{diff} holds true, one can extend $G$ to $\bR^+$ with  $G'(x)\leq 0,\ x\in\bR^+$. Therefore, 
\begin{eqnarray}
&&|G((j_l+\tfrac12)\hbar)^{-1/2}
-
G((j_l+\tfrac32)\hbar)^{-1/2}|\nonumber\\
&=&\tfrac12
|\int_)^\hbar
G'((j+\tfrac12)\hbar+\lambda)
G((j+\tfrac12)\hbar+\lambda)^{-\tfrac32}
d\lambda|\nonumber\\
(\mbox{since }G'\leq 0\leq G)\ \ \ \ \ &\leq&\tfrac12
\sup_{0\leq\lambda\leq\hbar}{(G((j+\tfrac12)\hbar+\lambda)^{-\frac32})}
|G((j+\tfrac32)\hbar-G((j+\tfrac12)\hbar)|\nonumber\\
&=&\tfrac12
\sup_{0\leq\lambda\leq\hbar}{(g((j+\tfrac12)\hbar+\lambda)^{3})}
|G((j+\tfrac32)\hbar-G((j+\tfrac12)\hbar)|
\nonumber\\
(by\ \eqref{sup1}\ and\ \eqref{diff})\ &\leq&
4^{\nu+1}\hbar C_0(2\pi e^{\frac18}C_1)^{-\frac 32}
(j\hbar+1/2)^{\frac\nu 2-1}.\nonumber
\end{eqnarray}
Recalling that $\nu=\tfrac32+2\epsilon$ we get that the first factor in the right hand-side of \eqref{semicla} is bounded by 
$$4^{\nu+1}
\hbar C_0(|k_l-1|+1)^{-1-\epsilon}=
4^{\nu+1}\hbar C_0(2\pi e^{\frac18}C_1)^{-\frac 32}(|(j+k)_l-j_l-1|+1)^{-1-\epsilon}
$$
 Therefore, it gives a contribution to $\|A[O,T^{-1/2}]\|$, by  \eqref{lemds} again, bounded by  $8\hbar C_0$.

Putting together the two contributions of \eqref{truc}, we get that
\be\label{aot}
\|A[O,T^{-1/2}]\|
\leq4^{\nu+1} C_0(2\pi e^{\frac18}C_1)^{-\frac d2-1}\hbar\mu'(\hbar).
\ee


The term with $O_l^*=x_l-\partial_{x_l}=(x_l+\partial_{x_l})$ is done the same way. 
{}

Putting  together \eqref{oscharm}, \eqref{noscharm} and  \eqref{aot}, 
we get the semiclassical condition (\ref{semic}') in Proposition \ref{main2soft} with
\be\label{munutotau}
\tau'(\hbar)=
(2\pi e^{\frac18}C_1)^{- \frac d2-1}
(4^{\nu+1}C_0+2\pi e^{\frac18}C_1\epsilon^{-d})\sup{(\hbar\mu'(\hbar),\nu'(\hbar))}.
\ee
so that, by \eqref{mutorho},
Theorem \ref{main} item \eqref{m2} is a consequence of item \eqref{m3} with
$$
\tau(\hbar)=\rho(\hbar)\tau'(\hbar)=
(2\pi e^{\frac18}C_1)^{- d-1}
(4^{\nu+1}C_0+2\pi e^{\frac18}C_1\epsilon^{-d})\sup{(\hbar\mu'(\hbar)^2,\nu'(\hbar)\mu'(\hbar))}
$$
Taking  $f(\rho^2)=C_1(1+\rho^2)^{-2-2\epsilon}$ with  $ 2\pi e^{\frac18}C_1 \leq 1$ leads to 
\be\label{munutotau2}
\tau(\hbar)
\leq 
(4^6(2+\epsilon)+\epsilon^{-d})
\sup{(\hbar\mu'(\hbar)^2,\nu'(\hbar)\mu'(\hbar))}
\ee
}
Therefore item $\ref{m2})$ is proven for 
\be\label{dm2}
D_{\ref{m2}}=(4^6(2+\epsilon)+\epsilon^{-d})D_{\ref{m3}}
\ee
 

\end{proof}

\vskip 1cm

\section{Proof of Theorem \ref{main} item \textit{\eqref{m1}}}\label{proofmain1}
In order to prove Theorem \ref{main} item \eqref{m1} as a corollary of item \eqref{m2},  we need to estimate, under the  hypothesis $(1)$ and $(2)$ of Theorem \ref{main1}, the matrix elements  of $\sqrt{R^{in}}$ and $[O,\sqrt{R^{in}}]$, $O\in\Omega_1$, between semiclassical  Hermite functions, uniformly in $\hbar$. We will carry out this computation out of their Wigner transform, or equivalently their Weyl symbols (which we recall to be their Wigner functions multiplied by $(2\pi\hbar)^d$).

Of course no exact formula for these matrix elements exists, and WKB expansions involve unsatisfactory (for our purpose) remainder terms. Nevertheless, there exist explicit formulas for $H_j$ expressed in terms of coherent states, which, together with a formulation of Weyl calculus involving Wigner operators (see below), will allow us to conclude. The proof is a bit involved, and we split it in several steps.
\subsection{Hermite as exact Gaussian quasimodes}\label{hermitegauss}
We first take $d=1$.

From the formula for the generating function of the Hermite polynomials $\{h^p_n\}_{n=0,\dots}$, namely
$$
\sum_{n=0}^\infty
h^p_n(x)\frac{z^n}{n!}=e^{2xz-z^2},\ \|h^p_n\|_{L^2(\bR,e^{-{x^2}}dx)}=\pi^\frac14 \sqrt{2^nn!}
$$
valid for all $z\in\bC$, we get easily that the normalized eigenfunctions of the harmonic oscillator $h_j,j=0,\dots$ satisfy
\be\label{defgj}
\sum_{j=0}^\infty
h_j(x)
\frac{(\frac z{\sqrt{
\hbar}})^{j}e^{-\frac{|z|^2}{2\hbar}}}
{\sqrt {
 j!}}
 =
(\pi\hbar)^{-\frac14}
e^{
-\frac{z^2-2\sqrt 2 zx+x^2}{2\hbar}}e^{-\frac{|z|^2}{2\hbar}}:=g^z(x).
\ee
Note that $\|g^z\|_{L^2(\bR,dx)}=1$.

We get that, for any choice of $z\in\bC, z\neq 0$,
$$
(2\pi)^{-\frac34}
\sqrt{\frac{|z|}{\sqrt\hbar}}
\int_0^{2\pi}
g^{e^{it}z}
e^{-ijt}
dt
=
\sqrt{\frac{\big(\frac{|z|^2}{\hbar}\big)^{j}e^{-\frac{|z|^2}{\hbar}}\sqrt{2\pi\frac {|z|^2}{\hbar}}}{j!}}
h_j
$$
so that \begin{eqnarray}
h_j
&=&
\sqrt{\frac{j!}{\big(\frac{|z|^2}{\hbar}\big)^{j}e^{-\frac{|z|^2}{\hbar}}\sqrt{2\pi\frac {|z|^2}{\hbar}}}}
(2\pi)^{-\frac34}
\sqrt{\frac{|z|}{\sqrt\hbar}}
\int_0^{2\pi}
g^{e^{it}z}
e^{-ijt}
dt\nonumber\\
&:=&C_j(z)
(2\pi)^{-\frac34}
\sqrt{\frac{|z|}{\sqrt\hbar}}
\int_0^{2\pi}
g^{e^{it}z}
e^{-ijt}
dt\label{defcj}
\end{eqnarray}
for all $z\neq 0\in\bC$.
In particular, taking now ${|z|^2}=(j+1/2)\hbar$,
$$
\bar C_j:=C_j(\sqrt{(j+1/2)\hbar})=
\sqrt{\frac{j!}
{
\sqrt{2\pi j}j^je^{_-j}
}}<1\mbox{ and }\sim 1\mbox{ as }
j\to\infty,
$$
since $\sqrt{2\pi (j+1/2)}(j+1/2)^je^{-(j+1/2)}>\sqrt{2\pi j}j^je^{_-j}$ and  $\sqrt{2\pi (j+1/2)}(j+1/2)^je^{-(j+1/2)}\sim\sqrt{2\pi j}j^je^{-j}$ as $j\to\infty$.

The same way,
\begin{eqnarray}
h_{j'}
&=&
C_{j'}(\sqrt{(j'+1/2)})
(2\pi)^{-3/4}(j'+1/2)^{1/4}
\int_0^{2\pi}
g^{e^{i\theta}\sqrt{(j'+1/2)\hbar}
}
e^{-ij'\theta}
d\theta,\nonumber\\
\mbox{ but also }\ \ \ 
h_{j'}&=&
C_{j'}(\sqrt{(j+1/2)})
(2\pi)^{-3/4}(j+1/2)^{1/4}
\int_0^{2\pi}
g^{e^{i\theta}\sqrt{(j+1/2)\hbar}
}
e^{-ij'\theta}
d\theta.\nonumber\\
\mbox{Therefore\ \ \ \ }
h_{j'}&\sim&
\sqrt{\frac{j'!}{j!}}
(2\pi)^{-3/4}(j+1/2)^{j-j'+1/4}
\int_0^{2\pi}
g^{e^{i\theta}\sqrt{(j+1/2)\hbar}
}
e^{-ij'\theta}
d\theta,\nonumber
\end{eqnarray}
since $\sqrt{2\pi j}j^je^{-j}\sim \sqrt{2\pi (j+1/2)}j^je^{-(j+1/2)}$ as $j\to\infty$.

Therefore, for any bounded operator $A$,

\begin{eqnarray}
\ket{j'}A\bra{j}
&=&
\frac{\bar C_j\bar C_{j'}}{(2\pi)^\frac3e}
(j+1/2)^{1/4}(j'+1/2)^{1/4}
\int_0^{2\pi}ds'\int_0^{2\pi}ds
e^{i(j's'-js)}
(g_{}^{e^{is'}\sqrt{j'\hbar}},Ag_{}^{e^{is}\sqrt{j\hbar}}).
\nonumber\\
\mbox{ but also }&&\nonumber\\
\ket{j'}A\bra{j}
&=&
\frac{\bar C_jC_{j'}(\sqrt{(j+1/2)\hbar})}{(2\pi)^\frac3e}
(j+1/2)^{1/2}
\int_0^{2\pi}ds'\int_0^{2\pi}ds
e^{i(j's'-js)}
(g_{}^{e^{is'}\sqrt{j\hbar}},Ag_{}^{e^{is}\sqrt{j\hbar}}).\nonumber
\end{eqnarray}
Here we have used the Dirac notation, and $\ket{j'}A\bra{j}$ is meant for 
$(h_{j'},Ah_j)_{L^2(\bR,dx)}$.

In particular,
\begin{eqnarray}
\label{jprimeajo}&&\\
|\ket{j'}A\bra{j}|
&\leq&\nonumber
\frac{(j+1/2)^{1/4}(j'+1/2)^{1/4}}{(2\pi)^\frac32}
\left|
\int_
{\bT^2}
ds'
ds
e^{i(j's'-js)}
(g_{}^{e^{is'}\sqrt{(j'+\frac12)\hbar}},Ag_{}^{e^{is}\sqrt{(j+\frac12)\hbar}})
\right|.
\\
\mbox{ but also }&&\nonumber\\
\label{jprimeaj}
|\ket{j'}A\bra{j}|
&\leq&
(2\pi)^{-\frac32}
\sqrt{\frac{j'!}{j!}(j+1/2)^{j-j'+1}}
\left|\int_
{\bT^2}
ds'
ds
e^{i(j's'-js)}
(g_{}^{e^{is'}\sqrt{(j'+\frac12)\hbar}},Ag_{}^{e^{is}\sqrt{(j+\frac12)\hbar}})\right|\nonumber\\
&\leq&
\frac{e^{\frac{(j-j')^2}{j+1/2}}(j+1/2)^{1/2}}{(2\pi)^{\frac32}}
\left|\int_
{\bT^2}
ds'
ds
e^{i(j's'-js)}
(g_{}^{e^{is'}\sqrt{(j'+\frac12)\hbar}},Ag_{}^{e^{is}\sqrt{(j+\frac12)\hbar}})\right|
\end{eqnarray}
Indeed, if $j'=j+k, k\geq 0$ (the case $k<0$ is done the same way),
\begin{eqnarray}
\frac{j'!}{j!}(j+1/2)^{j-j'}
\leq
\frac{(j+1/2+k)^k}{(j+1/2)^k}\leq
(1+\frac k{j+1/2})^k\leq e^{\frac{k^2}{j+1/2}}\nonumber
\end{eqnarray}
We'll estimate later $|\ket{j'}A\bra{j}|$ by \eqref{jprimeaj} for $e^{\frac{(j-j')^2}{j+1/2}}\leq 1$ and by \eqref{jprimeajo} for $e^{\frac{(j-j')^2}2}>1$.
%
 
\vskip 1cm
The extension to the $d$-dimensional case is straightforward, by tensorial factorization.

\subsection{Weyl calculus through Wigner operators}\label{weylwigner}
Let us suppose now that $A$ is 
a Hilbert-Schmidt operator of Wigner function $a\in L^2(\bR^d)$. 
Note that this is to say that, in a weak sense,
\be\label{wiga}
A=\int_{\bC^{d}} 
d^2z
a(z)
W(-z)IW(z)
:=\mbox{Wig}(a)
\ee where the Wigner operators $W(z)$ are the unitary operators on $L^2(\bR^d)$ defined through (denoting $z=(q,p)$)
$$
W(q,p)f(x)=f(x-q)e^{-i\frac{px}\hbar}e^{-i\frac{pq}4}\mbox{ and }If(x):=f(-x).
$$
Indeed the integral kernel of 
$W(-z)IW(z)$ is $\delta(-x-2q-y)e^{i\frac{p(2x-\frac q2)}\hbar}$ 
so that the integral kernel of $A=\int _{\bC^{d}}d^2z 
a(z)
W(-z)IW(z)$ is
$$
\rho_A(x,y)=\int_{\bR^d} a(\frac{x+y}2,\xi)e^{i\frac{\xi(x-y)}\hbar}
d\xi
$$
which is the well-known Weyl quantization formula expressed on the Wigner function (let us recall that the latter is nothing but the Weyl symbol divided by $(2\pi\hbar)^d$).


We get easily that

\begin{eqnarray}
(g^{z'},W(-z_0)IW(z_0)g^z)
&=&
e^{\frac{\bar{z}z_0-\bar{z_0}z}{\sqrt 2\hbar}}
e^{-\frac{\bar{z'}(z+\sqrt 2z_0)}\hbar}
e^{-\frac{|z'|}{2\hbar}}
e^{-\frac{|z+\sqrt 2z_0|}{2\hbar}}\nonumber\\
&=&
e^{-\frac{|z'|^2}{2\hbar}-\frac{|z|^2}{2\hbar}-\frac{|z_0|^2}{\hbar}
-\frac{\bar{z'}z}\hbar-\sqrt 2\frac{\bar{z'}z_0+z\bar{z_0}}\hbar}.\nonumber
\end{eqnarray}
so that
\begin{eqnarray}
(g^{z'},W(-\frac{z_0}{\sqrt 2})IW(\frac{z_0}{\sqrt 2})g^z)
&=&
e^{-\frac{|z'|^2}{2\hbar}-\frac{|z|^2}{2\hbar}-\frac{|z_0|^2}{2\hbar}
-\frac{\overline{z'}z}\hbar-\frac{\overline{z'}z_0+z\overline{z_0}}\hbar}\nonumber
\end{eqnarray}
Note that
$$
\Re\big(\tfrac{|z'|^2}{2}+\tfrac{|z|^2}{2}+\tfrac{|z_0|^2}{2}
+{\overline{z'}z}+{\overline{z'}z_0+z\overline{z_0}}\big)
=|z'+z+z_0|^2\geq 0.
$$

For $d=1$ 
we get that \eqref{jprimeajo} becomes

\begin{eqnarray}
|\ket{j'}A\bra{j}|
&\leq&
\frac{(j+1/2)^{1/4}(j'+1/2)^{1/4}}{(2\pi)^\frac32}B_{j,j'}\label{bonbonoa}
\\
B_{j,j'}
&:=&
\left|
\int_{S^1}d\theta'\left(
\int_{\bR^2} d^2z
\int_{S^1}d\theta 
a\big(\frac z{\sqrt2}\big)e^{-\frac{\phi_{j,j'}(z,\bar z,\theta,\theta')}\hbar}\right)e^{i(j'-j)\theta'}
\right|
\label{bonbono}\\
\label{defphi2}
\phi_{j,j'}(z,\overline{z},\theta,\theta')
&:=&
\frac{(j+\tfrac12)\hbar}2+\frac{(j'+\tfrac12)\hbar}2
+\frac{z\bar z}2+\sqrt{(j+\tfrac12)(j'+\tfrac12)}\hbar e^{i(\theta-\theta')}\\
&&+\bar z\sqrt{(j+\tfrac12)\hbar}e^{i\theta}
+
z\sqrt{(j'+\tfrac12)\hbar}e^{-i\theta'}+ij\hbar(\theta-\theta')
.\nonumber
\end{eqnarray}
 and \eqref{jprimeaj} becomes

\begin{eqnarray}
|\ket{j'}A\bra{j}|
&\leq&
\frac{e^{\frac{(j-j')^2}{j+1/2}}(j+1/2)^{1/2}}{(2\pi)^{\frac32}}
A_{j,j'}\label{bonbona}\\
A_{j,j'}
&:=&
\left|
\int_{S^1}d\theta'\left(
\int_{\bR^2} d^2z
\int_{S^1}d\theta 
a\big(\frac z{\sqrt2}\big)e^{-\frac{\phi_j(z,\bar z,\theta,\theta')}\hbar}\right)e^{i(j'-j)\theta'}
\right|
\label{bonbon}\\
\label{defphi}
\phi_j(z,\overline{z},\theta,\theta')
&:=&
(j+\tfrac12)\hbar+\frac{z\bar z}2+(j+\tfrac12)\hbar e^{i(\theta-\theta')}\\
&&+\bar z\sqrt{(j+\tfrac12)\hbar}e^{i\theta}
+
z\sqrt{(j+\tfrac12)\hbar}e^{-i\theta'}+ij\hbar(\theta-\theta')
.\nonumber
\end{eqnarray}
Performing in \eqref{bonbono} and \eqref{bonbon} first the change of variable $z\to ze^{i\theta'}$ and then the one $\theta\to\theta+\theta'$ we get easily that, after definition \eqref{baraka}
\begin{eqnarray}
B_{j,j'}&=&
\left|
\int_{\bR^2} d^2z
\int_{S^1}d\theta 
a_{j-j'}\big(\frac z{\sqrt2}\big)e^{-\frac{\Phi_{j,j'}(z,\bar z,\theta)}\hbar}
\right|
\label{ajjprimeo}
\\
A_{j.j'}&=&
\left|
\int_{\bR^2} d^2z
\int_{S^1}d\theta 
a_{j-j'}\big(\frac z{\sqrt2}\big)e^{-\frac{\Phi_{j}(z,\bar z,\theta)}\hbar}
\right|
\label{ajjprime}\\
\mbox{ with }\ \ \ 
\Phi_{j,j'}(z,\overline{z},\theta)
&:=&
\frac{(j+\tfrac12)\hbar}2+\frac{(j'+\tfrac12)\hbar}2
+\frac{z\bar z}2+\sqrt{(j+\tfrac12)(j'+\tfrac12)}\hbar e^{i\theta}\label{defphijjrpime}\\
&&+\bar z\sqrt{(j+\tfrac12)\hbar}e^{i\theta}
+
z\sqrt{(j'+\tfrac12)\hbar}+ij\hbar\theta
\nonumber\\
\mbox{ and }\ \ \ \ \ \ \Phi_j(z,\overline{z},\theta)&:=& \Phi_{j,j}(z,\overline{z},\theta)=
(j+\tfrac12)\hbar+\frac{z\bar z}2+(j+\tfrac12)\hbar e^{i\theta}
\label{defphi}\\
&&+\bar z\sqrt{(j+\tfrac12)\hbar}e^{i\theta}
+
z\sqrt{(j+\tfrac12)\hbar}+ij\hbar\theta.
\nonumber
\end{eqnarray}

\ 

%
Let us recall that
\be\label{imphi}
\Re{\Phi_{j,j'}(z,\bar z,\theta)}
=
|\sqrt{(j+\tfrac12)\hbar} e^{i\theta_l}+\sqrt{(j'+\tfrac12)\hbar} 
+z|^2\geq 0
\ee
so that, since $\Phi_j=\Phi_{j,j}$,
\be\label{alinfnity}
A_{j,j'},B_{j,j'}\leq \frac{(2\pi)^3\hbar}2\|{a}_{j-j'}\|_{L^\infty(\bR^{2d})}.
\ee

Since $\Phi^d_{j,j'}$ is a sum of one dimensional contributions, we will treat the one dimensional case. The (straightforward) extension to $d$ dimensions will be done in the next section.
\subsection{
The case
$d=1$}\label{endproofdone}
\ 

Note that for $d=1$, {\bcr \eqref{defmkn} reads, for any $K,M\in\bR^+,\ N\in \bN$ 
 \be\label{defmkninv}
 \sup_{0\leq |\alpha|\leq N}
 |\hbar^{\frac{|\alpha|}2}D^\alpha_z {a}_k(z)|,
 \leq \|a\|_{M,K,N}(|z|^2+1)^{-M}(|k|+1)^{-K}
 \ee 
 }
\vskip 1cm
\subsubsection{Matrix elements estimates}\label{matesti}\ 
In this section we prove the following crucial result.
\begin{Prop}\label{propmatesti}
Let $A$ by a Hilbert-Schmidt operator (say). 

Then for all $M,\alpha,\tau,\mu,n\geq 0$,
\begin{eqnarray}
&&|\ket{j'}A\bra{j}|\leq
D_{\mu,n}\frac{2\pi\hbar\|W_\hbar[A]\|_{M,1+\epsilon+\tau,n}}{||j'-j|+\tfrac12|^{\min{(1+\alpha+\epsilon,1+\epsilon+\tau)}}((j+\tfrac12)\hbar+\tfrac12)^{\min{(\frac{\bvc-\alpha-1}2,\mu-\frac32,M,\frac {n-1}2)}}}.\label{moinsplus2}
\end{eqnarray} 
where $D_{\mu,n}$ is given below by \eqref{constot}.
\end{Prop}
\begin{proof}
The desired estimates of $|\ket{j'}A\bra{j}|$, expressed as in \eqref{jprimeajo}, that is \eqref{bonbonoa} or \eqref{bonbona}, will be obtained by different methods, depending on the size of $|j'-j|$. 
\begin{enumerate}
\item for $|j'-j|^2> j+\tfrac12$, we will use \eqref{bonbono} and \eqref{ajjprimeo}, together with \eqref{alinfnity}, to estimate $|\ket{j'}A\bra{j}|$.
\item for $|j'-j|^2\leq j+\tfrac12$, we will use \eqref{bonbon} and \eqref{ajjprime}: we will perform a smooth partition of the domain of integration in $|z|$ in the integral present in the expression $A_{j',j}$ in \eqref{ajjprime}.
\begin{enumerate}
\item for $|z|\geq \frac52\sqrt{(j+\tfrac12)\hbar}$, the decay given by the negative imaginary part of the phase will be enough.
\item 
for $\sqrt{(j+\tfrac12)\hbar}\leq|z|\leq 3\sqrt{(j+\tfrac12)\hbar}$, we will have two cases:
\begin{enumerate}
\item $|j'-j|>0$, the desired estimates will be obtained by a brutal estimate of the integral of $|a_{j'-j}|$ on the domain.
\item $j+j'$, the stationary phase lemma applied to the integral on the domain will give the needed estimate.
\end{enumerate}
\item
for $|z|\leq \frac32\sqrt{(j+\tfrac12)\hbar}$, the non stationary phase Lemma will apply and give the result.
\end{enumerate}

\end{enumerate}
\vskip 1cm
Note that since $A$ is self-adjoint, one can restrict to the case $j'\leq j$ without loss of generality.
\vskip 1cm


\begin{enumerate}
\item Suppose $|j'-j|^2>j$.
By \eqref{bonbono}, \eqref{ajjprimeo} and \eqref{alinfnity} 
we get for any $\alpha$
{\bcr
\begin{eqnarray}
|\ket{j'}A\bra{j}|&\leq &
\tfrac{(2\pi)^{\frac32}\hbar}2(j+\tfrac12)^{1/4}(j'+\tfrac12)^{1/4}
 \|{a}_{j-j'}\|_{L^\infty(\bR^{2d})}
 \nonumber\\
&\leq &
\tfrac{(2\pi)^{\frac32}\hbar}2
\|{a}\|_{M,1+\tau+\epsilon,N}
(j+\tfrac12)^{1/4}(j'+\tfrac12)^{1/4}(|j-j'|+1)^{-1-\epsilon-\bvc}\nonumber\\
&\leq&
\tfrac{(2\pi)^{\frac32}\hbar}2
\|{a}\|_{M,1+\tau+\epsilon,N}(j+\tfrac12)^{1/2}((|j-j'|+1)^{-1-\alpha-\epsilon}(\sqrt j+1)^{-\bvc+\alpha}\nonumber\\
&\leq&\tfrac{(2\pi)^{\frac32}\hbar}2
\|{a}\|_{M,1+\tau+\epsilon,N}(|j-j'|+1)^{-1-\alpha-\epsilon}( j+1)^{-(\bvc-\alpha-1)/2}\nonumber\\
&\leq&\tfrac{(2\pi)^{\frac32}}2\hbar
\|{a}\|_{M,1+\tau+\epsilon,N}(|j-j'|+1)^{-1-\alpha-\epsilon}( (j+\tfrac12)\hbar+\tfrac12)^{-(\bvc-\alpha-1)/2}.\label{plus}
\end{eqnarray}
}

\vskip 1cm
\item

We now consider  the case $|j'-j|^2\leq j+\tfrac12$. We will use \eqref{bonbon} which becomes
\be\label{bonboninf}
|\ket{j'}A\bra{j}|
\leq
\frac{(j+1/2)^{1/2}}{(2\pi)^{\frac32}}
A_{j,j'}
\ee

We first perform a change of variable $z\to ((j+\tfrac12)\hbar)^{ 1/2}z$ in the integral inside the definition of $A_{j,j'}$ in \eqref{bonbon}.
 We get
 \begin{eqnarray}\label{ajjrpime31}
 A_{j,j'}&=&
\left|
\int_{\bR^2} d^2z
\int_{S^1}d\theta 
{a}_{j-j'}\big(\frac{z}{\sqrt2}\big)
e^{-\frac{\Phi_j(z,\bar z,\theta')}\hbar}
\right|\nonumber\\
&=&
(j+\tfrac12)\hbar
\left|
\int_{\bR^2} d^2z
\int_{S^1}d\theta 
{a}_{j-j'}\big(\tfrac{\sqrt{(j+\frac12)\hbar}}{\sqrt2}z\big)
e^{-(j+\frac12)\frac{\Phi(z,\bar z,\theta)}\hbar}
\right|.
\end{eqnarray}
{with }
\be\label{phiprime}
{\Phi(z,\bar {z},\theta)}
=
1+
\frac{z\bar{ z}}2+ e^{i\theta}+\bar {z}e^{i\theta}
+
z
+i\theta.
\ee

We first  perform a smooth decomposition of the identity $1=\chi^{\leq}+\chi^{=}+\chi^{\geq}$, such that, 
\be\label{chichi}
\left\{
\begin{array}{ccl}
\sup{\chi^{\leq}}&=&\{z,|z|\leq
\frac32
\}\\ &&\\
\sup{\chi^{=}}&=&\{z,
1
\leq |z|\leq 
3
\}\\ &&\\
\sup{\chi^{\geq}}&=&\{z,|z|\geq 
\frac52
\}
\end{array}
\right.
\ee
and decompose
$a_{j-j'}$ in three terms {\bcr $a^\bullet_{j-j'},
\ \bullet\in\{\leq,=,\geq\}$:
$$
a^\bullet_{j-j'}\big(\tfrac{\sqrt{(j+\tfrac12))\hbar}z}{\sqrt2}\big):=\chi^\bullet(z)  a_{j-j'}\big(\tfrac{\sqrt{(j+\tfrac12))\hbar}z}{\sqrt2}\big).
$$
}
Note that, since the $\chi^\bullet$ are smooth and constant for $|z|\geq 3$,
\be\label{faitch}
|D_z^\alpha a^\bullet_{j-j'}\big(\tfrac{\sqrt{(j+\tfrac12))\hbar}z}{\sqrt2}\big)|\leq
C_{\chi,|\alpha|}\sum_{|\beta|\leq|\alpha|}
|D_z^\beta a_{j-j'}\big(\tfrac{\sqrt{(j+\tfrac12))\hbar}z}{\sqrt2}\big)|
\ee
with
\be\label{cchi}
C_{\chi,|\alpha|}=2^{|\alpha|}\sum_{\substack{\bullet\in\{\leq,=,\geq\}\\|\beta|\leq|\alpha|}}\|D_z^\beta\chi^\bullet\|_{L^\infty(\bR^2)}.
\ee
Of course, $C_{\chi,0}=1$.
\begin{enumerate}
\item $\bullet=\ \geq$.

By \eqref{imphi}, $\Re\Phi(z,\bar z\theta,\theta') \geq\frac{1}4$ for $z$ in the support of $a^{\geq}_{j-j'}$, the latter gives a contribution to $A_{j,j'}$ (which is linear in $a$) in \eqref{ajjprime} equal to
{\bcr
\begin{eqnarray}\label{suppp}
A_{j,j'}^\geq&\leq& 2(2\pi)^2(j+\tfrac12)\hbar\|a_{j-j'}^\geq\|_{L^\infty(\bR^{2d})}\int_{|z|\geq \frac52}
e^{-(j+\frac12)|z+e^{i\theta}+e^{-i\theta'}|^2}dzd\bar zd\theta d\theta'.\nonumber\\
&\leq&
2(2\pi)^2(j+\tfrac12)\hbar\|a_{j-j'}^\geq\|_{L^\infty(\bR^{2d})}e^{-(j+\frac12)/4}\int_{|z|\geq \frac52}
e^{-(j+\frac12)(|z+e^{i\theta}+e^{-i\theta'}|^2-\frac14)}dzd\bar zd\theta d\theta'.\nonumber\\
&\leq&
2(2\pi)^2(j+\tfrac12)\hbar\|a_{j-j'}^\geq\|_{L^\infty(\bR^{2d})}e^{-(j+\frac12)/4}\int_{|z|\geq \frac52}
e^{-\frac12(|z+e^{i\theta}+e^{-i\theta'}|^2-\frac14)}dzd\bar zd\theta d\theta'\nonumber\\
&\leq&
2(2\pi)^2(j+\tfrac12)\hbar\|a_{j-j'}\|_{L^\infty(\bR^{2d})}e^{-(j+\frac12)/4}(2\pi)^2\int_{\bC}e^{\frac18}
e^{-\frac12(|z|^2}dzd\bar z\nonumber\\
&\leq&
2(2\pi)^5e^{\frac18}(j+\tfrac12)\hbar
\|{a}\|_{M,1+\epsilon+\tau,N}
(|j-j'|+1)^{-1-\epsilon-\tau}e^{-(j+\frac12)/4}.\nonumber
\end{eqnarray}

so that, by \eqref{defmkninv} again,
\begin{eqnarray}\label{moinssup}
|\ket{j'}\mbox{Wig}(a^\geq)\bra{j}|
&\leq&
2(2\pi)^{\frac72}e^{\frac18}
\hbar\|{a}\|_{M,1+\epsilon+\tau,N}
(|j-j'|+1)^{-1-\epsilon-\tau}
(j+\tfrac12)^{\frac32}e^{-\frac{j+\frac12}4}\nonumber\\
&\leq&
{2(2\pi)^{\frac72}e^{\frac18}(\tfrac{4\mu}e)^\mu }
\hbar\|{a}\|_{M,1+\epsilon+\tau,N}
(|j'-j|+1)^{-1-\epsilon-\bvc}((j+\tfrac12)\hbar+\tfrac12)^{\frac32-\mu} 
\end{eqnarray}
for all $\mu>\frac12$.
}
\vskip 1cm
\item $\bullet=\ =$

 We will use the stationary phase Lemma for estimating $A^=_{j,j}$ given by \eqref{ajjprime}.  
We get
\begin{eqnarray}
\partial_{z }\Phi(z,\bar z,\theta) 
&=&
\frac{\bar {z }}2+1\nonumber\\
\partial_{\bar {z }}\Phi(z,\bar z,\theta)
&=&
\frac{ z }2+e^{i\theta''}\nonumber\\
\partial_{\theta}\Phi (z,\bar z,\theta)
&=&
i( e^{i(\theta)}+\bar {z }e^{i\theta}+1),\nonumber
\end{eqnarray}
so that the critical point of $\Phi $ is at $\theta=0,z =-2$ and, at the critical point,

$$\Phi =0,\ d^2\Phi =\begin{pmatrix}
0&1/2&0\\
1/2&0&1\\
0&1&1
\end{pmatrix}\mbox{ and }\det(d^2\Phi )=-1/4.
$$

The stationary phase Lemma (with quadratic phase) gives immediately that, as $j\to\infty$ (independently of $\hbar$),
\begin{eqnarray}\label{moinsequal}
A^=_{j,j}&\leq& 
{\bcr \frac2{\pi^{\frac32}}}
\frac{(j+\tfrac12)\hbar}{(j+\tfrac12)^{3/2}}
\left(|a_{j'-j}^=(\sqrt{2(j+\tfrac12)\hbar})|\right.\nonumber\\
&&\left. +\frac1{16}\frac{
1
} {j+\tfrac12}
\sup_{\substack{|\alpha|=2\\
1\leq|z|\leq 3
}}
|D_z^\alpha a_{j'-j}^=\big(\sqrt{\tfrac{(j+\frac12)\hbar}2}z\big)|
\right),\ \ \ \nonumber\\
&\leq&{\bcr \frac2{\pi^{\frac32}}(1+\tfrac{C_{\chi,2}}{4})}
\|a\|_{M,1+\epsilon+\tau,N}(|j-j'|+1)^{-1-\epsilon-\tau}((j+\tfrac12)\hbar+1)^{-M}
 \nonumber
 \end{eqnarray}
 since, by \eqref{faitch}, 
 \begin{eqnarray}
 &&|D_z^\alpha a_{j'-j}^=\big(\sqrt{\tfrac{(j+\frac12)\hbar}2}z\big)|\nonumber\\
 &\leq&
 C_{\chi,|\alpha|}
 \sum_{|\beta|\leq|\alpha|}
|D_z^\beta a_{j-j'}\big(\tfrac{\sqrt{(j+\tfrac12))\hbar}z}{\sqrt2}\big)|
\nonumber\\
&\leq&
 C_{\chi,|\alpha|}
\sum_{|\beta|\leq|\alpha|}
\big(\tfrac{\sqrt{(j+\tfrac12))}}{\sqrt2}\big)^{|\beta|}
|\big((\sqrt\hbar D_z)^\beta a_{j-j'}\big)\big(\tfrac{\sqrt{(j+\tfrac12))\hbar}z}{\sqrt2}\big)|\nonumber\\
&\leq&
 C_{\chi,|\alpha|}2^{\frac{|\alpha|}2}(j+\tfrac12)^{\frac{|\alpha|}2}\sum_{|\beta|\leq|\alpha|}
\big(\tfrac{\sqrt{(j+\tfrac12))}}{\sqrt2}\big)^{|\beta|}
|\big((\sqrt\hbar D_z)^\beta a_{j-j'}\big)\big(\tfrac{\sqrt{(j+\tfrac12))\hbar}z}{\sqrt2}\big)|\nonumber\\
&\leq&
C_{\chi,|\alpha|}2^{\frac{|\alpha|}2}(j+\tfrac12)^{\frac{|\alpha|}2}
\|a\|_{M,1+\epsilon+\tau,N}(|j-j'|+1)^{-1-\epsilon-\tau}((j+\tfrac12)\hbar|z|^2+1)^{-M}\label{dingue}\\
&\leq&
C_{\chi,|\alpha|}2^{\frac{|\alpha|}2}(j+\tfrac12)^{\frac{|\alpha|}2}
\|a\|_{M,1+\epsilon+\tau,N}(|j-j'|+1)^{-1-\epsilon-\tau}((j+\tfrac12)\hbar+1)^{-M},\ |z|\geq 1.\nonumber
 \end{eqnarray}
 Here we have use the fact that $\sqrt{j+\tfrac12}\leq2(j+\tfrac12), j=0.1,\dots$.
 
 Therefore, by \eqref{defmkninv},
 \begin{eqnarray}\label{moinsequal}
|\ket{j'}\mbox{Wig}(a^=)\bra{j}|&\leq&
{\bcr \frac1{(2\pi)^{\frac32}}\frac2{\pi^{\frac32}}(1+\tfrac{C_{\chi,2}}{4})}
\hbar
\|a\|_{M,1+\epsilon+\tau,N}(|j-j'|+1)^{-1-\epsilon-\tau}((j+\tfrac12)\hbar+1)^{-M}\nonumber\\
&\leq&
\frac{1+\tfrac{C_{\chi,2}}{4}}{\sqrt2\pi^3}
\hbar
\|a\|_{M,1+\epsilon+\tau,N}(|j-j'|+1)^{-1-\epsilon-\tau}((j+\tfrac12)\hbar+1)^{-M}
\end{eqnarray}
%
\vskip 1cm
\item $\bullet=\ \leq$

Finally, we will estimate 
the contribution given by $a_{j-j'}^{\leq}$
by first estimating the integral  
$\int_{
} d^2z
{a}_{j-j'}^{\leq}\big(\frac {\sqrt{(j+\tfrac12)\hbar} z}{\sqrt2}\big)e^{-(j+\frac12){\Phi(z,\bar z,\theta)}}$ pointwise in $\theta$ by the non-stationary phase method. Indeed, 
\begin{eqnarray}
\partial_z\Phi
&=&
\frac{\bar z}2+
e^{-i\theta'}\\
\partial_{\bar z}\Phi
&=&
\frac{ z}2+e^{i\theta}.
\end{eqnarray}


Therefore, 
for $z\in\sup{a^\leq_{j-j'}}$,
$$
\frac12
\leq|D_z\Phi|\leq 
\frac72
.
$$
%
Using now that
$$
e^{-(j+\tfrac12){\Phi(z,\bar z,\theta)}}
=
\frac1{j+\tfrac12}
\frac1{|D_z\Phi|^2}D_z\Phi\cdot D_z
e^{-(j+\tfrac12){\Phi(z,\bar z,\theta)}}
$$
{\bcr Let
$$
D_\Phi:=D_z\cdot(D_z\Phi)|D_z\Phi|^{-2}=\big({|D_z\Phi|^{-2}}D_z\Phi\cdot  D_z\big)^*.
$$
We get by integration by part that, for any $n\in\bN$, }
{\color{black}
\begin{eqnarray}\label{nk}
A_{j,j'}&\leq&(j+\tfrac12)\hbar(j+\tfrac12)^{-n}\left|
\int_{
}
D_\Phi^n\left({a}_{j-j'}^{\leq}\big(\tfrac {\sqrt{(j+\tfrac12)\hbar}z}{\sqrt2}\big)\right)e^{-(j+\tfrac12){\Phi(z,\bar z,\theta)}} d^2z d\theta 
\right|\nonumber\\
&\leq& 
(j+\tfrac12)\hbar\frac{2\pi}{(j+\tfrac12)^{n}}
\sup_{\substack{|z|\leq \frac32\\(\theta,\theta')\in\bT^2}}\left| D_\Phi^n\left({a}_{j-j'}^{\leq}\big(\tfrac {\sqrt{(j+\tfrac12)\hbar}z}{\sqrt2}\big)\right)\right|\int 
e^{-(j+\tfrac12){\Re\Phi(z,\bar z,\theta,\theta')}} d^2z.\nonumber
\end{eqnarray}
One has
\begin{eqnarray}
&&\sup_{\substack{|z|\leq \frac32\\(\theta,\theta')\in\bT^2}}\left|D_\Phi^n
\left({a}_{j-j'}^{\leq}\big(\tfrac {\sqrt{(j+\tfrac12)\hbar}z}{\sqrt2}\big)\right)\right|\nonumber\\
&\leq&
2^n\sup_{\substack{|z|\leq \frac32\\(\theta,\theta')\in\bT^2}}\sup_{|\beta|+k\leq n}
|D^\beta
(\tfrac{\Phi(z)}{|D_z\Phi(z)|^2})^k
|
\sum_{|\alpha|\leq n}
\left| D_z^\alpha 
\left({a}_{j-j'}^{\leq}\big(\tfrac {\sqrt{(j+\tfrac12)\hbar}z}{\sqrt2}\big)\right)\right|\nonumber\\
&\leq&
2^n
C_\Phi
|
\sum_{|\alpha|\leq n}
\left| D_z^\alpha 
\left({a}_{j-j'}^{\leq}\big(\tfrac {\sqrt{(j+\tfrac12)\hbar}z}{\sqrt2}\big)\right)\right|\nonumber\\
&\leq&
2^nC_\Phi
C_{\chi,n}2^{\frac{n}2}(j+\tfrac12)^{\frac{n}2}
\|a\|_{M,1+\epsilon+\tau,N}(|j-j'|+1)^{-1-\epsilon-\tau}\nonumber
\end{eqnarray}
by \eqref{dingue}. 

We have used
 $(j+\tfrac12)^k\leq 2^n(j+\tfrac12)^n,\ j=0,1,\dots,\ 0\leq k\leq n$ and define
\be\label{defccc}
C_\Phi=\sup_{\substack{|z|\leq \frac32\\(\theta,\theta')\in\bT^2}}\sup_{|\beta|+k\leq n}
|D^\beta
(\tfrac{\Phi(z)}{|D_z\Phi(z)|^2})^k
|
.
\ee
 
Therefore
\begin{eqnarray}\label{nk}
A_{j,j'}
&\leq &
2^{\frac32 n}C_\Phi
C_{\chi,n}
\frac{2\pi^2}{(j+\tfrac12)^{\frac n2}}\hbar
\|a\|_{M,1+\epsilon+\tau,N}(|j-j'|+1)^{-1-\epsilon-\tau}
\nonumber
\nonumber,
\end{eqnarray}


so that,
for 
by \eqref{bonbona},
\begin{eqnarray}\label{moinsinf}
&&|\ket{j'}\mbox{Wig}(a_{j'-j}^{\leq})\bra{j}|\\
&\leq&
\frac{2^{\frac32 n}}{(2\pi)^{\frac32}}C_\Phi
C_{\chi,n}
\frac{2\pi^2}{(j+\tfrac12)^{\frac {n-1}2}}\hbar
\|a\|_{M,1+\epsilon+\tau,N}(|j-j'|+1)^{-1-\epsilon-\tau}\nonumber\\
&\leq&
C_\Phi
C_{\chi,n}\frac{2^{\frac32(n-1)}\pi^{\frac12}}{(j+\tfrac12)^{\frac {n-1}2}}\hbar
\|a\|_{M,1+\epsilon+\tau,N}(|j-j'|+1)^{-1-\epsilon-\tau}\nonumber\\
&\leq&
2^{\frac{n-1}2}C_\Phi
C_{\chi,n}
\hbar\|a\|_{M,1+\epsilon+\tau,N}(|j-j'|+1)^{-1-\epsilon-\tau}
((j+\tfrac12)+1)^{-\frac{n-1}2}\nonumber
\end{eqnarray}
}
%
%
\end{enumerate}
\end{enumerate}
{\bcr Adding the estimates given by \eqref{plus}, \eqref{moinssup}, \eqref{moinsequal} and \eqref{moinsinf}, we get  \eqref{moinsplus2}
with
\be\label{constot}
2\pi D_{\mu,n}=
\frac{(2\pi)^{3/2}}2
+
{2(2\pi)^{\frac72}e^{\frac18}(4\mu)^\mu e^{-\mu}}
+
\frac{1+\tfrac{C_{\chi,2}}{4}}{\sqrt2\pi^3}
+
2^{\frac{n-1}2}C_\Phi
C_{\chi,n}.
\ee
}
\end{proof}

\subsubsection{Hypothesis \ref{tightwig} in Theorem \ref{main1} $\Longrightarrow$  Condition \ref{iaj} in Theorem \ref{main2}}\label{tightonetotwo}\ 
 
{\bcr
By Proposition \ref{propmatesti} $\sqrt{R^{in}}$ satisfies the hypothesis \eqref{iaj} in Theorem \ref{main2} as soon as
$$
\min{(1+\alpha+\epsilon,1+\epsilon+\tau)}
=
2
+\epsilon,\ 
\min{(\tfrac{\bvc-\alpha-1}2,\mu-\tfrac32,M,\tfrac {n-1}2)}= 
{\color{black}\tfrac34}
+\epsilon,
$$
that is
$$
{\color{black} 
\alpha= 1,\ \tau= 
\tfrac72
+2\epsilon,\ \mu= 
\tfrac94
+\epsilon
, M= 
\tfrac34
+\epsilon,
 \ n=
3
}
$$ and 
\be\label{conda}
\|W_\hbar[\sqrt{R^{in}}]
\|_{
{\color{black}\frac34+\epsilon,\frac72+3\epsilon,3}
}\leq (2\pi\hbar)^{-\frac12}\frac{\mu'(\hbar)}{D_{\frac94+\epsilon,3}}
:=(2\pi\hbar)^{-\frac12}\mu(\hbar)
\ee
%
%

\subsubsection{Hypothesis 
\ref{semicwig} in Theorem \ref{main1} $\Longrightarrow$  Condition \ref{ioaj} in Theorem \ref{main2}}\label{semiconetotwo}\ 

Since
\begin{eqnarray}
W_\hbar[[x,\sqrt{R^{in}}]](q,p)&=&i\hbar\partial_pW_\hbar[\sqrt{R^{in}}](q,p)\nonumber\\
  W_\hbar[[-i\hbar\nabla,\sqrt{R^{in}}]](q,p)&=&-i\hbar\partial_qW_\hbar[\sqrt{R^{in}}](q,p),\nonumber
\end{eqnarray}
we get by Proposition \ref{propmatesti},
that the hypothesis \eqref{ioaj} in Theorem \ref{main2} is satisfied as soon as
\be\label{condnabla1}
\hbar\|\nabla W_\hbar[\sqrt{R^{in}}]
\|
_{
\frac34
+\epsilon,
\frac52
+3\epsilon,
3
}
\leq (2\pi\hbar)^{-\frac12}
\frac{\nu'(\hbar)}{D_{\frac94+\epsilon,3}}
\ee
It is easy to check that, by the definition \eqref{defmkn} of
 $\|a\|_{M,K,N}$
 $$
 \|\nabla a\|_{M,K,N}\leq \hbar^{-\frac12}2^K\|a\|_{M,K,N+1}.
 $$
Indeed, let us do the computation in the $z,\bar z$ variables. We have
$$
|(\partial_za)_k|=|\partial_za_{k+1}|
$$ 
so that, since $|k-1|+1\leq |k|+2=2(|k|/2+1)\leq 2(|k|+1)$,
\begin{eqnarray}
\|\partial_z a\|_{M,K,N}&=&\sup_{z\in\bR^{2d}}(|z|^2+1)^{M}(|k|+1)^{K}\sup_{0\leq |\alpha|\leq N}
 |(\sqrt\hbar D)^\alpha  \partial_za_{k(z)+1}|\nonumber\\
 &\leq &
 \sup_{z\in\bR^{2d}}(|z|^2+1)^{M}(|k-1|+1)^{K}
\sup_{0\leq |\alpha|\leq N}
 |(\sqrt\hbar D)^\alpha  \partial_za_{k}(z)|\nonumber\\
 &\leq & 
 2^K\hbar^{-\frac12}\sup_{z\in\bR^{2d}}(|z|^2+1)^{M}(|k|+1)^{K}
\sup_{0\leq |\alpha|\leq N+1}
 |(\sqrt\hbar D)^\alpha  a_{k}(z)|.\nonumber
\end{eqnarray}
  The term $\partial_{\bar z}a$ is obtained the same way.
 
Therefore \eqref{condnabla1} is satisfied as soon as

\be\label{condnabla2}
\hbar^{\frac12}\| W_\hbar[\sqrt{R^{in}}]
\|
_{
\frac34
+\epsilon,
\frac52
+3\epsilon,
4
}
\leq (2\pi\hbar)^{-\frac12}
\frac{\nu'(\hbar)}{2^{\frac52
+3\epsilon}D_{\frac94+\epsilon,4}}
:=
(2\pi\hbar)^{-\frac 12}
{\nu(\hbar)}
\ee

\vskip 1cm
%
%
\subsection{
The case of
any dimension $d$}\label{endproofd}\ 
Let us perform the change of variable 
$$
z_l\to((j_l+\tfrac12)\hbar)^\frac12z_l,\ l=1,\dots,d,
$$
 and decompose $a=\sum_{}a^{\pm\bullet_1,\dots,\pm\bullet_d}$, $\bullet_l\in\{\leq,=,\geq\},\ l=1,\dots,d$, where
$$
a^{\pm\bullet_1,\dots,\pm\bullet_d}_{j-j'}\big(\tfrac{\sqrt{(j+\tfrac12))\hbar}z}{\sqrt2}\big):
=\prod_{l=1}^d\chi^{\bullet_l}(z_l)\chi_{\bN}(\pm((j_l-j'_l)^2-j_l)  a_{j-j'}\big(\tfrac{\sqrt{(j+\tfrac12))\hbar}z}{\sqrt2}\big),
$$
where $\chi^\bullet$ is defined in \eqref{chichi}, $\chi_{\bN}$ is the characteristic function on $\bN$ and, by abuse of notation,
$$
\sqrt{(j+\tfrac12))\hbar}z:=(\sqrt{(j_1+\tfrac12))\hbar}z_1,\dots,\sqrt{(j_d+\tfrac12))\hbar}z_d).
$$

$\la j|\mbox{Wig}(a^{\pm\bullet_1,\dots,\pm\bullet_d}_{j-j'})|j'\ra$ will involve $d$ integrations in the variable $z_1,\dots,z_d$. 

\begin{enumerate}
\item For each $l$ such that the sign of $\bullet_l$ is positive, we will have $(j_l-j'_l)^2\geq j_l$ so that the argument of the case (1) in the proceeding section  apply (after a inverse change of variable $z_l\to((j_l+\tfrac12)\hbar)^{-\frac12}z_l$.
\item For each of the other variables $z_k$, the integral will be reduced to the domain of $\chi^{\bullet_k}$ and can be treated as in the cases (a), (b), (c) of the preceding section.
\item Combining in each dimension the different points (1), (2) (a), (2) (b) and (2) (c) of the proof  of Proposition \ref{propmatesti} in Section \ref{matesti}, we realize easily that
\eqref{moinsplus2}
becomes in dimension $d$
\begin{eqnarray}\label{moinsplusd}
&&|\ket{j'}\mbox{Wig}(a_{j-j'}\bra{j}|\leq (D_{\mu,n}\hbar)^d2^{(M+K)d}\nonumber\\
&&\times
\|a_{j-j'}\|_{M,1+\epsilon+\tau,n}||j'-j|+\tfrac12|^{-\min{(1+\alpha+\epsilon,1+\epsilon+\tau)}}((j+\tfrac12)\hbar+\tfrac12)^{-\min{(\frac{\bvc-\alpha-1}2,\mu-\frac32,M,\frac {n-1}2)}}.\nonumber
\end{eqnarray}
Therefore, $\sqrt{R^{in}}$ satisfies the hypothesis \eqref{iaj} in Theorem \ref{main2} as soon as
\be\label{condad}
\|W_\hbar[\sqrt{R^{in}}]
\|_{
{\color{black}\frac34+\epsilon,\frac72+3\epsilon,3}
}\leq (2\pi\hbar)^{-\frac d2}\frac{\mu'(\hbar)}{D_{\frac94+\epsilon,3}}
=(2\pi\hbar)^{-\frac d2}
{\mu(\hbar)}
\ee
By the same argument,  hypothesis \eqref{ioaj} in Theorem \ref{main2} is satisfied as soon as
\be\label{condnabla}
\hbar^{\frac12}\| W_\hbar[\sqrt{R^{in}}]
\|
_{
\frac34
+\epsilon,
\frac52
+3\epsilon,
4
}
\leq (2\pi\hbar)^{-\frac d2}
\frac{\nu'(\hbar)}{2^{\frac52
+3\epsilon}D_{\frac94+\epsilon,4}}=
(2\pi\hbar)^{-\frac d2}
{\nu(\hbar)}
\ee
\end{enumerate}
}
\subsection{End of the proof of item \textit{\eqref{m1}} in Theorem \ref{main} }\ 

We just proved that hypothesis $(\ref{tightwig})  $ and $(\ref{semicwig})$ in Theorem \ref{main1} imply conditions$(\ref{ajjprimeo})$ and  $(\ref{ajjprime})  $  in Theorem  \ref{main2} for 
$$
\mu'(\hbar)=D_{\frac94+\epsilon,3}\mu(\hbar),\ 
\nu'(\hbar)=D_{\frac94+\epsilon,4}\nu(\hbar).
$$
Therefore, by item $(\ref{m2})$, item $(\ref{m1})$ holds true for
\be\label{dm1}
D_{\ref{m1}}=D_{\frac94+\epsilon,4}D_{\ref{m2}}
\ee 
since $D_{\frac94+\epsilon,4}>D_{\frac94+\epsilon,3}>1$.

\section{The case of $N$ particles}\label{npart}
{\bcr As it was already mentioned, a peculiar feature of Theorem \ref{armaarma} compared to usual semiclassical/microlocal methods is the fact that it provides an upper bound of a quantity at time $t$ linear in the same quantity at time $t=0$, without any extra remainder term to estimate. We will use this fact in the case of an $N$-body problem with factorized initial data through the use of the following result.}
\begin{Lem}\label{ptilem}
Let $R_1,\dots,R_N$ be $N$ density matrices on $L^2(\bR^d)$ and let $f_1,\dots,f_N$  be $N$ probability densities on $\bR^{2d}$. Let us denote 
\begin{eqnarray}
R&=&R_1\otimes\dots\otimes R_N\nonumber\\
f(X,\Xi)&=&f_1(x_1,\xi_1)\dots f(x_N,\xi_n),\ X=(x_1,\dots,x_N), \chi=(\xi_1\dots\xi_n).
\end{eqnarray}

Then
$$
E_\hbar(f,R)^2\leq E_\hbar(f_1,R_1)^2+\dots+E_\hbar(f_N,R_N)^2.
$$
\end{Lem}
\begin{proof}
Let 
$$
\Pi(X,\Xi)=\pi^{op}_1(x_1,\xi_1)\otimes\dots\otimes\pi^{op}_N(x_N,\xi_N),
$$
where $\pi^{op}_i$ is an optimal coupling of $f_i$ and $R_i$. Then it is easy to see that $\Pi$ is a coupling of $f$ and $R$ and therefore, writing $c_N(X,\Xi)=\sum\limits_{i=1}^Nc_i(x_i,\xi_i)$ where $c_i(x_i,\xi_i)=\frac12(|x_i-y_i|^2+|\xi_i+i\hbar\nabla_{y_i}|^2)$ on $L^2(\bR^{Nd},dy_1\dots dy_N)$,
\begin{eqnarray}
E_\hbar(f,R)^2&\leq& \int \Tr c_N(X,\Xi)\Pi(X,\Xi)dXd\Xi\nonumber\\
&=&
\sum_{i=1}^N\int\Tr{(c_i(x_i,\xi_i)\pi^{op}_i(x_i,\xi_i))}dx_id\xi_i\leq \sum_{i=1}^NE_\hbar(f_i,R_i)^2.\nonumber
\end{eqnarray}
\end{proof}
{\bcr Lemma \ref{ptilem} suggests to multiply $E_\hbar(f,R)^2$ by $\frac1N$ in the situation involving $N$ particles. Likewise we will divide  $\MKd^2$ by $N$  as  is customary, and the square of the distance  $\delta$ as well.

We will consider, as in \cite{FGPaul}, Section 2.3, the quantum evolution of interacting $N$ particles through the following von Neumann equation.
\be\label{vneN}
i\hbar\frac{d}{dt} R_N(t)=
[-\frac12\hbar^2\Delta_N+V_N,R_N(t)],\ R_N(0)=R^{in}_N
\ee
with the factorized initial condition 
\be\label{inicondN}
R^{in}_N=(R^{in})^{\otimes N}
\ee
where $R^{in}$ 
 is a density matrix on $L^2(\bR^d)$.
 
 In \eqref{vneN}, $\Delta_N$ is the Laplacien on $L^2(\bR^{Nd})$ and 
 \be\label{defpotN}
 V_N(x_1,\dots,x_N)=\frac1N\sum_{i<j}V(x_1-x_j),\ \ \  V\in C^{1,1}(\bR^d),\ V\ even,\ V(0)=0.
 \ee
 Under these conditions on $V$, the $N$-body dynamics is well defined (see \cite{FGPaul}, Section 2.3).

 \vskip 1cm 
 We denote by $\rho_N(t)$ the solution of the corresponding Liouville equation on $\bR^{2d}$ with initial condition $\widetilde W[R^{in}_N]=(W[R^{in}])^{\otimes N}$:
 \be\label{liuovN}
 \dot\rho_N=\{\tfrac12\sum_{i=1}^N\xi_i^2+\tfrac1N\sum_{i<j}V(x_i-x_j),\rho_N\},\ \ \ \rho(0)=\widetilde W[R^{in}],
 \ee
 and $\Phi_N^t$ the Hamiltonian flow of Hamiltonian $\tfrac12\sum\limits_{i=1}^Np_i^2+\tfrac1N\sum\limits_{i<j}V(q_i-q_j)$ so that
 \be\label{phitN}
 \rho_N(t)=\rho(0)\circ\Phi_N^{-t}.
 \ee
 Moreover we define again
 $$
 \lambda=\frac{1+\max{(4Lip(\nabla V)^2,1)}}2.
 $$

\begin{Thm}\label{mainN} Let either $W_\hbar[\sqrt{R^{in}}]$ satisfy the hypothesis of Theorem \ref{main1}, or $\sqrt{R^{in}}$ satisfy the hypothesis  of Theorem \ref{main2} or Theorem \ref{main3}. Let us suppose for simplicity that $\mu(\hbar)=\mu'(\hbar)=1,\ \nu(\hbar)=\nu'(\hbar)=\hbar$ (the general case is straightforward to state). Then 
$$
\tfrac1N\MKd(\widetilde W_\hbar[R_N(t)],\widetilde W_\hbar[R^{in}_N]\circ\Phi_N^{-t})^2\leq Ce^{
\lambda
|t|}\hbar,
$$
where $C$ is independent of $N$.

Moreover, if $\|W_\hbar[R^{in}]\|_{L^2(\bR^{2d})}\leq C'<1$ , uniformly in $\hbar$, that is to say  $\|R^{in}\|_2\leq C'(2\pi\hbar)^d$, then
$$
\frac1{\sqrt N}
\delta(W_\hbar[R(t)],W_\hbar[R^{in}]^{\otimes N}\circ\Phi_N^{-t})\leq
 C''
\sqrt\hbar(e^{
\lambda
|t|}+2e^{1+\Lip{(\nabla V)})|t|)}\tfrac{C^N}{\sqrt N})
$$ 
where $C''$ is independent of $N$.
\end{Thm}
\begin{proof}
By Theorem 2.7 in \cite{FGPaul} in the case $n=N$ we know that, by Lemma \ref{ptilem},
$$
\tfrac1N E(\widetilde W_\hbar[R^{in}_N\circ\Phi_N(t)], R_N(t))^2\leq e^{\lambda|t|}
\tfrac1NE(\widetilde W_\hbar[R^{in}_N], R_N^{in})^2
\leq e^{\lambda|t|} E(\widetilde W_\hbar[R^{in}], R^{in})^2.
$$
Theorem \ref{armaprop}, (3) and Theorem \ref{main}, $I-III$, give then the first inequality of Theorem \ref{mainN}.

Theorem \ref{ineqphi} gives directly the second one.
\end{proof}
\vskip 1cm
}
\begin{appendix}

\section{Proof of the triangle inequality}\label{profftriang}
In this section, we prove Theorem \ref{armatriang}.

The proof makes use of some inequalities between the (classical and/or quantum) transportation cost operators. We begin with an elementary, but useful lemma, which can be viewed as the Peter-Paul
inequality for operators.

\begin{Lem}\lb{L-PP}
Let $T,S$ be unbounded self-adjoint operators on $\fH=L^2(\bR^n)$, with domains $\Dom(T)$ and $\Dom(S)$ respectively such that $\Dom(T)\cap\Dom(S)$ is dense in $\fH$. Then, for all $\a>0$, one has
$$
\la v|TS+ST|v\ra\le\a\la v|T^2|v\ra+\frac1\a\la v|S^2|v\ra\,,\qquad\text{ for all }v\in\Dom(T)\cap\Dom(S)\,.
$$
\end{Lem}

\begin{proof}

Indeed, for each $\a>0$ and each $v\in\Dom(T)\cap\Dom(S)$, one has
$$
\ba
\a\la v|T^2|v\ra+\tfrac1\a\la v|S^2|v\ra-\la v|TS+ST|v\ra
\\
=|\sqrt{\a} Tv|^2+|\tfrac1{\sqrt{\a}}Sv|^2-\la\sqrt{\a}Tv|\tfrac1{\sqrt{\a}}Sv\ra-\la\tfrac1{\sqrt{\a}}Sv|\sqrt{\a}Tv\ra
\\
=\left|\sqrt{\a} Tv-\tfrac1{\sqrt{\a}}Sv\right|^2\ge 0&\,.
\ea
$$
\end{proof}

Let us redefine the cost $c_\hbar(x,\xi)$ that way
\be\label{defcpseudo}
c(x,\xi;z,\hb D_z):=|x-z|^2+|\xi-\hbar D_z|^2.
\ee
\begin{Lem}\lb{L-IneqCost}
For each $x,\xi,y,\eta,z\in\bR^d$ and each $\a>0$, one has
$$
\ba
c(x,\xi;z,\hb D_z)\le(1+\a)(|x-y|^2+|\xi-\eta|^2)+(1+\tfrac1\a)c(y,\eta;z,\hb D_z)\,,
\\
|x-z|^2+|\xi-\zeta|^2\le(1+\a)c(x,\xi;y,\hb D_y)+(1+\a)c(z,\zeta;y,\hb D_y)\,.
\ea
$$
\end{Lem}

These two inequalities are of the form $A\le B$ where $A$ and $B$ are unbounded self-adjoint operators on $L^2(\bR^n)$ for some $n\ge 1$, with 
$$
\cW:=\{\psi\in H^1(\bR^n)\text{ s.t. }|x|\psi\in\cH\}\subset\Dom_f(A)\cap\Dom_f(B)\,,
$$
denoting by $\Dom_f(A)$ (resp. $\Dom_f(B)$) the form-domain of $A$ (resp. of $B$) --- see \S VIII.6 in \cite{RS1} on pp. 276--277. The inequality $A\le B$ means that the bilinear form associated to $B-A$ is nonnegative, i.e. that
$$
\la w|A|w\ra\le\la w|B|w\ra\,,\qquad\text{ for all }w\in\cW\,.
$$

\begin{proof}
These two inequalities are proved in the same way. Let us prove for instance the first inequality:
$$
\ba
c(x,\xi;z,\hb D_z)=&|x-y+y-z|^2+|\xi-\eta+\eta-\hb D_z|^2
\\
=&|x-y|^2+|\xi-\eta|^2+c(y,\eta;z,\hb D_z)
\\
&+2(x-y)\cdot(y-z)+2(\xi-\eta)\cdot(\eta-\hb D_z)\,.
\ea
$$
By Lemma \ref{L-PP} with $T=(x-y)\mathbb I_z$ and $S=\eta-\hb D_z$ we get
$$
\ba
2(x-y)\cdot(y-z)+2(\xi-\eta)\cdot(\eta-\hb D_z)
\\
\le\a (|x-y|^2+|\xi-\eta|^2)+\frac1\a c(y,\eta;z,\hb D_z)&\,,
\ea
$$
which concludes the proof of the first inequality. 
\end{proof}
Finally we will need the following proposition whose proof is postponed to the end of this section.
\begin{Prop}\label{C-Energ}
Let $A=A^*\ge 0$ be an unbounded self-adjoint operator on $\cH$ with domain $\Dom(A)$, and let $E$ be its spectral decomposition.
Let $T\in\cL^1(\fH)$ satisfy $T=T^*\ge 0$, and let $(e_j)_{j\ge 1}$ be a complete orthonormal system of eigenvectors of $T$ with $Te_j=\tau_je_j$ and $\tau_j\in[0,+\infty)$ for each $j\ge 1$.

Assume that
\be\lb{TrTA}
\sum_{j\ge 1}\tau_j\int_0^\infty \l\la e_j|E(d\l)|e_j\ra<\infty\,.
\ee
Let $\Phi_n:\,\bR_+\to\bR_+$ be a sequence of continuous, bounded and nondecreasing functions such that
$$
0\le\Phi_1(r)\le\Phi_2(r)\le\ldots\le\Phi_n(r)\to r\quad\text{ as }n\to\infty\,.
$$
Set
$$
\Phi_n(A):=\int_0^\infty\Phi_n(\l)E(d\l)\in\cL(\cH)\,.
$$
Then $\Phi_n(A)=\Phi_n(A)^*\ge 0$ for each $n\ge 1$ and
the sequence $T^{1/2}\Phi_n(A)T^{1/2}$ converges weakly to $T^{1/2}AT^{1/2}$ as $n\to\infty$. Moreover
$$
\Tr_\cH(T\Phi_n(A))\to\Tr_\cH(T^{1/2}AT^{1/2})\qquad\text{ as }n\to\infty\,.
$$
\end{Prop}
\vskip 1cm
\begin{proof}[Proof of Theorem \ref{armatriang}]\ 

We start by the following  ``disintegration" result.
\begin{Lem}\lb{L-Disint}
Let $f\in\cP^{ac}(\bR^d\times\bR^d)$, let $R\in\cD(\fH)$ and let $Q\in\cC(f,R)$. There exists a $\si(\cL^1(\fH),\cL(\fH))$ weakly measurable function $(x,\xi)\mapsto Q_f(x,\xi)$ defined a.e. on $\bR^d\times\bR^d$ with values in $\cL^1(\fH)$ 
such that
$$
Q_f(x,\xi)=Q^*_f(x,\xi)\ge 0\,,\quad\Tr(Q_f(x,\xi))=1\,,\quad\hbox{ and }Q(x,\xi)=f(x,\xi)Q_f(x,\xi)
$$
for a.e. $(x,\xi)\in\bR^d\times\bR^d$.
\end{Lem}

\begin{proof}[Proof of the Lemma]\ 

Let $f_1$ be a Borel measurable function defined on $\bR^d\times\bR^d$ and such that $f(x,\xi)=f_1(x,\xi)$ for a.e. $(x,\xi)\in\bR^d\times\bR^d$. Let $\cN$ be the Borel measurable set defined as follows: 
$\cN:=\{(x,\xi)\in\bR^d\times\bR^d\hbox{ s.t. }f(x,\xi)=0\}$, and let $u\in\fH$ satisfy $|u|=1$.
Consider the function
$$
(x,\xi)\mapsto Q_f(x,\xi):=\frac{Q(x,\xi)+\indc_\cN(x,\xi)|u\ra\la u|}{f_1(x,\xi)+\indc_\cN(x,\xi)}\in\cL(\fH)
$$
defined a.e. on $\bR^d\times\bR^d$. The function $f_1+\indc_{\cN}>0$ is Borel measurable on $\bR^d\times\bR^d$ while $(x,\xi)\mapsto\la\phi|Q(x,\xi)|\psi\ra$ is measurable and defined a.e. on $\bR^d\times\bR^d$ for each $\phi,\psi\in\fH$.
Set $\cA:\,\cL(\fH)\times(0,+\infty)\ni(T,\l)\mapsto\l^{-1}T\in\cL(\fH)$; since $\cA$ is continuous, the function $Q_f:=\cA(Q+\indc_\cN\otimes|u\ra\la u|,f_1+\indc_\cN)$ is weakly measurable on $\bR^d\times\bR^d$. Since $f_1+\indc_\cN>0$, 
and since $Q(x,\xi)=Q^*(x,\xi)\ge 0$, one has $(Q(x,\xi)+\indc_\cN\otimes|u\ra\la u|)^*=Q(x,\xi)+\indc_\cN\otimes|u\ra\la u|\ge 0$ for a.e. $(x,\xi)\in\bR^d\times\bR^d$. On the other hand, for a.e. $(x,\xi)\in\bR^d\times\bR^d$, one has
$\Tr(Q(x,\xi)+\indc_\cN\otimes|u\ra\la u|)=f(x,\xi)+\indc_\cN(x,\xi)$, so that $\Tr(Q_f(x,\xi))=1$. Finally
$$
f(x,\xi)Q_f(x,\xi)=\frac{f(x,\xi)Q(x,\xi)}{f_1(x,\xi)+\indc_\cN(x,\xi)}=Q(x,\xi)\quad\hbox{ for a.e. }(x,\xi)\in\bR^d\times\bR^d\,,
$$
since $f=f_1$ a.e. on $\bR^d\times\bR^d$ and $\indc_\cN(x,\xi)=0$ for a.e. $(x,\xi)\in\bR^d\times\bR^d$ such that $f(x,\xi)>0$. Since $Q_f$ satisfies $\Tr(Q_f(x,\xi))=1$ for a.e. $(x,\xi)\in\bR^d\times\bR^d$ and is weakly measurable on 
$\bR^d\times\bR^d$, it is $\si(\cL^1(\fH),\cL(\fH))$ weakly measurable.
\end{proof}
By Theorem 2.12 in chapter 2 of \cite{VillaniAMS}, there exists an optimal coupling for $W_2(f,g)$, of the form $f(x,\xi)\de_{\grad\Phi(x,\xi)}(dyd\eta)$, where $\Phi$ is a convex function on $\bR^d\times\bR^d$. Let $Q\in\cC(g,R_1)$ and set 
$$
P(x,\xi;dyd\eta):=f(x,\xi)\de_{\grad\Phi(x,\xi)}(dyd\eta)Q_g(y,\eta)\,,
$$
where $Q_g$ is the disintegration of $Q$ with respect to $f$ obtained in Lemma \ref{L-Disint}. Then $P$ is a nonnegative, self-adjoint operator-valued measure satisfying
$$
\Tr_\fH(P(x,\xi;dyd\eta))=f(x,\xi)\de_{\grad\Phi(x,\xi)}(dyd\eta),
$$
while
$$
\int Pdxd\xi=(\grad\Phi\#f)(y,\eta)dyd\eta Q_g(y,\eta)=g(y,\eta)Q_g(y,\eta)dyd\eta=Q(y,\eta)dyd\eta\,.
$$
In particular
\be\lb{PCfR1}
\int P(x,\xi;dyd\eta)=f(x,\xi)Q_g(\grad\Phi(x,\xi))\in\cC(f,R_1)\,.
\ee
Therefore
$$
\ba
E_\hb(f,R_1)^2\le\int\Tr_\fH(Q_g(\grad\Phi(x,\xi))^{1/2}c_\hb(x,\xi)Q_g(\grad\Phi(x,\xi))^{1/2})f(x,\xi)dxd\xi\,.
\ea
$$
By the first inequality in Lemma \ref{L-IneqCost}, one has
$$
c(x,\xi;z,\hb D_z)\le(1+\a)|(x,\xi)-\grad\phi(x,\xi)|^2+(1+\tfrac1\a)c(\grad\Phi(x,\xi);z,\hb D_z)
$$
for a.e. $(x,\xi)\in\bR^d\times\bR^d$ and all $\a>0$. Since $g\in\cP_2^{ac}(\bR^d\times\bR^d)$ and $R_1\in\cD_2(\fH)$ and $Q\in\cC(g,R_1)$, then
$$
\ba
\int\Tr_\fH(Q(y,\eta)^{1/2}c(y,\eta)Q(y,\eta)^{1/2})dyd\eta&
\\
=\int\Tr_\fH(Q_g(\grad\Phi(x,\xi))^{1/2}c(\grad\Phi(x,\xi))Q_g(\grad\Phi(x,\xi))^{1/2})f(x,\xi)dxd\xi&<\infty\,.
\ea
$$
For each $\eps>0$, set
$$
c^\eps(x,\xi;z,\hb D_z)=(I+\eps c(x,\xi;z,\hb D_z))^{-1}c(x,\xi;z,\hb D_z)\le c(x,\xi;z,\hb D_z)\,.
$$
Then, for a.e. $(x,\xi)\in\bR^d\times\bR^d$ and each $\eps>0$, one has
$$
Q_g(\grad\Phi(x,\xi))^{1/2}c(\grad\Phi(x,\xi);z,\hb D_z)Q_g(\grad\Phi(x,\xi))^{1/2}\in\cL^1(\fH)\,,
$$
and
$$
\ba
Q_g(\grad\Phi(x,\xi))^{1/2}&c^\eps_\hb(x,\xi;z,\hb D_z)Q_g(\grad\Phi(x,\xi))^{1/2}
\\
\le&(1+\a)|(x,\xi)-\grad\phi(x,\xi)|^2Q_g(\grad\Phi(x,\xi))
\\
&+(1+\tfrac1\a)Q_g(\grad\Phi(x,\xi))^{1/2}c(\grad\Phi(x,\xi);z,\hb D_z)Q_g(\grad\Phi(x,\xi))^{1/2}\,.
\ea
$$
Integrating both sides of this inequality with respect to the probability distribution $f(x,\xi)$, one finds
$$
\ba
\int\Tr_\fH(Q_g(\grad\Phi(x,\xi))^{1/2}c^\eps_\hb(x,\xi)Q_g(\grad\Phi(x,\xi))^{1/2})f(x,\xi)dxd\xi
\\
\le(1+\a)\int|(x,\xi)-\grad\phi(x,\xi)|^2f(x,\xi)dxd\xi
\\
+(1+\tfrac1\a)\int\Tr_\fH(Q_g(\grad\Phi(x,\xi))^{1/2}c(\grad\Phi(x,\xi))Q_g(\grad\Phi(x,\xi))^{1/2})f(x,\xi)dxd\xi
\\
\le(1+\a)\MKd(f,g)^2
\\
+(1+\tfrac1\a)\int\Tr_\fH(Q_g(y,\eta)^{1/2}c(y,\eta)Q_g(y,\eta)^{1/2})g(y,\eta)dyd\eta
\\
\le(1+\a)\MKd(f,g)^2+(1+\tfrac1\a)\int\Tr_\fH(Q(y,\eta)^{1/2}c(y,\eta)Q(y,\eta)^{1/2})dyd\eta&\,.
\ea
$$
Minimizing the last right hand side of this inequality in $Q\in\cC(g,R_1)$ shows that
$$
\ba
\int\Tr_\fH(Q_g(\grad\Phi(x,\xi))^{1/2}c^\eps_\hb(x,\xi)Q_g(\grad\Phi(x,\xi))^{1/2})f(x,\xi)dxd\xi
\\
\le(1+\a)\MKd(f,g)^2+(1+\tfrac1\a)E(g,R_1)^2&\,.
\ea
$$
Passing to the limit as $\eps\to 0^+$ in the left hand side and applying Proposition \ref{C-Energ} shows that
$$
\ba
E_\hb(f,R_1)^2\le&\int\Tr_\fH(Q_g(\grad\Phi(x,\xi))^{1/2}c(x,\xi)Q_g(\grad\Phi(x,\xi))^{1/2})f(x,\xi)dxd\xi
\\
\le&(1+\a)\MKd(f,g)^2+(1+\tfrac1\a)E(g,R_1)^2\,,
\ea
$$
the first inequality being a consequence of the definition of $E_\hb$ according to \eqref{PCfR1}.

Finally, minimizing the right hand side of this inequality as $\a>0$, i.e. choosing $\a=E_\hb(f,g)/\MKd(f,g)$ if $f\not=g$ a.e. on $\bR^d\times\bR^d$, or letting $\a\to+\infty$ if $f=g$, we arrive at the inequality
$$
\ba
E_\hb(f,R_1)^2\le&\MKd(f,g)^2+E_\hb(g,R_1)^2+2E_\hb(g,R_1)\MKd(f,g)
\\
=&(\MKd(f,g)+E_\hb(g,R_1))^2\,,
\ea
$$
which is precisely the desired inequality. Theorem \ref{armatriang} is proven.
\end{proof}
\begin{proof}[Proof of Proposition \ref{C-Energ}]

Since $E$ is a resolution of the identity on $[0,+\infty)$, and since $\Phi_n$ is continuous, bounded and with values in $[0,+\infty)$, the operators $\Phi_n(A)$ satisfy
$$
0\le\Phi_n(A)=\Phi_n(A)^*\le\left(\sup_{z\ge 0}\Phi_n(z)\right)I_\cH
$$
and
$$
0\le\Phi_1(A)\le\Phi_2(A)\le\ldots\le\Phi_n(A)\le\ldots
$$
Set $R_n:=T^{1/2}\Phi_n(A)T^{1/2}$; by definition $0\le R_n=R_n^*\in\cL^1(\fH)$ and one has
$$
0\le R_1\le R_2\le\ldots\le R_n\le\ldots
$$
together with
$$
\Tr_\cH(R_n)=\sum_{j\ge 1}\tau_j\int_0^\infty\Phi_n(\l)\la e_j|E(d\l)|e_j\ra\le\sum_{j\ge 1}\tau_j\int_0^\infty\l\la e_j|E(d\l)|e_j\ra<\infty
$$
by \eqref{TrTA}. 

Therefore, since $0\leq R_n\in\mathcal L^1(\cH)$,
$$
\sup_{n\ge 1}\la x|R_n|x\ra
\le\sup_{n\ge 1}\||x\ra\la x|\|_{\cL(\cH)}\Tr_{\cH}(R_n)
=\|x\|_\cH^2\sup_{n\ge 1}\Tr_{\cH}(R_n)
<\infty\text{ for all }x\in\cH\,.
$$
Since the sequence $\la x|R_n|x\ra\in[0,+\infty)$ is nondecreasing for each $x\in\cH$, 
$$
\la x|R_n|x\ra\to\sup_{n\ge 1}\la x|R_n|x\ra=:q(x)\in[0,+\infty)\quad\text{ for all }x\in\cH
$$
as $n\to\infty$. Hence
$$
\la x|R_n|y\ra=\la y|R_n|x\ra\to\tfrac14(q(x+y)-q(x-y)+iq(x-iy)-iq(x+iy))=:b(x,y)\in\bC
$$
as $n\to+\infty$. By construction, $b$ is a nonnegative sesquilinear form on $\cH$. 

Consider, for each $k\ge 0$,
$$
F_k:=\{x\in\cH\text{ s.t. }\la x|R_n|x\ra\le k\text{ for each }n\ge 1\}\,.
$$
The set $F_k$ is closed for each $k\ge 0$, being the intersection of the closed sets defined by the inequality $\la x|R_n|x\ra\le k$ as $n\ge 1$. Since the sequence $\la x|R_n|x\ra$ is bounded for each $x\in\cH$, 
$$
\bigcup_{k\ge 0}F_k=\cH\,.
$$
Applying Baire's theorem shows that there exists $N\ge 0$ such that $\mathring{F}_N\not=\varnothing$. In other words, there exists $r>0$ and $x_0\in\cH$ such that
$$
|x-x_0|\le r\implies|\la x|R_n|x\ra|\le N\text{ for all }n\ge 1\,.
$$
By linearity and positivity of $R_n$, this implies
$$
|\la z|R_n|z\ra|\le \tfrac2r(M+N)\|z\|^2\text{ for all }n\ge 1\,,\quad\text{ with }M:=\sup_{n\ge 1}\la x_0|R_n|x_0\ra\,.
$$
In particular
$$
\sup_{|z|\le 1}q(z)\le\tfrac2r(M+N)\,,\quad\text{ so that }|b(x,y)|\le\frac2r(M+N)|\|x\|_\cH\|y\|_\cH
$$
for each $x,y\in\cH$ by the Cauchy-Schwarz inequality. By the Riesz representation theorem, there exists $R\in\cL(\cH)$ such that
$$
R=R^*\ge 0\,,\quad\text{ and }\quad b(x,y)=\la x|R|y\ra
=\lim_{n\to\infty}\la x|R_n|y\ra\,,
$$
so that  $R_n\to R\in\cL(\cH)$ weakly as $n\to\infty$.

Observe now that $R\ge R_n$ for each $n\ge 1$, so that 
\be\label{ineqinv}
\sup_{n\ge 1}\Tr_\cH(R_n)\le\Tr_\cH(R)\,.
\ee
In particular
$$
\sup_{n\ge 1}\Tr_\cH(R_n)=+\infty\implies\Tr_\cH(R)=+\infty\,.
$$
Since the sequence $\Tr_\cH(R_n)$ is nondecreasing,
$$
\Tr_\cH(R_n)\to\sup_{n\ge 1}\Tr_\cH(R_n)\quad\text{ as }n\to\infty\,.
$$
By the noncommutative variant of Fatou's lemma (Theorem 2.7 (d) in \cite{Simon}), 
$$
\sup_{n\ge 1}\Tr_\cH(R_n)<\infty\implies R\in\cL^1(\cH)\text{ and }\Tr_\cH(R)\le\sup_{n\ge 1}\Tr_\cH(R_n)\,,
$$
so that, by \eqref{ineqinv},
$$
\Tr_\cH(R)=\sup_{n\ge 1}\Tr_\cH(R_n).
$$

Finally
$$
T^{1/2}AT^{1/2}-R_n=\sum_{j,k\ge 1}\tau_j^{1/2}\tau_k^{1/2}\left(\int_0^\infty(\l-\Phi_n(\l))\la e_j|E(d\l)|e_k\ra\right)|e_j\ra\la e_k|
$$
so that
$$
\ba
\la x|T^{1/2}AT^{1/2}-R_n|x\ra=&\int_0^\infty(\l-\Phi_n(\l))\La\sum_{j\ge 1}\tau_j^{1/2}\la e_j|x\ra e_j|E(d\l)|\sum_{k\ge 1}\tau_k^{1/2}\la e_k|x\ra e_k\Ra
\\
=&\int_0^\infty(\l-\Phi_n(\l))\la T^{1/2}x|E(d\l)|T^{1/2}x\ra\ge 0\,.
\ea
$$
Hence
$$
0\le T^{1/2}AT^{1/2}-R_n=(T^{1/2}AT^{1/2}-R_n)^*\in\cL^1(\cH)
$$
so that
$$
\ba
\|T^{1/2}AT^{1/2}-R_n\|_1=&\Tr_\cH(T^{1/2}AT^{1/2}-R_n)
\\
=&\sum_{j\ge 1}\tau_j\int_0^\infty(\l-\Phi_n(\l))\la e_j|E(d\l)|e_j\ra\to 0
\ea
$$
as $n\to\infty$ by monotone convergence. Hence $R_n\to T^{1/2}AT^{1/2}$ in $\cL^1(\cH)$ and one has in particular
$$
\Tr_{\cH}(T\Phi_n(A))=\Tr_{\cH}(T^{1/2}\Phi_n(A)T^{1/2})\to\Tr_{\cH}(T^{1/2}AT^{1/2})\,.
$$
\end{proof}

\section{Quantum Wasserstein and weak topologies}\label{wassweak}

{\bcr  It is well known that Wasserstein metrics dominate weak topologies. It is natural to wonder whether the $E_\hbar$ does the same. The answer will be positive when considering, for two density  matrices  $S$ and $R$, the quantity $E_\hbar(\widetilde W_\hbar[S],R)$.  Since, by Theorem \ref{armaprop},3),
 $$
 E_\hbar(\widetilde W_\hbar[S],R)\geq \tfrac1{\sqrt2} \MKd(\widetilde W_\hbar[S], \widetilde W_\hbar[R]),
 $$ 
 the following results will give the answer.
 }
 \vskip 0.5cm
 {\bf Notation}: unless the contrary is  specified, we will define the norm of a vector   as the {\bf maximum} value  of the  norm over all the components.
 \subsection{$L^2$ test functions}\label{ltwo}
 We will denote by $\cD_{HS}(\fH)$ the set of Hilbert-Schmidt operators on $\fH$.
 \begin{Def}\label{defddel}
 Let $R,S\in\cD_{HS}(\fH)$, with Wigner transforms $W_{\hbar}[R],W_{\hbar}[S]$. We define
 
 \be\label{defd}
 d(R,S)=
 \sup_{\substack{
\|F\|_{1}, \ 
 \|\frac1{i\hbar}[x,F]\|_{1},\  \|\frac1{i\hbar}[-i\hbar\nabla,F]\|_{1}\leq 1\\
}}
|\Tr{(F(R-S))}|
 \ee
 and
 \be\label{defdel}
 \delta
 (
 W_{\hbar}[R],W_{\hbar}[S]
 )=
 \sup_{\substack{ 
 Lip(f)\leq 1\\
\| f\|_{L^2(\bR^{2d})},\ \|\nabla f\|_{L^2(\bR^{2d})}\leq 1
}}
|\int (W_\hbar[R]-W_\hbar[S])f(x,\xi)dxd\xi|
 \ee
 \end{Def}
 \noindent 
 Note that, for $R,S\in\cD_{HS}(\fH)$, we have  
 $d(R,S)\leq \|R-S\|\leq\|R-S\|_{HS}<\infty$ and $\delta(W_{\hbar}[R],W_{\hbar}[S))\leq(2\pi\hbar)^{-d/2}\|R-~S\|_{HS}<\infty$.
 \begin{Prop}\label{distineq}
 Let $\hbar\leq \frac2\pi$. 
 
 Then the functions $d(\cdot,\cdot),\delta(W_{\hbar}[\cdot],W_{\hbar}[\cdot]):\,\cD_{HS}(\fH)\times\cD_{HS}(\fH)\to[0,+\infty[$ are distances.
 
 Moreover,

 $$
 d(\cdot,\cdot)\leq 2^d\delta(W_{\hbar}[\cdot],W_{\hbar}[\cdot]).
 $$\end{Prop}
\begin{proof}
We first prove the inequality. Let us recall the following elementary facts:
\be\label{putded}
 \Tr{(F(R-S))}={(2\pi\hbar)^{d}}\int W_\hbar[F](W_\hbar[R]-W_\hbar[S])(x,\xi){dxd\xi}
\ee
and $W_\hbar[\frac1{i\hbar}[x,F]]=\nabla_pW_\hbar[F],\ W_\hbar[\frac1{i\hbar}[-i\hbar\nabla,F]]=-\nabla_xW_\hbar[F]$.

\noindent By the part $(b)$ of the proof of Proposition B.5 in \cite{splitting} we know that, for $\|D\|_1<\infty$,
$$
\|W_\hbar[D]\|_\infty\leq\frac{\|D\|_1}{(\pi\hbar)^d},
$$
so that, for any $F\in\mathcal D(\mathfrak H)$, 
$$
\|F\|_{1}, \ 
 \|\tfrac1{i\hbar}[x,F]\|_{1},\  \|\tfrac1{i\hbar}[-i\hbar\nabla,F]\|_{1}\leq 1
 \Longrightarrow
 \|{(\pi\hbar)^d}{W_\hbar[F]}\|_{L^\infty}\mbox{ and }Lip({(\pi\hbar)^d}{W_\hbar[F]})\leq 1.
 $$
 
Finally, a straightforward computation shows that, for $\|D\|_1<\infty$,
\be\label{ldeuxdeuxun}
 (2\pi\hbar)^{d}\|W_\hbar[D]\|^2_{L^2}\leq\|D\|^2_{2},
 \leq \|D\|_{1}^2,
 \ee
 so that, for $\tfrac{\pi\hbar}2\leq 1$,
 $$
 \|D\|_{1}\leq 1\Longrightarrow\|D\|_{2}\leq(\tfrac{\pi\hbar}2)^{-\frac d2}
 \Longrightarrow 
 \|(\pi\hbar)^dW_\hbar[D]\|_{L^2(\bR^{2d})}\leq 1.
 $$
 Therefore, 
 $$
\|F\|_{1},\ 
\|\tfrac1{i\hbar}[x,F]\|_{1},\ \|\tfrac1{i\hbar}[-i\hbar\nabla,F]\|_{1}\leq 1
\Longrightarrow
\|
(\pi\hbar)^dW_\hbar[F]\|_{L^2(\bR^{2d})};\ 
\|\nabla 
(\pi\hbar)^dW_\hbar[F]\|_{L^2(\bR^{2d})}\leq 1,
 $$
 and, by \eqref{putded} and the change of test function $f\to (\pi\hbar)^dW_\hbar[F]$, we get $d(\cdot,\cdot)\leq 2^d\delta(W_{\hbar}[\cdot],W_{\hbar}[\cdot])$.
 
 The symmetry in the argument and the triangle inequality for $d$ and $\delta$ are obvious by construction. Moreover $2^d\delta(W_{\hbar}[\cdot],W_{\hbar}[\cdot])\geq d(\cdot,\cdot)= d_1(\cdot,\cdot)$ where $d_1$ is the distance defined  in \eqref{defdinfty} below. Therefore $d(\cdot,\cdot)$ and $\delta(W_{\hbar}[\cdot],W_{\hbar}[\cdot])$ separate points.
\end{proof} 
 \begin{Prop}\lb{quantwassweak}
Let $R,S\in\cD^2(\fH)$, with Wigner transforms $W_{\hbar}[R],W_{\hbar}[S]$ and Husimi transforms $\widetilde W_{\hbar}[R],\widetilde W_{\hbar}[S]$. 
Then
\begin{eqnarray}
\delta(W_{\hbar}[R],W_{\hbar}[S])&\le&
\MKd(\widetilde W_{\hbar}[R],\widetilde W_{\hbar}[S])+\sqrt\hbar
\|W_\hbar[R]-W_\hbar[S]\|_{L^2(\bR^{2d})}\nonumber\\
&=&
\MKd(\widetilde W_{\hbar}[R],\widetilde W_{\hbar}[S])+\sqrt\hbar\frac{\|R-S\|_2}{(2\pi\hbar)^{d/2}}
\end{eqnarray}
\end{Prop}
\begin{proof}
Using successively formulas (7.1) and formula (7.3) in chapter 7 of \cite{VillaniAMS}, we have
$$
\left|\int f(x,\xi)(\widetilde W_\hb[R]-\widetilde W_\hb[S])(x,\xi)dxd\xi\right|\le\Lip(f)\MKd(\widetilde W_\hb[R],\widetilde W_\hb[S])
$$
Since $\widetilde W_\hbar=e^{\hbar\Delta/4}W_\hbar$ one find
\begin{eqnarray}
\left|\int f(x,\xi)( W_\hb[R]- W_\hb[S])(x,\xi)dxd\xi\right|&-
&\Lip(f)\MKd(\widetilde W_\hb[R],\widetilde W_\hb[S])\nonumber\\
&\leq&\left|\int (e^{\hbar\Delta/4}-1)f(x,\xi)( W_\hb[R]-  W_\hb[S])(x,\xi)dxd\xi\right|\nonumber\\
&\leq&
\|(e^{\hbar\Delta/4}-1)f\|_{L^2(\bR^{2d})}\|  W_\hb[R]-  W_\hb[S]\|_{L^2(\bR^{2d})}\nonumber
\end{eqnarray}
 by Cauchy-Schwarz. 
 
 \noindent By  $(e^{\hbar\Delta/4}-1)^2=
(e^{\hbar\Delta/2}-1)-2(e^{\hbar\Delta/4}-1)$ 
 and $ 1-e^{\hbar\Delta/\lambda}\leq -\hbar\Delta/{\lambda},\lambda>0$, one gets
 \be\label{normg}
 \|(e^{\hbar\Delta/4}-1)f\|_{L^2(\bR^{2d})}\leq \sqrt\hbar \|\nabla f\|_{L^2(\bR^{2d})}
 \ee
 which gives the desired inequality.
\end{proof}
Let $R(t)$ solves the von Neumann equation \eqref{vne} with initial condition $R^{in}$ and  let $\Phi^t$ be the underlying Hamiltonian flow of Hamiltonian $\tfrac12p^2+V(q)$  as defined by \eqref{phit}. Let moreover $\mathcal R(t)$ be defined by
\be\label{defcalr}
W_\hbar[\mathcal R(t)]=
W_\hbar[ R^{in}]\circ \Phi^{-t}
\ee
Note that, since $R^{in}$ is Hilbert-Schmidt and $\Phi^{-t}$ is symplectic, $W_\hbar[\mathcal R(t)]$ is square integrable on $\bR^{2d}$ so that $\mathcal R(t)$ is well defined as a Hilbert-Schmidt operator.

\noindent Moreover, at the contrary of $\widetilde W_\hbar[\cR(t)]$ which is, as $\cR(t)$, not positive, $\widetilde W_\hbar[R^{in}]\circ\Phi^{-t}$ is a probability measure since $R^{in}$ is positive (and $\Phi^{-t}$  symplectic).
\begin{Thm}\label{ineqphi}
Let $V\in C^{1,1}$. 
 Then
\begin{eqnarray}
&&\delta(W_\hbar[R(t)],W_\hbar[\mathcal R(t)])\nonumber\\&\leq& \MKd(\widetilde W_\hbar[R(t)],\widetilde W_\hbar[R^{in}]\circ\Phi^{-t})
+
\sqrt\hbar(1+e^{1+\Lip{(\nabla V)})|t|)})\|W_\hbar[R^{in}]\|_{L^2}\nonumber\\
&=& \MKd(\widetilde W_\hbar[R(t)],\widetilde W_\hbar[R^{in}]\circ\Phi^{-t})
+
\sqrt\hbar(1+e^{(1+\Lip{(\nabla V)}|t|)})\|R^{in}\|_{2}/(2\pi\hbar)^{d/2}\nonumber
\end{eqnarray}
\end{Thm}
\begin{proof}
One starts again with
$$
\left|\iint f(x,\xi)(\widetilde W_\hb[R(t)]-\widetilde W_\hb[R^{in}]\circ\Phi^{-t})(x,\xi)dxd\xi\right|\le\Lip(f)\MKd(\widetilde W_\hb[R],\widetilde W_\hb[S]\circ\Phi^{-t}))
$$
Since $\Phi^t$ preserves the measure on phase space we get
\begin{eqnarray}
\int f(x,\xi)\widetilde W_\hb[R^{in}]\circ\Phi^{-t}(x,\xi)dxd\xi
&=&
\int f\circ\Phi^{t}(x,\xi)\widetilde W_\hb[R^{in}](x,\xi)dxd\xi\nonumber\\
&=&
\int 
\big(
e^{\hbar\Delta/4}
(f\circ\Phi^{t})
\big)
(x,\xi) W_\hb[R^{in}](x,\xi)dxd\xi\nonumber\\
&=&
\int f(x,\xi)\big(W_\hbar[R^{in}]\circ\Phi^{-t}\big)(x,\xi)dxd\xi\nonumber\\
&+&
\int g^t(x,\xi)W_\hbar[R^{in}](x,\xi)dxd\xi\nonumber
\end{eqnarray}
with
\begin{eqnarray}
g^t(x,\xi)&=&
\big(
e^{\hbar\Delta/4}-1\big)
(f\circ\Phi^{t})
(x,\xi)\nonumber
\end{eqnarray}
Therefore
\begin{eqnarray}
\left|\iint f( W_\hb[R(t)]- W_\hb[R^{in}]\circ\Phi^{-t})dxd\xi\right|
&\le&\Lip(f)\MKd(\widetilde W_\hb[R(t)],\widetilde W_\hb[S]\circ\Phi^{-t})\nonumber\\
&+&\int (g^{t=0}W_\hbar[R(t)]
-g^tW_\hbar[R^{in}])(x,\xi)dxd\xi\nonumber\\
&\leq&
\|g^{t=0}\|_{L^2}\|W_\hbar[R(t)]\|_{L^2}+\|g^t\|_{L^2}\|W_\hbar[R^{in}]\|_{L^2}\nonumber\\
&=&(\|g^{t=0}\|_{L^2}+\|g^t\|_{L^2})\|W_\hbar[R^{in}]\|_{L^2}\label{firstsecond}
\end{eqnarray}
Since $g^{t=0}=(e^{\hbar\Delta/4}-1)f$, the first term in the parenthesis in \eqref{firstsecond} can be estimated by \eqref{normg}. The second one will be treated by  the following result.

\begin{Lem}\label{clldeux}
Let $\Phi^t$ the Hamiltonian flow associated to $\tfrac12 p^2+V(q)$ with $V\in C^{1,1}$. Then, for all $t\in\bR$ and all $f\in W^{1,2}$, and all $1\leq p<\infty$,
$$
\|\mbox{d}\Phi^t\|_{L^\infty}\leq e^{(1+\Lip{(\nabla V)}|t|)}
\ \ \mbox{ and }\ \ 
\|\nabla(f\circ\Phi^t)\|_{L^
p
}\leq e^{(1+\Lip{(\nabla V)}|t|)} \|\nabla f\|_{L^
p
}.
$$
\end{Lem}
\begin{proof}
By the Lipschitz condition on $V$ we have, for all $z,z'\in\bR^{2d}$,
\begin{eqnarray}
|\Phi^t(z)-\Phi^s(z')|
&=&
||z-z'|+\int_0^t\partial_s(\Phi^t(z)-\Phi^s(z'))ds\nonumber\\
&\leq&
||z-z'|+(1+\Lip{(\nabla V)})\int_0^t|\Phi^s(z)-\Phi^s(z')|ds,\nonumber
\end{eqnarray}
so that, by the Gronwall Lemma,
$\Phi^t$ is Lipschitz for all $t$ with Lipschtz constant smaller than $e^{(1+\Lip{(\nabla V)})|t|)}$.
Therefore, $\mbox{d}\Phi^t$ exists a.e., 
$\|\mbox{d}\Phi^t
\|_{L^\infty}\leq e^{(1+\Lip{(\nabla V)}|t|)}$. Moreover  $|\nabla(f\circ\Phi^t)|\leq e^{(1+\Lip{(\nabla V)}|t|)}|(\nabla f)\circ\Phi^t|$ a.e. and
we can perform the change of variable $z\to \Phi^t(z)$ in $\|(\nabla f)\circ\Phi^t|\|_{L^2}$. Since, being a symplectomorphism, $\Phi^t$ preserves the Lebesgue measure on phase space, we get the second inequality in Lemma \ref{clldeux}.
\end{proof}
The end of the proof of the first inequality in Theorem \ref{ineqphi} is achieved by \eqref{normg} and Lemma \ref{clldeux}. The second one follows from the first one by \eqref{ldeuxdeuxun}.
\end{proof}

 \subsection{$L^\infty$ test functions}\label{linfty}\ 
 
 
\begin{Def}\label{defddelM}
 Let $R,S\in\cD(\fH)\cup\{\cR\in\cD_{HS}(\fH),\ W_\hbar[\cR]\in L^1(\bR^{2d})\}$, with Wigner transforms $W_{\hbar}[R],W_{\hbar}[S]$. We define, for any integer $M>0$,
 \be\label{defdeltainfty}
\delta_M(W_{\hbar}[R],W_{\hbar}[S])
:=\sup_{\substack{
\max\limits_{|\a|,|\b|\le M}\|\d^\a_x\d^\b_\xi f\|_{L^\infty}\leq 1}}|\int (W_{\hbar}[R]
-W_{\hbar}[S])
(x,\xi)
f(x,\xi)dxd\xi|\,.
\ee
and 
\be\label{defdinfty}
d_M(R,S):=\sup_{\substack{
\max\limits_{\substack{|\a|,|\b| \le M}}\|\mathcal D^\a_{-i\hbar\nabla}\mathcal D^\b_{x}F\|_{1}\leq 1}}|\Tr{(F(R-S))}|\,,
\ee
where $\cD_A=\tfrac1{i\hbar}[A,\cdot]$ for each (possibly unbounded) self-adjoint operator $A$ on $\fH$.
\end{Def}
Note that when $R\in\cD(\fH)$, $|\Tr{FR}|\leq\|F\|_\infty\|R\|_1\leq 1$ and when $R\in\cD_{HS}(\fH)$, $|\Tr {FR}|\leq (F,R)_{HS}\leq \|R\|_{HS}\|F\|_{HS}\leq\|R\|_{HS}\|F\|_{1}\leq\|R\|_{HS}<\infty$, so that $d<~\infty$.

Moreover, when $M\geq[d/2]+2$, $\max\limits_{|\a|,|\b|\le M}\|\d^\a_x\d^\b_\xi f\|_{L^\infty}\leq 1$ implies by the Calderon-Vaillancourt theorem that $f$ is the Wigner function of a bounded operator $F$ such that $\|F\|\leq\gamma_d(2\pi\hbar)^d$. Therefore, when $R\in\cD(\fH)$ $\int W_\hbar[R]f(x,\xi)dxd\xi=(2\pi\hbar)^{-d}\Tr{RF}\leq \gamma_d<\infty$, and when $W_\hbar[R]\in L^1(\bR^{2d})$, $\int\int W_\hbar[R]f(x,\xi)dxd\xi\leq\|W_\hbar[R]\|_1<\infty$, so that $\delta_M<\infty,\ M\geq[d/2]+2$.

By the same proof than the one of  Lemma B2 in \cite{splitting} one sees that $d_M$ is a distance, and by the same argument as for $\delta$ in the proof of Proposition \ref{distineq}, together with the first inequality in Proposition \ref{propnest} below, one sees that $\delta_M$, considered as a functions of Wigner functions is also a distance.

In order to use the estimate of $\gamma_d$ proved in the next section, we will restrict ourselves to  the following 
restriction (because $2[d/4]+2<[d/2]+2$) of Proposition B5 in \cite{splitting}.

\begin{Prop}\label{propnest}
Let $R,S\in\cD_{HS}(\fH)$. Then
$$
d_{2[d/4]+3}(R,S)\le 2^d
\delta_{2[d/4]+3}(W_{\hbar}[R],W_{\hbar}[S]).
$$
Moreover, if $R,S$ are density operators, then
$$\delta_{2[d/4]+3}(W_{\hbar}[R],W_{\hbar}[S])\le\MKd(\widetilde W_{\hbar}[R],\widetilde W_{\hbar}[S])+\frac{2d\gamma_d}{\sqrt\pi}\sqrt{\hbar}
$$
where $\g_d$ is the constant that appears in Theorem \ref{cvgamma} below.
\end{Prop}
Let us remark that $\delta_M$ is decreasing with respect to $M$, so the second inequality is valid for any $\delta_M, M\geq 2[d/4]+3$.
\begin{proof}
The second inequality (for $M=[d/2]+2$ but the proof is insensitive to $M$) is contained in Proposition B5 in \cite{splitting}, together with the first one in the case of density operators. It is easy to check that the proof of the latter extends verbatim to the Hilbert-Schmidt case, as it use only the fact that $W_\hbar[R-S]$ is square integrable.
\end{proof}
The following result is the analogue of Theorem \ref{ineqphi} in Section \ref{linfty} above.

Let $R(t)$ solves the von Neumann equation \eqref{vne} with initial condition $R^{in}$ and  let $\Phi^t$ be the underlying Hamiltonian flow of Hamiltonian $\tfrac12p^2+V(q)$  as defined by \eqref{phit}. 
\begin{Thm}\label{ineqphinf}
Let $V\in C^{1,1}$ and $W_\hbar[R^{in}]\in L^1(\bR^{2d})$. 
 Then
\begin{eqnarray}
\delta_{2[d/4]+3}(W_{\hbar}[R(t)],W_{\hbar}[R^{in}]\circ\Phi^t)&\leq& \MKd(\widetilde W_\hbar[R(t)],\widetilde W_\hbar[R^{in}]\circ\Phi^{-t})\nonumber\\
&+&
\sqrt\hbar\tfrac d\pi\big(\gamma_d+e^{1+\Lip{(\nabla V)})|t|)}\|W_\hbar[R^{in}]\|_{L^1(\bR^{2d})}\big)
\nonumber
\end{eqnarray}
\end{Thm}

\begin{proof}

One again one starts  with
$$
\left|\iint f(x,\xi)(\widetilde W_\hb[R(t)]-\widetilde W_\hb[R^{in}]\circ\Phi^{-t})(x,\xi)dxd\xi\right|\le\Lip(f)\MKd(\widetilde W_\hb[R],\widetilde W_\hb[S]\circ\Phi^{-t}))
$$
The term involving $\int f\widetilde W[R(t)]$ is treated along the same lines as for the beginning of the proof of theorem \ref{quantwassweak} and we get that
\be\label{phi1}
\big|\int f\widetilde W[R(t)]dxd\xi-\int f W[R(t)] \big|dxd\xi\leq\sqrt\hbar\tfrac{d\gamma_d}{\sqrt{2\pi}}.
\ee

Indeed, by the definition \eqref{defweyl}, \eqref{intww} and Theorem \ref{cvgamma} in Appendix \ref{cvconst} below,
\begin{eqnarray}
\left|\int f(\widetilde W[R]-W[R])dxd\xi\right|
&=&
\left|\int(e^{\hbar\Delta/4}-1)f
W[R]dxd\xi\right|\nonumber\\
&=&|\Tr(\Op^W_\hbar(
(e^{\hbar\Delta/4}-1)f)R)|\nonumber\\
&\le&
\|\Op^W_\hbar((e^{\hbar\Delta/4}-1)f)\||\Tr R|\nonumber\\
&\le&
\gamma_d
\max_{|\alpha|,|\beta|\le 2[d/4]+2}
|D^\alpha_xD^\beta_\xi 
(e^{\hbar\Delta/4}-1)f|\nonumber\\
&\leq&
{\sqrt\hbar}\frac{d\gamma_d}\pi
\max_{|\alpha|,|\beta|\le 2[d/4]+3}
|D^\alpha_xD^\beta_\xi f|\nonumber
\end{eqnarray}

For the second term we write, denoting $z=(x,\xi),z'+(x',\xi')$,
\begin{eqnarray}
\int f\widetilde W[R^{in}]\circ\Phi^{-t}dxd\xi
&=& (\pi\hbar)^{-d}
\int f(z)e^{-(\Phi^{-t}(z)-z')^2/\hbar} W[R^{in}](z')dzdz'\nonumber\\
&=&
(\pi\hbar)^{-d}\int f(\Phi^t(z))e^{-(z-z')^2/\hbar} W[R^{in}](z')dzdz'\nonumber\\
&=&
\int \left(e^{\hbar\Delta/4}(f\circ\Phi^t)\right)(z')W[R^{in}](z')dz'\nonumber\\
&=&
\int \left(f\circ\Phi^t-g_t\right)(z')W[R^{in}](z')dz'\nonumber\\
&=&\int f(W[R^{in}]\circ\Phi^{-t})dxd\xi-\int g_tW[R^{in}]dxd\xi\label{plustard}
\end{eqnarray}
with 
\begin{eqnarray}
g_t(x,\xi)&=&f\circ\Phi^t(x,\xi)-e^{\hbar\Delta/4}(f\circ\Phi^t)(x,\xi)\nonumber\\
&=&
\int \big(f\circ\Phi^t(x+\sqrt\hbar q,\xi+\sqrt\hbar p)
-f\circ\Phi^t(x,\xi)\big)e^{-|q|^2-|p|^2}dqdp/\pi^d\nonumber
. 
\end{eqnarray}
By Lemma \ref{clldeux},
$$|g_t(x,\xi)|
\leq\sqrt\hbar\|d\phi^t\|_{L^\infty}\tfrac d\pi\Lip{(f)}.
\leq
\sqrt\hbar e^{1+\Lip{(\nabla V)})|t|)}\frac d\pi.
$$
and therefore
\begin{eqnarray}\label{phi2}
\big|\int f\widetilde W[R^{in}]\circ\Phi^t
-
\int fW[R^{in}]\circ\Phi^t\big|dxd\xi
&\leq&|-\int g_tW[R^{in}]|dxd\xi\nonumber\\
&\leq &
\sqrt\hbar\tfrac d\pi e^{1+\Lip{(\nabla V)})|t|)}\|W_\hbar[R^{in}]\|_{L^1(\bR^{2d})}.
\end{eqnarray}
We conclude by adding \eqref{phi1} and \eqref{phi2}.
\end{proof}

Let us define finally
\be\label{defdeltainfty1}
\ba
\delta_{M,1}(W_{\hbar}[R],W_{\hbar}[S])
:=\sup_{\substack{f\in L^2(\bR^{2d})\\ \max\limits_{|\a|,|\b|\le M}\|\d^\a_x\d^\b_\xi f\|_{L^\infty}\leq 1\\
\|\nabla f\|_{L^1(\bR^{2d})}\leq 1}}\left|\int (W_{\hbar}[R](x,\xi)-W_{\hbar}[S](x,\xi))f(x,\xi)dxd\xi\right|&\,.
\ea
\ee

\begin{Thm}\label{ineqphinf1}
Let $V\in C^{1,1}$ and $W_\hbar[R^{in}]\in L^\infty(\bR^{2d})$. 
 Then
\begin{eqnarray}
\delta_{2[d/4]+3,1}(W_{\hbar}[R(t)],W_{\hbar}[R^{in}]\circ\Phi^t)&\leq& \MKd(\widetilde W_\hbar[R(t)],\widetilde W_\hbar[R^{in}]\circ\Phi^t)\nonumber\\
&+&
\sqrt\hbar\big(\tfrac{d\gamma_d}{\pi}+e^{1+\Lip{(\nabla V)})|t|)}\|W_\hbar[R^{in}]\|_{L^\infty(\bR^{2d})}\big)
\nonumber
\end{eqnarray}
\end{Thm}
\begin{proof}
The proof is quasi identical to the one of Theorem \ref{ineqphinf}. the only change consists in estimating the last integral in \eqref{plustard} by
$$
|\int g_tW[R^{in}]dxd\xi|
\leq\|W[R^{in}]\|_{L^\infty(\bR^{2d})}\|g_t\|_{L^1(\bR^{2d})}.
$$
We estimate $\|g_t\|_{L^1(\bR^{2d})}$ by
\begin{eqnarray}
\int |g_t(x,\xi)|dxd\xi&=&\int|f\circ\Phi^t(x,\xi)-e^{\hbar\Delta/4}(f\circ\Phi^t)(x,\xi)|dxd\xi\nonumber\\
&=&
\int |\big(f\circ\Phi^t(x+\sqrt\hbar q,\xi+\sqrt\hbar p)
-f\circ\Phi^t(x,\xi)\big)e^{-|q|^2-|p|^2}|dqdpdxd\xi/\pi^d\nonumber\\
&\leq&
\|\nabla(f\circ\Phi^t)\|_{L^1(\bR^{2d})}\leq ||d\Phi^t||_{L^\infty(\bR^{2d})}\|\nabla f\|_{L^1(\bR^{2d})}.
 \nonumber
\end{eqnarray}
And we conclude  by Lemma \ref{clldeux} in the $L^1$ version,
\end{proof}

\section{An estimation of the Calderon-Vaillancourt constant}\label{cvconst}
In this section we revisit the proof of the Caderon-Vaillancourt Theorem by keeping track of the constants. We follow the proof by Hwang \cite{Hwang}, as presented in \cite{lerner}.

Let us remind first the definition of the Weyl quantization (see \cite{Fo} for extensive details).

\noindent To $a\in\cS(\bR^{2d})$ we associate the operator $\Op_\hbar^W(a)$ on $L^2(\bR^d)$ defined by its integral kernel given by
\be\label{defweyl}
\Op_\hbar^W(a)(x,y):=\int_{\bR^d}a(\tfrac{x+y}2,\xi)e^{i\frac{(x-y)\xi}\hbar}\frac{d\xi}{(2\pi\hbar)^d}.
\ee
A link with Wigner functions can be expressed by the (easily checkable) following identity, valid e.g. for $\Op_\hbar^W(a)$ bounded and $R$ trace class
\be\label{intww}
\Tr{(\Op_\hbar^W(a)E)}=\int_{\bR^{2d}}a(x,\xi)W_\hbar[R](x,\xi)dxd\xi.
\ee
Finally, defining $a_\hbar(x,\xi):=a(x,\hbar\xi)$ we immediatly get that
\be\label{aahbar}
\Op_\hbar^W(a)=(Op_1^W(a_\hbar):=a_\hbar(x,D).
\ee

We first define, for $k\in\bN,\ k>[d/2],\ k\ even,\ \mbox{that is }k=2[d/4]+2,\ x\in\bR^d$,
$$
P_k(x)=(1+|x|^2)^{k/2}
$$ 
and, for $u\in L^2$,
$$
W_{u}(x, \xi)=\int u(y)P_k(x-y)^{-1}e^{-iy\xi}dy
$$
Obviously
\be\label{fo2}
\|W_u\|_{L^2}
\leq \|u\|_{L^2}\|P_k^{-1}\|_{L^2}:=C_k\|u\|_{L^2}.
\ee
\begin{Lem}\label{f2}
$$
C_k^2=\int_{\bR^d}\frac1{(1+x^2)^{k}}dx\leq
2\frac{vol(S_{d-1})}{2k-d}=\frac{2\pi^{\frac d2}}{\Gamma{(d/2+1)}(2k-d)}
$$
\end{Lem}
Note that $C_k<\infty$ only for $k>d/2$.

\begin{proof}
$$
\int_0^1\frac{\rho^{d-1}}{(1+\rho^2)^k}d\rho\leq 
\int_0^1{\rho^{d-1}}d\rho=d^{-1},\  
\int_1^\infty\frac{\rho^{d-1}}{(1+\rho^2)^k}d\rho\leq
\int_1^\infty\rho^{d-1-2k}d\rho=(2k-d)^{-1}
$$
\end{proof}
Moreover, for all $\alpha,\ |\alpha|\leq k$, and calling $D:=-i\nabla$,
\be\label{f01}
\|D^\alpha P_k^{-1}\|_{L^\infty}\leq C_{\alpha,k}.
\ee

\begin{Lem}\label{f1}
$$
C_{\alpha,k}\leq 
\sum_{m=0}^{|\alpha|-1}k^{|\alpha|-m}(3|\alpha|)^m|\alpha|^{d}
\leq
|\alpha|^dk^{|\alpha|}(3|\alpha|/k)^{|\alpha|}=
|\alpha|^d(3|\alpha|)^{|\alpha|}.
$$
\begin{proof}
One first remark that, for $|\beta|=1$,
$$
D^\beta P_k(x)^{-1}=k\frac{x^\beta}{1+x^2}P_k(x)^{-1}
$$
Therefore the highest term in $k$ in $D^\alpha P_k^{-1}$ will be 
$$
k^{|\alpha|}
\sum_{|\alpha'|=|\alpha|}
\frac{x^{\alpha'}}{(1+x^2)^{|\alpha'|}}P_k(x)^{-1}\leq k^{|\alpha|}|\alpha|^d
$$
The preceding term will be
$$
k^{|\alpha|-1}\sum_{|\alpha'|=|\alpha|-1}
\big(D^{\alpha-\alpha'}\frac{x^{\alpha'}}{(1+x^2)^{|\alpha'|}}\big)P_k(x)^{-1}
\leq k^{|\alpha|-1}|\alpha|^d3|\alpha|
$$
since, for $|\beta|=1$,
\begin{eqnarray}
|D^\beta\frac{x^{\alpha'}}{(1+x^2)^{|\alpha'|}} |&=&|\frac{D^\beta x^{\alpha'}}{(1+x^2)^{|\alpha'|}} - 
2|\alpha'|\frac{x^{\alpha'}x^\beta}{(1+x^2)^{|\alpha'|+1}}|\nonumber\\
&=&|\beta\cdot\alpha'\frac{ x^{\alpha'-\beta}}{(1+x^2)^{|\alpha'|}} - 
2|\alpha'|\frac{x^{\alpha'}x^\beta}{(1+x^2)^{|\alpha'|+1}}|
\leq 3|\alpha|,\nonumber
\end{eqnarray}
 due to, 
$$
\frac{|x_1^{\alpha_1}\dots x_m^{\alpha_m}|}{(1+x^2)^M}\leq
\frac{|x_1|^{\alpha_1}}{(1+x^2)^{\alpha_i}}\dots\frac{ |x_m|^{\alpha_m}}{(1+x^2)^{\alpha_m}}
\leq 1
,\ |\alpha_1+\dots+\alpha_m|\leq M.
$$
The next term will be
$$
k^{|\alpha|-2}
\sum_{|\alpha''|=|\alpha|-2}
\big(D^{\alpha-\alpha''} \frac{x^{{\alpha''}}}{(1+x^2)^{|{\alpha''}|}}\big)P_k(x)^{-1}
\leq k^{|\alpha|-2}(|\alpha|-2)^d
d^23|\alpha|^2
\leq
k^{|\alpha|-2}
3^2
|\alpha|^2
|\alpha|^{d}
$$
since, decomposing $\alpha-\alpha''=\beta+\beta',\ |\beta|=|\beta'|=1$,
\begin{eqnarray}
|D^{\beta'}D^\beta\frac{x^{{\alpha''}}}{(1+x^2)^{|{\alpha''}|}} |
&=&|\beta\cdot\alpha''D^{\beta'}\frac{ x^{{\alpha''}-\beta}}{(1+x^2)^{|{\alpha''}|}} - 
2|{\alpha''}|D^{\beta'}\frac{x^{{\alpha''}}x^\beta}{(1+x^2)^{|{\alpha''}|+1}}|\nonumber\\
&=&
|(\beta\cdot\alpha'')(\beta'\cdot(\alpha''-\beta))\frac{x^{{\alpha''}-\beta-\beta'}}{(1+x^2)^{|{\alpha''}|}}-2|{\alpha''}|\beta\cdot\alpha''
\frac{x^{{\alpha''}-\beta+\beta'}}{(1+x^2)^{|{\alpha''}|+1}}\nonumber\\
&&
-2|{\alpha''}|
\beta'\cdot(\alpha''+\beta)\frac{x^{{\alpha''}+\beta-\beta'}}{(1+x^2)^{|{\alpha''}|+1}}
+4|{\alpha''}|(|{\alpha''}|+1)
\frac{x^{{\alpha''}+\beta+\beta'}}{(1+x^2)^{|{\alpha''}|+2}}
|\nonumber\\
&\leq& 9|{\alpha''}|^2\leq 9|\alpha|^2=3^2|\alpha|^2\nonumber
\end{eqnarray}
Obviously, for $|\beta_l|=1,l=1,\dots,m$ and $|\alpha'|\leq|\alpha|$,
$$
|D^{\beta_m}D^{\beta_{m-1}}D^{\beta_1}\frac{x^{\alpha'}}{(1+x^2)^{|\alpha'|}} |
\leq 3^m|\alpha|^m
$$
and the $m$th preceding term will be estimated by
$$
k^{|\alpha|-m}(3|\alpha|)^m|
\alpha|^{d}.
$$
\end{proof}
\end{Lem}

We turn now to the proof of the Calderon-Vaillancourt Theorem  in the framework of Weyl quantization. In order to lighten the fiormulas, we will perform the computation in the homogeneous (non semiclassical) case $\hbar=1$ and get back to the semiclassical situation thanks to \eqref{aahbar}.

Denoting now by $\hat v=(2\pi)^{-d}\int_{\bR^d}v(y)e^{i\xi y}dy$ the (renormalized and non unitary) Fourier transform of $u$, we first note that
$$
e^{i\xi x}v(x)=P_k(D_x)\left(W_{\hat v}(\xi,x)e^{ix\xi}\right).
$$
We get, for $a(x,\xi)$ $k$ times differentiable,
\begin{eqnarray}
(2\pi)^d(\bar u,a(x,D)v)&=&\int a(\tfrac{x+y}2,\xi)e^{i\xi(x-y)}u(y)v(x)dxdyd\xi\nonumber\\
&=&
\int a(\tfrac{x+y}2,\xi)P_k^{-1}(x-y)P_k(D_\xi)e^{i\xi(x-y)}u(y)v(x)dxdyd\xi\nonumber\\
&=&
\int \big(\int u(y)P_k^{-1}(x-y)e^{-i\xi y}P_k(D_\xi)a(\tfrac{x+y}2,\xi)dy\big)e^{i\xi x}v(x)dxd\xi\nonumber\\
&=&
\int W_{uP_k(D_\xi)a(\tfrac{x+.}2.\xi)}(x,\xi)e^{ix\xi}v(x)dxd\xi\nonumber\\
&=&
\int W_{uP_k(D_\xi)a(\tfrac{x+.}2.\xi)}(x,\xi)
P_k(D_x)\left(W_{\hat v}(\xi,x)e^{ix\xi}\right)dxd\xi\nonumber\\
&=&
\int P_k(D_x)\left(W_{uP_k(D_\xi)a(\tfrac{x+.}2.\xi)}(x,\xi)\right)
W_{\hat v}(\xi,x)e^{ix\xi}dxd\xi\nonumber
\end{eqnarray}
since, by \eqref{f01}, $P_k^{-1}$ is $k$ times differentiable (as well as, let us recall, $a(x,\xi)$).

Therefore
\be\label{ftrde}
(2\pi)^d|(\bar u,a(x,D)v)|
\leq
\|P_k(D_x)W_{uP_k(D_\xi)a(\tfrac{x+.}2.\xi)}(x,\xi)\|_{L^2(\bR^{2d})}
\|W_{\hat v}(\xi,x)\|_{L^2(\bR^{2d})}.
\ee
We know that, by \eqref{fo2},
\be\label{fo3}
\|W_{\hat v}\|_{L^2}\leq 
C_k\|v\|_{L^2}.
\ee
Moreover, for $l\in\bN$,
$$
(D_x^2)^{l}\left(W_{uP_k(D_\xi)a(\tfrac{x+.}2.\xi)}(x,\xi)\right)
=
\sum_{|\alpha+\beta|\leq 2l}c^l_{\alpha,\beta}
\int u(y)P_k(D_\xi)D_x^\alpha a(\tfrac{x+y}2.\xi)D_x^\beta P_k(x-y)^{-1}e^{-iy\xi}dy
$$
Indeed
$$
(D_x^2)^l=\sum_{l_1+\dots+l_d=l}
\binom{l}{(l_1,\dots,l_d)}\prod_{i=1}^dD_{x_i}^{2l_i}
.
$$
where the multinomial coefficient $\binom{l}{(l_1,\dots,l_d)}$ is defined by
$$
\big(\sum_{i=1}^dx_i\big)^l=\sum_{l_1+\dots l_d=l}\binom{l}{(l_1,\dots,l_d)}
\prod_{i=1}^mx_i^{l_i}.
$$
In particular
\be\label{formulti}
\sum_{l_1+\dots l_d=l}\binom{l}{(l_1,\dots,l_d)}
=
d^l.
\ee
We have
\begin{eqnarray}
&&D_{x_i}^{2l_i}\left(W_{uP_k(D_\xi)a(\tfrac{x+.}2.\xi)}(x,\xi)\right
)\nonumber\\
&=&
D_{x_i}^{2l_i}\int 
u(y)P_k(D_\xi)(a(\tfrac{x+y}2,\xi))
P_k(x-y)^{-1}e^{-iy\xi}
dy\nonumber\\
&=&
\int
\sum_{m_i=1}^{2l_i}\binom{2l_i}{m_1}D_{x_i}^{2l_i-m_i}(P_k(D_\xi)a(\tfrac{x+y}2,\xi))D_{x_i}^{m_i}(P_k^{-1}(x-y))
u(y)e^{-i\xi y}dy\nonumber
\end{eqnarray}
so that
\begin{eqnarray}
&&(D_x^2)^l\left(W_{uP_k(D_\xi)a(\tfrac{x+.}2.\xi)}(x,\xi)\right)\nonumber\\
&=&\int\sum_{l_1+\dots+l_d=l}
\binom{l}{{(l_1,\dots,l_d)}}\prod_{i=1}^d
\sum_{m_i=1}^{2l_i}\binom{2l_i}{m_1}D_{x_i}^{2l_i-m_i}P_k(D_\xi)a(\tfrac{x+y}2,\xi)D_{x_i}^{m_i}P_k^{-1}(x-y)u(y)e^{-i\xi y}dy\nonumber\\
&=&
\int\sum_{l_1+\dots+l_d=l}\binom{l}{{(l_1,\dots,l_d)}}\prod_{i=1}^d\sum_{\alpha_i+\beta_i=2l_i}
D_{x_i}^{\alpha_i}P_k(D_\xi)a(\tfrac{x+y}2,\xi)
D_{x_i}^{\beta_i}P_k^{-1}(x-y)u(y)e^{-i\xi y}dy
\nonumber\\
&=&
\int\sum_{\substack{|\alpha+\beta|=2l\\ \alpha_i+\beta_i\ even}}{\scriptsize \binom{l}{(\tfrac{\alpha_1+\beta_i}2,\dots,\tfrac{\alpha_d+\beta_d}2)}\binom{\alpha+\beta}{\beta}}
\prod_{i=1}^d
D_{x_i}^{\alpha_i}P_k(D_\xi)a(\tfrac{x+y}2,\xi)
D_{x_i}^{\beta_i}P_k^{-1}(x-y)
u(y)e^{-i\xi y}dy\nonumber\\
&=&
\int\sum_{\substack{|\alpha+\beta|=2l\\ \alpha_i+\beta_i\ even}}
c^l_{\alpha,\beta}
\prod_{i=1}^d
D_{x_i}^{\alpha_i}P_k(D_\xi)a(\tfrac{x+y}2,\xi)
D_{x_i}^{\beta_i} P_k^{-1}(x-y)u(y)e^{-i\xi y}dy\nonumber
\end{eqnarray}
with
$$
c^l_{\alpha,\beta}
=
\binom{l}{(\tfrac{\alpha_1+\beta_1}2,\dots,\tfrac{\alpha_d+\beta_d}2)}\binom{\alpha+\beta}{\beta}
$$
\begin{Lem}\label{a3}
$$
c^l_{\alpha,\beta}
\leq  2^{2l}\binom{l}{(\tfrac{\alpha_1+\beta_i}2,\dots,\tfrac{\alpha_d+\beta_d}2)}.
$$
\end{Lem}
\begin{proof}
Just observe that
$\binom{\alpha+\beta}{\beta}\leq 2^{|\alpha+\beta|}\leq 2^{2l}$.
\end{proof}
Therefore, developing $P_k(D_x)$ and $P_k(D_\xi)$,
\begin{eqnarray}
&&P_k(D_x)\left(W_{uP_k(D_\xi)a(\tfrac{x+.}2.\xi)}(x,\xi)\right)\nonumber\\
&=&
\sum_{l,l'=1}^{k/2}\binom{k}{l}\binom{k}{l'}
\sum_{\substack{|\alpha+\beta|=2l\\ \alpha_i+\beta_i\ even}}
c^l_{\alpha,\beta}
\int u(y)(D_\xi^2)^{l'}D_x^\alpha a(\tfrac{x+y}2.\xi)D_x^\beta P_k(x-y)^{-1}e^{-iy\xi}dy,\nonumber
\end{eqnarray}
and, using \eqref{fo2},
\begin{eqnarray}
&&\|
P_k(D_x)\left(W_{uP_k(D_\xi)a(\tfrac{x+.}2.\xi)}(x,\xi)\right)
\|_{L^2}\label{ftrde3}\\
&\leq&
\sum_{l,l'=1}^{k/2}\binom{k}{l}\binom{k}{l'}
\sum_{\substack{|\alpha+\beta|=2l\\ \alpha_i+\beta_i\ even}}
c^l_{\alpha,\beta}
\|u (D_\xi^2)^{l'}D_x^\alpha a(\tfrac{x+\cdot}2.\xi)\|_{L^2}
\|D_x^\beta P_k^{-1}\|_{L^2}
\nonumber\\
&\leq&
\|u\|_{L^2}
\sum_{l,l'=1}^{k/2}\binom{k}{l}\binom{k}{l'}
\sum_{\substack{|\alpha+\beta|=2l\\ \alpha_i+\beta_i\ even}}
C_{\beta,k}
\max\limits_{\substack{l'\leq k\\|\alpha|\leq k}}\|(D_\xi^2)^{l'}D_x^\alpha a\|_{L^\infty}\nonumber\\
&:=&C^k\max\limits_{\substack{l'\leq k/2\\|\alpha|\leq k}}\|(D_\xi^2)^{l'}D_x^\alpha a\|_{L^\infty}\|u\|_{L^2}.
\nonumber
\end{eqnarray}
with, by Lemma \ref{a3} and Lemma \ref{f1}
\begin{eqnarray}
C^k&=&\sum_{l,l'=1}^{k/2}\binom{k}{l}\binom{k}{l'}
\sum_{\substack{|\alpha+\beta|=2l\\ \alpha_i+\beta_i\ even}}
c^l_{\alpha,\beta}C_{\beta,k}\nonumber\\
&\leq& 2^{k}\sum_{l,l'=1}^{k/2}
\sum_{\substack{|\alpha+\beta|=2l\\ \alpha_i+\beta_i\ even}}
c^l_{\alpha,\beta}
|\beta|^d(3|\beta|)^{|\beta|}
\nonumber\\
&\leq &
\frac{k^2}42^{k}
k^d(3k)^{k}
\sum_{\substack{|\alpha+\beta|=2l\\ \alpha_i+\beta_i\ even}}
c^l_{\alpha,\beta}
\nonumber\\
&\leq &
\frac{k^2}42^{k}
k^d(3k)^{k}
2^{2k}
\sum_{\substack{k_1+\dots+k_d=k}}\binom{k}{(k_1,\dots,k_d)}\nonumber\\
&\leq &
\frac{k^2}42^{k}
k^d(3k)^{k}
2^{2k}d^k=\frac{k^22^{3k}}4(3k)^kk^dd^k\nonumber
\end{eqnarray}
by \eqref{formulti}. 

Finally, by \eqref{ftrde}, \eqref{fo3} and \eqref{ftrde3},
$$
|(\bar u,a(x,D)v)|
\leq 
D_k\max\limits_{\substack{l'\leq k/2\\|\alpha|\leq k}}\|(D_\xi^2)^{l'}D_x^\alpha a\|_{L^\infty}\|u\|_{L^2}\|v\|_{L^2}.
$$
with, by 
\eqref{fo3},
\be\label{ddk}
D_k^d:=(2\pi)^{-d}C^kC_k= 
(2\pi)^{-d}\frac{\sqrt2\pi^{-d/4}k^22^{3k}}{4\sqrt{\Gamma{(d/2+1)}(2k-d)}}(3k)^kk^dd^k
\ee
for any integer $k$ such that $2k>d$.

Therefore, in particular, 
$$
\|a(x,D)\|\leq D^d_{2[d/4]+2}\max\limits_{\substack{2l'\leq 
2[d/4]+2
\\|\alpha|\leq 2[d/4]+2}}\|(D_\xi^2)^{l'})D_x^\alpha a\|_{L^\infty}
$$ and, e.g., for $d\geq 4$ (so that $2[d/4]+2\leq d$),

$$
D_{2[d/4]+2}\leq
\frac{d^{3/4}(192e^{-\frac14}\pi^{-\frac54})^{ d}}{4e^{\frac14} }(d^d)^
{11/4}.
$$
We just proved the following result
\begin{Thm}[Calderon-Vaillancourt]\label{cvgamma}
$$
\|a(x,D)\|
\leq \gamma_d
\max\limits_{\substack{|\beta|\leq 2[d/4]+2\\|\alpha|\leq 2[d/4]+2}}
\|D_\xi^\beta D_x^\alpha a\|_{L^\infty}.
$$
with
{
$$\boxed{
\gamma_d=
D^d_{2[d/4]+2}
}.$$}

The semiclassical result is the same by 
\eqref{aahbar}: for $\hbar\leq 1$,
$$
\|\Op_\hbar^W(a)\|\leq \gamma_d
\max\limits_{\substack{|\beta|\leq 2[d/4]+2\\|\alpha|\leq 2[d/4]+2}}
\|D_\xi^\beta D_x^\alpha a\|_{L^\infty}.
$$
\end{Thm}

\end{appendix}
\textbf{Acknowledgements.} The work of Thierry Paul was partly supported by LIA LYSM (co-funded by AMU, CNRS, ECM and INdAM). T.P. thanks also the Dipartimento di Matematica, Sapienza Universit\`a di Roma, for its kind hospitality for several stays during the development of this work.

\end{document}